\documentclass[final,3p,11pt]{elsarticle}

\usepackage{amssymb}
\usepackage{amsmath}

\usepackage{amsmath}
\usepackage{amsfonts}
\usepackage{epsfig}
\usepackage{bigcenter}
\usepackage{cancel}
\usepackage{lineno}
\usepackage{subfloat}
\usepackage{subfig}
\usepackage{enumitem}
\usepackage{tikz}
\usepackage{algorithm}
\usepackage{algorithmic}
\usepackage{svg}
\usepackage{comment}

\usepackage{hyperref}
\usepackage{cleveref}
\crefname{equation}{}{}
\crefname{figure}{Fig.}{Figs.}
\crefname{appendix}{}{}
\crefname{table}{Tab.}{Tabs.}
\Crefname{ALC@unique}{Line}{Lines} 

\newcommand{\R}{\mathbb{R}}

\newcommand{\uu}{{\bf u}}

\def\bxi{{\boldsymbol \xi}}

\newcommand{\T}{\mathrm{T}}
\newcommand{\p}{\mathrm{p}}
\newcommand{\s}{\mathrm{s}}

\newcommand{\HO}{\scalebox{0.8}{$\scriptstyle\rm HO$}}
\newcommand{\LO}{\scalebox{0.8}{$\scriptstyle\rm LO$}}
\newcommand{\CFL}{\scalebox{0.8}{${\rm CFL}$}}

\newcommand{\TOL}{\scalebox{0.75}{$\rm TOL$}}
\newcommand{\red}[1]{\textcolor{red}{#1}}

\def\XXint#1#2#3{{\setbox0=\hbox{$#1{#2#3}{\int}$}
     \vcenter{\hbox{$#2#3$}}\kern-.5\wd0}}

\DeclareMathAlphabet\mathbfcal{OMS}{cmsy}{b}{n}
\DeclareMathOperator{\diag}{diag}
\DeclareMathOperator*{\essinf}{ess\,inf}
\DeclareMathOperator*{\esssup}{ess\,sup}

\newcommand{\D}{\mathcal{D}}

\newtheorem{corollary}{Corollary}[section]
\newtheorem{proposition}{Proposition}[section]

\newdefinition{remark}{Remark}[section]
\newdefinition{example}{Example}[section]
\newproof{proof}{Proof}

\begin{document}

\begin{frontmatter}



\title{Invariant domain preserving limiting of time explicit and time implicit discretizations for systems of conservation laws}


\author{Bartolomeo Fanizza\fnref{label1}} 
\ead{bartolomeo.fanizza@onera.fr}

\author{Florent Renac\corref{cor1}\fnref{label2}} 
\ead{florent.renac@onera.fr}
\cortext[cor1]{Corresponding author. Tel.: +33 1 46 73 37 44}

\affiliation[label1]{organization={DAAA, ONERA, Institut Polytechnique de Paris},
            postcode={92320}, 
            city={Châtillon},
            country={France}}

\affiliation[label2]{organization={MONHADE, Equipe Inria-ONERA, DAAA, ONERA, Institut Polytechnique de Paris},
            postcode={92320}, 
            city={Châtillon},
            country={France}}

\begin{abstract}
This work concerns the design and analysis of a limiting technique that allows the preservation of invariant domains for high-order numerical approximations of nonlinear hyperbolic systems of conservation laws. The method can be applied to any conservative discretization method in space as well as to a wide range of explicit and implicit time integration schemes. The method limits the high-order solution around a low-order accurate solution that is known to preserve all the invariant domains. It generalizes the flux-corrected transport limiter [J. P. Boris and D. L. Book, J. Comput. Phys., 11, 1973; S. T. Zalesak, J. Comput. Phys., 31, 1979] to systems of conservation laws and relies on the limitation of antidiffusive fluxes, but defines the limiting coefficients so as to express the limited solution as a convex combination of invariant domain preserving quantities similarly to the convex limiting framework [Guermond et al., Comput. Methods Appl. Mech. Engrg., 347, 2019]. 
We give details on the derivation of this limiting technique and provide some illustration with finite volume or discontinuous Galerkin (DG) space discretizations associated to explicit or implicit Runge-Kutta methods as well as to time DG integrations. The limiter is applied iteratively to refine the limited solution around the high-order one, while preserving the invariant domains, and a heuristic is proposed to accelerate its convergence. Numerical experiments solving one- and two-dimensional problems involving scalar hyperbolic equations and the compressible Euler equations are presented to illustrate the properties of these schemes.
\end{abstract}




\end{frontmatter}



%
%
\section{Introduction}

Let $\Omega\subset\mathbb{R}^d$ be a bounded domain in $1\leq d\leq3$ space dimensions, we are interested in the numerical approximation of the following problem in conservative form:

\begin{subequations}\label{eq:hyp_sys_cons_laws}
\begin{align}
 \partial_t{\bf u} + \nabla\cdot{\bf f}({\bf u}) &= 0, \quad \mbox{in }\Omega\times(0,\infty), \label{eq:hyp_sys_cons_laws-a} \\
 {\bf u}(\cdot,0) &= {\bf u}_{0}(\cdot),\quad\mbox{in }\Omega. \label{eq:hyp_sys_cons_laws-b} 
\end{align}
\end{subequations}

Here ${\bf u}$ from $\Omega\times(0,\infty)$ into $\Omega^a\subset\mathbb{R}^{n_{eq}}$ is the vector of $n_{eq}$ conserved variables with initial value ${\bf u}_{0}$ in $L^\infty(\Omega)^{n_{eq}}\cap\Omega^a$. 
By $\Omega^a$ we denote the convex set of admissible states. The functions ${\bf f}=({\bf f}_1,\dots,{\bf f}_d)$ in ${\cal C}^1(\Omega^a,\mathbb{R}^{n_{eq}\times d})$ are the fluxes.

The solutions to \cref{eq:hyp_sys_cons_laws} may develop discontinuities in finite time and \cref{eq:hyp_sys_cons_laws} has to be understood in the sense of distributions where we look for weak solutions. Weak solutions are not necessarily unique and \cref{eq:hyp_sys_cons_laws} must be supplemented with further admissibility conditions to select the physical solution. We here focus on entropy inequalities of the form

\begin{equation}\label{eq:entropy_inequ}
 \partial_t\eta({\bf u}) + \nabla\cdot{\bf q}({\bf u}) \leq 0,
\end{equation}

\noindent in the sense of distributions, for all entropy -- entropy flux pairs $(\eta,{\bf q})$ with $\eta({\bf u})$ in ${\cal C}^2(\Omega^a,\mathbb{R})$ convex and ${\bf q}({\bf u})$ in ${\cal C}^1(\Omega^a,\mathbb{R}^d)$ satisfying the following compatibility conditions

\begin{equation}\label{eq:entropy_compatibility_cond}
 \boldsymbol{\eta}'({\bf u})^\top{\bf f}_i'({\bf u})={\bf q}_i'({\bf u})^\top, \quad 1\leq i\leq d, \quad \forall {\bf u} \in \Omega^a.
\end{equation}

The solution to \cref{eq:hyp_sys_cons_laws} is also known to preserve some convex invariant domains $\D\subseteq\Omega^a$ \cite{Chueh_elal_IDP_sys_77,serre_inv_reg_87,Hoff_idp_85}, i.e., if ${\bf u}_0$ lies in $\D\subset\Omega^a$ for almost all ${\bf x}$ in $\Omega$, then ${\bf u}$ remains in $\D$ for almost all $t$ and ${\bf x}$. From now on, by  $\D:=\cap_i\D_i$, with $\D_0=\Omega^a$, we denote the intersection of all the invariant domains we consider in this work. The objective of the present limiter is to ensure that the numerical scheme preserves $\D$ at the discrete level as will be detailed in \cref{sec:IDP_limiter}.

\begin{example}[Scalar conservation laws]\label{ex:scalar_eq}
In the scalar case, $\Omega^a\subseteq\R$, assuming a locally Lipschitz continuous flux ${\bf f}(\cdot)$, entropy solutions to \cref{eq:hyp_sys_cons_laws} are known to satisfy a maximum principle: $\D=[m,M]$ is an invariant domain with $m:=\essinf_{{\bf x}\in\Omega}u_0({\bf x})$ and $M:=\esssup_{{\bf x}\in\Omega}u_0({\bf x})$. The family of Kru\v{z}kov's entropies \cite{Krukov1970FIRSTOQ}, $\eta(u)=|u-K|$, ${\bf q}(u)=\text{sgn}(u-K)\left({\bf f}(u)-{\bf f}(K)\right)$, $K\in\mathbb{R}$, plays an important role in proving uniqueness of solutions to \cref{eq:hyp_sys_cons_laws}.
\end{example}

\begin{example}[Compressible Euler equations]\label{ex:Euler_eq}
For the compressible Euler equations for gas dynamics, the conservative variables and fluxes in \cref{eq:hyp_sys_cons_laws-a} are defined by
\begin{equation}
\label{eq:Euler}
\uu = \begin{pmatrix} \rho \\ \rho {\bf v} \\ \rho E \end{pmatrix},
\quad
{\bf f}(\uu) = \begin{pmatrix} \rho {\bf v}^\top \\ \rho {\bf v}{\bf v}^\top + \p {\bf I}_d \\ (\rho E + \p) {\bf v}^\top \end{pmatrix},
\end{equation} 

\noindent where $\rho$, ${\bf v}$, and $E$ denote the density, velocity vector, and specific total energy, respectively. The system is closed by defining the equation of state $\p=\p(\frac{1}{\rho},e)$ with $e=E-\frac{1}{2}{\bf v}\cdot{\bf v}$ the specific internal energy. Assuming that $\partial_\tau\p(\tau,e)<0$, the system is hyperbolic over the set of states $\Omega^a=\{\uu\in\mathbb{R}^{d+2}:\; \rho>0, {\bf v}\in\mathbb{R}^d, e>0\}$. The compressible Euler equations \cref{eq:hyp_sys_cons_laws-a,eq:Euler} possess the natural entropy -- entropy flux pair $\eta({\bf u})=-\rho \s$ and ${\bf q}({\bf u})= -\rho\s{\bf v}$ with $\s=\s(\frac{1}{\rho},e)$ a strictly convex function defined by the second law of thermodynamics $\T d\s=de+\p d\big(\tfrac{1}{\rho}\big)$ with $\T$ the temperature. $\D=\{\uu\in\Omega^a:\;\s(\uu)\geq\s_0\}$, with $\s_0$ in $\mathbb{R}$, is an invariant domain. 
\end{example}

We are here interested in the approximation of \cref{eq:hyp_sys_cons_laws} with a high-order (HO) discretization. These methods accurately capture smooth solutions, but may not satisfy the above properties of invariant domain preservation at the discrete level. In contrast, low-order (LO) numerical schemes such as the Godunov scheme \cite{Hoff_idp_85}, the Lax-Friedrichs and flux vector splitting schemes \cite{Frid_idp_LF_01}, or the guaranteed maximum speed graph viscosity scheme \cite{guermond_popov_GV_16} have been shown to be invariant domain preserving (IDP). The flux-corrected transport (FCT) limiter  \cite{Lohner_etal_FCT_87,BORIS_Book_FCT_73,zalesak1979fully,kuzmin_turek_FCT_02,renac_mpp_dgsem_nlsca_24,MRR_BEDGSEM_23} and the convex limiter \cite{Guermond_etal_IDP_conv_lim_19,Guermond_etal_IDP_conv_lim_18,PAZNER_idg_DGSEM20211} make use of these LO IDP schemes to limit the HO solution around the LO solution so as to preserve the invariant domains, while keeping conservation thanks to the introduction of antidiffusive fluxes. The FCT limiter can be applied to both explicit and implicit time stepping in contrast to convex limiting that is restricted to explicit time stepping. The FCT limiter is, however, designed to impose bounds on the components of the solution, while it is often necessary to impose bounds on nonlinear functions of the solution, such as the internal energy in the compressible Euler equations, or the void and mass fractions in multiphase and multispecies models. For this reason, the FCT limiter is mainly restricted to scalar equations. 
Monolithic convex limiting was then proposed that is based on adding a graph Laplacian to the HO scheme with a coefficient large enough to ensure that the so-called bar states preserve invariant domains \cite{KUZMIN_MCL_20}. The antidiffusive fluxes are then limited in order to reduce the effect of the graph viscosity, while preserving invariant domains and keeping conservation. It was then extended to time implicit integration of the compressible Euler equations in \cite{Moujaes_Kuzmin_MCL_impl_24} by exploiting the homogeneity property of the flux Jacobian. The introduction of bar states in finite volume (FV) and discontinuous Galerkin (DG) schemes was possible in \cite{Guermond_etal_IDP_conv_lim_19} through the use of the Rusanov numerical flux at interfaces. Even though HO approximations may not be affected by the choice of the numerical flux for smooth solutions or large scale oscillations around discontinuities \cite{qui_et-al06,renac15a}, this is no longer the case when small scale flow features are involved and the numerical flux has a strong effect on their resolution \cite{moura_etal_17,chapelier_etal_14}. 

In this work, we introduce a fairly general limiter that may be applied to any hyperbolic conservation laws with any conservative discretization method on general meshes and a large range of explicit and implicit time integration schemes. Our approach also relies on the limitation of antidiffusive fluxes as in the FCT framework, but the limiting coefficients are defined in order to express the limited solution as a convex combination of IDP quantities similarly to convex limiting. We depart from the latter approach in the definition of the limiting coefficients and the convex combinations that do not use bar states and thus allow the use of any consistent and conservative numerical flux in the HO scheme. The coefficients of the convex combination here define the limited solution as closely as possible to local convex combinations of the LO and HO solutions, while maintaining conservation. We embed the limiter in an iterative procedure \cite{Schar_iterative_FCT_96} that refines the limited solution as closely as possible to the HO solution, while preserving invariant domains. This iterative procedure also allows to recover positive HO solutions and we propose a heuristic to accelerate its convergence. We focus on finite volume (FV) and discontinuous Gakerkin (DG) schemes, as well as the DG spectral element method (DGSEM) as illustrative examples, but the present approach may be applied to any conservative discretization. The space discretizations may be associated to explicit and implicit Runge-Kutta methods, explicit and implicit multistep methods, as well as to time DG integrations among others. 

The paper is organized as follows. \Cref{sec:IDP_limiter} describes the present IDP limiter for a general class of conservative discretization schemes in space (\cref{sec:generic_discr_meth,sec:IDP_limiter_details}), while we propose an iterative algorithm (\cref{sec:iterative_limiter,sec:acc_iterative_limiter}) to refine the limited solution. The limiter is then applied to HO explicit and implicit time stepping including Runge-Kutta (RK) methods (\cref{sec:RK_schemes}), multistep methods (\cref{sec:multistep_schemes}) and DG time integration (\cref{sec:time_DG_discretization}). In \cref{sec:FV-DG_IDP_limiters}, we describe the derivation of the antidiffusive fluxes for FV schemes (\cref{sec:FV_schemes}), general DG schemes (\cref{sec:DG_schemes}), as well as for the DGSEM (\cref{sec:DGSEM_IDP_limiter}). The present limiter is assessed by numerical experiments in \cref{sec:num-xp} with FV and DG schemes associated to various explicit and implicit time discretizations for solving one-dimensional (1D) and two-dimensional (2D) problems involving scalar hyperbolic equations and the compressible Euler equations, while concluding remarks about this work are given in \cref{sec:conclusion}.

%
%
\section{Invariant domain limiting}\label{sec:IDP_limiter}

We here introduce the present limiter in the context of a generic scheme in space and restrict to a first-order time discretization without loss of generality. The application to high-order explicit and implicit schemes will be described in \cref{sec:time_discretization}, while examples of space discretization methods will be given in \cref{sec:FV-DG_IDP_limiters}. 

\subsection{Generic discretization method}\label{sec:generic_discr_meth}

The space domain is discretized with a shape-regular mesh $\Omega_h\subset\mathbb{R}^d$ consisting of nonoverlapping and nonempty open elements $\kappa$ and we assume that it forms a partition of $\Omega$. By $p+1\geq1$ we usually refer to the formal order of space accuracy of the scheme and $N_p\geq1$ denotes the number of DOFs per equation and per element $\kappa$ in $\Omega_h$. 

Using a method of lines, we consider conservative space discretization schemes for \cref{eq:hyp_sys_cons_laws-a} in the general following form

\begin{equation}\label{eq:generic-semi-discr-scheme}
  {\bf M}d_t{\bf u}_h + {\bf R}_h({\bf u}_h) = 0,
\end{equation}

\noindent where ${\bf u}_h$ denotes the global numerical solution represented by the general ansatz

\begin{equation}\label{eq:link-uh-DOFs}
  {\bf u}_h({\bf x},t) = \sum_{i=1}^{N_p}{\bf U}_\kappa^i(t)\phi_\kappa^i({\bf x}) \quad \forall {\bf x}\in\kappa\in\Omega_h, t \geq 0,
\end{equation}

\noindent where ${\bf U}_\kappa^i(t)\in\R^{n_{eq}}$ are the DOFs of the discrete problem to be solved. Canonical examples we will use in the numerical experiments of \cref{sec:num-xp} are the finite volume (FV) schemes where $N_p=1$ and $\phi_\kappa^1\equiv1_\kappa$ the indicator function of the cell $\kappa$; or the discontinuous Galerkin (DG) method where the $(\phi_\kappa^{1\leq i\leq N_p})$ are some basis functions spanning a given space of polynomials of degree $p\geq0$ restricted to the element $\kappa$ and $N_p$ is its dimension.

By ${\bf R}_h$ we denote the discretization of the space derivatives in \cref{eq:hyp_sys_cons_laws-a}, while ${\bf M}$ denotes the global mass matrix. To avoid technical complexity, we assume here that ${\bf M}$ is diagonal: ${\bf M}={\bf I}_{n_{eq}}\otimes\diag(M_\kappa^i:\;\kappa\in\Omega_h, 1\leq i\leq N_p)$ with ${\bf I}_{n_{eq}}$ the identity matrix of size $n_{eq}$. However, it is possible to  handle general mass matrices in the same way as similar limiters such as FCT and convex limiting \cite{kuzmin_turek_FCT_02,KUZMIN_FCT_09}.

The time is now discretized into time steps $\Delta t^{(n)}>0$ and we set $t^{(n+1)}=t^{(n)}+\Delta t^{(n)}$ for $n\geq0$. We restrict ourselves to a first-order time discretization in this section for the sake of clarity. Given a global numerical solution ${\bf u}_h^{(n)}\in\D$ at time $t^{(n)}$ that is assumed to preserve all invariant domains of the PDE, here we consider a high-order space discretization scheme in the general following form

\begin{equation}\label{eq:HO-generic-scheme}
 \frac{M_\kappa^i}{\Delta t^{(n)}}\big({\bf U}_{\kappa,\HO}^{i,n+1} - {\bf U}_\kappa^{i,n}\big) + {\bf R}_{\kappa,\HO}^i({\bf u}_{\HO}^{n+\theta}) = 0 \quad \forall \kappa\in\Omega_h, 1\leq i\leq N_p,
\end{equation}

\noindent where ${\bf u}_{\HO}^{n+\theta}$ denotes the HO solution at time $t^{n+\theta}$ with DOFs ${\bf U}_{\kappa,\HO}^{i,n+\theta}$, while $\theta=0$ corresponds to the explicit forward-Euler scheme and $\theta=1$ to the implicit backward-Euler scheme. In the explicit case $\theta=0$, we have ${\bf u}_{\HO}^{n}={\bf u}_h^{n}$. Here, $M_\kappa^i>0$ refers to the $i$th entry of the mass matrix in the element $\kappa$.

Scheme \cref{eq:HO-generic-scheme} is HO accurate in space, but may fail to preserve the invariant domains. Following the FCT \cite{kuzmin_turek_FCT_02,KUZMIN_FCT_09} and convex limiting \cite{Guermond_etal_IDP_conv_lim_18} frameworks, among others, we aim at limiting the HO solution around a LO IDP solution computed from the LO scheme

\begin{equation}\label{eq:LO-generic-scheme}
 \frac{M_\kappa^i}{\Delta t^{(n)}}\big({\bf U}_{\kappa,\LO}^{i,n+1} - {\bf U}_\kappa^{i,n}\big) + {\bf R}_{\kappa,\LO}^i({\bf u}_{\LO}^{n+\theta}) = 0 \quad \forall \kappa\in\Omega_h, 1\leq i\leq N_p,
\end{equation}

\noindent where again ${\bf u}_{\LO}^{n}={\bf u}_h^{n}$ for $\theta=0$. The comparison of the HO and LO schemes \cref{eq:HO-generic-scheme,eq:LO-generic-scheme} will allow us to introduce the antidiffusive fluxes in the next section that commonly constitute an essential aspect to keep conservation of the present limiter.


%
\subsection{Invariant domain preserving limiter}\label{sec:IDP_limiter_details}

\subsubsection{Antidiffusive fluxes}

Subtracting scheme \cref{eq:LO-generic-scheme} for the LO solution from \cref{eq:HO-generic-scheme} for the HO solution gives

\begin{align}\label{eq:HO-LO-generic-scheme}
 \frac{M_\kappa^i}{\Delta t^{(n)}}\big({\bf U}_{\kappa,\HO}^{i,n+1} - {\bf U}_{\kappa,\LO}^{i,n+1}\big) &= {\bf R}_{\kappa,\LO}^i({\bf u}_{\LO}^{n+\theta}) - {\bf R}_{\kappa,\HO}^i({\bf u}_{\HO}^{n+\theta}) \nonumber\\ &= \sum_{(\kappa',j)\in{\cal N}_\kappa^i} {\bf A}_{\kappa,\kappa'}^{i,j,n+\theta} \quad \forall \kappa\in\Omega_h, 1\leq i\leq N_p,
\end{align}

\noindent where ${\bf A}_{\kappa,\kappa'}^{i,j,n+\theta}$ correspond to the so-called antidiffusive fluxes \cite{zalesak1979fully,kuzmin_turek_FCT_02} and satisfy the conservation property: ${\bf A}_{\kappa,\kappa'}^{i,j,n+\theta}=-{\bf A}_{\kappa',\kappa}^{j,i,n+\theta}$ (hence ${\bf A}_{\kappa,\kappa}^{i,i,n+\theta}=0$). In the usual terminology, ${\cal N}_\kappa^i$ denotes the set of antidiffusive element contributions needed to recover the HO accuracy from the LO scheme \cite[Sec.~2]{kuzmin_turek_FCT_02}. This corresponds to the faces of the element $\kappa$ \cite{BORIS_Book_FCT_73,zalesak1979fully} in the original FCT framework, or to the stencil of the scheme in the convex limiting framework \cite{Guermond_etal_IDP_conv_lim_18}. The antidiffusive fluxes will be defined in \cref{sec:FV-DG_IDP_limiters}. 

The limited solution ${\bf u}_h^{n+1}$ is then defined by

\begin{equation}\label{eq:HO-LO-generic-limited-scheme}
 \frac{M_\kappa^i}{\Delta t^{(n)}}\big({\bf U}_{\kappa}^{i,n+1} - {\bf U}_{\kappa,\LO}^{i,n+1}\big) = \sum_{(\kappa',j)\in{\cal N}_\kappa^i} l_{\kappa,\kappa'}^{i,j}{\bf A}_{\kappa,\kappa'}^{i,j,n+\theta} \quad \forall \kappa\in\Omega_h, 1\leq i\leq N_p,
\end{equation}

\noindent where the limiting coefficients $0\leq l_{\kappa,\kappa'}^{i,j}\leq1$ should be defined to guarantee that ${\bf u}_h^{n+1}\in\D$. Note that setting all the $l_{\kappa,\kappa'}^{i,j}=1$ gives the HO solution, while $l_{\kappa,\kappa'}^{i,j}=0$ gives the LO solution. We further require $l_{\kappa,\kappa'}^{i,j}=l_{\kappa',\kappa}^{j,i}$ to maintain the global conservation of the limited scheme: 

\begin{equation*}
 \sum_{\kappa\in\Omega_h}\sum_{i=1}^{N_p}M_\kappa^i{\bf U}_{\kappa}^{i,n+1} = \sum_{\kappa\in\Omega_h}\sum_{i=1}^{N_p}M_\kappa^i{\bf U}_{\kappa,\HO}^{i,n+1} = \sum_{\kappa\in\Omega_h}\sum_{i=1}^{N_p}M_\kappa^i{\bf U}_{\kappa,\LO}^{i,n+1} = \sum_{\kappa\in\Omega_h}\sum_{i=1}^{N_p}M_\kappa^i{\bf U}_{\kappa}^{i,n}
\end{equation*}

\noindent in the case of compactly supported solutions.

\subsubsection{Limiting coefficients}\label{sec:limiting-coeffs}

The limiting coefficients $l_{\kappa,\kappa'}^{i,j}$ in \cref{eq:HO-LO-generic-limited-scheme} are defined as follows. Let us introduce coefficients $\alpha_{\kappa'}^j>0$ such that $\sum_{(\kappa',j)\in{\cal N}_\kappa^i}\alpha_{\kappa'}^j=1$, then rewrite \cref{eq:HO-LO-generic-limited-scheme} as

\begin{equation}\label{eq:generic-limited-scheme-cv-comb}
 {\bf U}_{\kappa}^{i,n+1} = {\bf U}_{\kappa,\LO}^{i,n+1} + \frac{\Delta t^{(n)}}{M_\kappa^i} \sum_{(\kappa',j)\in{\cal N}_\kappa^i} l_{\kappa,\kappa'}^{i,j}{\bf A}_{\kappa,\kappa'}^{i,j,n+\theta} = \sum_{(\kappa',j)\in{\cal N}_\kappa^i} \alpha_{\kappa'}^j\Big({\bf U}_{\kappa,\LO}^{i,n+1} + \frac{l_{\kappa,\kappa'}^{i,j}}{\alpha_{\kappa'}^j}\frac{\Delta t^{(n)}}{M_\kappa^i}{\bf A}_{\kappa,\kappa'}^{i,j,n+\theta} \Big).
\end{equation}

Now we look for the largest coefficients $0\leq k_{\kappa'}^j\leq1$ such that

\begin{equation}\label{eq:tentative_limited_DOF_contrib}
 {\bf U}_{\kappa,\LO}^{i,n+1} + k_{\kappa'}^j\frac{\Delta t^{(n)}}{M_\kappa^i}{\bf A}_{\kappa,\kappa'}^{i,j,n+\theta} \in\D
\end{equation}

\noindent and then set 

\begin{equation}\label{eq:choice-for-lkappa-coeffs}
 l_{\kappa,\kappa'}^{i,j}=l_{\kappa',\kappa}^{j,i}:=\min(\alpha_{\kappa'}^jk_{\kappa'}^j,\alpha_\kappa^ik_\kappa^i)
\end{equation}

\noindent to keep conservation. Note that the coefficients $k_{\kappa'}^j$ in \cref{eq:tentative_limited_DOF_contrib} are well defined by convexity of $\D$ and the fact that ${\bf U}_{\kappa,\LO}^{i,n+1}\in\D$. 

The optimal choices for the $\alpha_{\kappa'}^j$ are the ones that balance the effects of the limiter from every contribution $(\kappa',j)\in{\cal N}_\kappa^i$ to the $i$th DOF. We therefore impose $\alpha_{\kappa'}^jk_{\kappa'}^j=\alpha_\kappa^ik_\kappa^i$ for all $(\kappa',j)$ in ${\cal N}_\kappa^i$, so using $\sum_{(\kappa',j)\in{\cal N}_\kappa^i}\alpha_{\kappa'}^j=1$ gives

\begin{equation}\label{eq:optimal-choice-for-alpha}
 \alpha_{\kappa'}^j = \frac{1}{k_{\kappa'}^j}\frac{1}{\sum\limits_{(\kappa'',l)\in{\cal N}_\kappa^i}\frac{1}{k_{\kappa''}^l}} \quad \forall (\kappa',j) \in {\cal N}_\kappa^i.
\end{equation}

Let us propose an interpretation of the present limiter. Assuming that \cref{eq:choice-for-lkappa-coeffs} results in $ l_{\kappa,\kappa'}^{i,j}=\alpha_\kappa^ik_\kappa^i$ for all $(\kappa',j)$ in ${\cal N}_\kappa^i$ and using \cref{eq:generic-limited-scheme-cv-comb} gives

\begin{align*}
 {\bf U}_{\kappa}^{i,n+1} &= {\bf U}_{\kappa,\LO}^{i,n+1} + \frac{1}{\sum\limits_{(\kappa',j)\in{\cal N}_\kappa^i}\frac{1}{k_{\kappa'}^j}} \frac{\Delta t^{(n)}}{M_\kappa^i} \sum\limits_{(\kappa',j)\in{\cal N}_\kappa^i} {\bf A}_{\kappa,\kappa'}^{i,j,n+\theta} \\
 &= \Big(1-\frac{1}{\sum\limits_{(\kappa',j)\in{\cal N}_\kappa^i}\frac{1}{k_{\kappa'}^j}}\Big) {\bf U}_{\kappa,\LO}^{i,n+1} + \frac{1}{\sum\limits_{(\kappa',j)\in{\cal N}_\kappa^i}\frac{1}{k_{\kappa'}^j}}{\bf U}_{\kappa,\HO}^{i,n+1},
\end{align*}

\noindent which would define ${\bf U}_{\kappa}^{i,n+1}$ as a convex combination of ${\bf U}_{\kappa,\LO}^{i,n+1}$ and the HO update ${\bf U}_{\kappa,\HO}^{i,n+1}$.

Note that the above relation may not hold when it happens that $l_{\kappa,\kappa'}^{i,j}<\alpha_\kappa^ik_\kappa^i$ for some $(\kappa',j)\in {\cal N}_\kappa^i$ due to \cref{eq:choice-for-lkappa-coeffs}; however, the choice \cref{eq:tentative_limited_DOF_contrib,eq:choice-for-lkappa-coeffs,eq:optimal-choice-for-alpha} is an attempt to keep the limited solution as close as possible to local convex combinations of the LO and HO solutions, while maintaining conservation.

As a consequence of the limiter \cref{eq:HO-LO-generic-limited-scheme}, the limited solution is a convex combination of quantities that are in $\D$ and therefore satisfies all the invariant domains: ${\bf U}_\kappa^{i,n+1}\in\D$. Indeed, rewriting again \cref{eq:HO-LO-generic-limited-scheme} as in \cref{eq:generic-limited-scheme-cv-comb} and using the fact that $l_{\kappa,\kappa'}^{i,j}\leq\alpha_{\kappa'}^jk_{\kappa'}^j$ by \cref{eq:choice-for-lkappa-coeffs}, we have

\begin{equation*}
 {\bf U}_{\kappa,\LO}^{i,n+1} + \frac{l_{\kappa,\kappa'}^{i,j}}{\alpha_{\kappa'}^j}\frac{\Delta t^{(n)}}{M_\kappa^i}{\bf A}_{\kappa,\kappa'}^{i,j,n+\theta} = \Big(1-\frac{l_{\kappa,\kappa'}^{i,j}}{\alpha_{\kappa'}^jk_{\kappa'}^j}\Big){\bf U}_{\kappa,\LO}^{i,n+1} + \frac{l_{\kappa,\kappa'}^{i,j}}{\alpha_{\kappa'}^jk_{\kappa'}^j}\Big({\bf U}_{\kappa,\LO}^{i,n+1} + k_{\kappa'}^j\frac{\Delta t^{(n)}}{M_\kappa^i}{\bf A}_{\kappa,\kappa'}^{i,j,n+\theta}\Big),
\end{equation*}

\noindent which is a convex combination of quantities in $\D$ from \cref{eq:tentative_limited_DOF_contrib} and we conclude that ${\bf U}_\kappa^{i,n+1}\in\D$ from \cref{eq:generic-limited-scheme-cv-comb}.

\begin{remark}
 The decomposition \cref{eq:generic-limited-scheme-cv-comb} follows the convex limiting strategy \cite[Eq.~(7.7)]{Guermond_etal_IDP_conv_lim_18} and quasiconcave functions may be also used to estimate the $k_{\kappa'}^j$ in \cref{eq:tentative_limited_DOF_contrib}. We however differ from this framework by the choice of the limiting coefficients $l_{\kappa,\kappa'}^{i,j}$ and the convex coefficients $\alpha_{\kappa'}^j$ in \cref{eq:tentative_limited_DOF_contrib,eq:choice-for-lkappa-coeffs,eq:optimal-choice-for-alpha} as described above. In particular, the choice of the convex coefficients $\alpha_{\kappa'}^j$ is arbitrary in the convex limiting approach \cite[Rem.~7.22]{Guermond_etal_IDP_conv_lim_18}, while they are here defined by \cref{eq:optimal-choice-for-alpha} to optimize the effect of the limiter. Likewise, the present limiter applies to both implicit and explicit time stepping as described in \cref{sec:time_discretization}. 
\end{remark}

\subsection{Iterative limiter}\label{sec:iterative_limiter}

Once the limiter is applied to produce a limited solution ${\bf u}_h^{n+1}\in\D$, this solution may be reused to again limit the HO solution ${\bf u}_{\HO}^{n+1}$ around this new IDP solution ${\bf u}_h^{n+1}$. This defines an iterative limiting procedure that has also been proposed in the context of the FCT limiter \cite{Schar_iterative_FCT_96}. We here describe this procedure in the context of the present limiter.

Let ${\bf u}_h^{n+1,(k)}\in\D$ be an IDP solution after iteration $k\geq0$, the objective being to apply the limiter introduced above in \cref{sec:IDP_limiter_details} to define a new conservative IDP solution ${\bf u}_h^{n+1,(k+1)}\in\D$ closer to ${\bf u}_{\HO}^{n+1}$. 
The unlimited and limited schemes read 

\begin{align}
 \frac{M_\kappa^i}{\Delta t^{(n)}}\big({\bf U}_{\kappa,\HO}^{i,n+1} - {\bf U}_{\kappa}^{i,n+1,(k)}\big) &= \sum_{(\kappa',j)\in{\cal N}_\kappa^i} {\bf A}_{\kappa,\kappa'}^{i,j,n+\theta,(k)}, \nonumber\\
 \frac{M_\kappa^i}{\Delta t^{(n)}}\big({\bf U}_{\kappa}^{i,n+1,(k+1)} - {\bf U}_{\kappa}^{i,n+1,(k)}\big) &= \sum_{(\kappa',j)\in{\cal N}_\kappa^i} l_{\kappa,\kappa'}^{i,j,(k)}{\bf A}_{\kappa,\kappa'}^{i,j,n+\theta,(k)}, \label{eq:iterative-limiter-b} 
\end{align}

\noindent where we look for coefficients $0\leq l_{\kappa,\kappa'}^{i,j,(k)}=l_{\kappa',\kappa}^{j,i,(k)}\leq 1$ such that ${\bf u}_h^{n+1,(k+1)}\in\D$.

Using the notations in \cref{sec:IDP_limiter_details}, we have ${\bf u}_h^{n+1,(0)}\equiv{\bf u}_{\LO}^{n+1}$ defined from \cref{eq:LO-generic-scheme} and ${\bf u}_h^{n+1,(1)}\equiv{\bf u}_h^{n+1}$ the first limited solution defined from \cref{eq:HO-LO-generic-limited-scheme}, while $l_{\kappa',\kappa}^{j,i,(0)}=l_{\kappa,\kappa'}^{i,j}$ and ${\bf A}_{\kappa,\kappa'}^{i,j,n+\theta,(0)}={\bf A}_{\kappa,\kappa'}^{i,j,n+\theta}$. 

Subtracting the two above equations, we get the relation between the HO and limited solutions:

\begin{equation*}
 \frac{M_\kappa^i}{\Delta t^{(n)}}\big({\bf U}_{\kappa,\HO}^{i,n+1} - {\bf U}_{\kappa}^{i,n+1,(k+1)}\big) = \sum_{(\kappa',j)\in{\cal N}_\kappa^i} (1-l_{\kappa,\kappa'}^{i,j,(k)}){\bf A}_{\kappa,\kappa'}^{i,j,n+\theta,(k)} = \sum_{(\kappa',j)\in{\cal N}_\kappa^i} {\bf A}_{\kappa,\kappa'}^{i,j,n+\theta,(k+1)}, 
\end{equation*}

\noindent  where the ${\bf A}_{\kappa,\kappa'}^{i,j,n+\theta,(k+1)}:=(1-l_{\kappa,\kappa'}^{i,j,(k)}){\bf A}_{\kappa,\kappa'}^{i,j,n+\theta,(k)}$ constitute the antidiffusive fluxes between the new LO solution, ${\bf u}_{h}^{n+1,(k+1)}$, and the HO solution, ${\bf u}_{\HO}^{n+1}$ that may be reused for the next step and further limited with $l_{\kappa,\kappa'}^{i,j,(k+1)}$ coefficients. By induction, we obtain an explicit form of the antidiffusive fluxes ${\bf A}_{\kappa,\kappa'}^{i,j,n+\theta,(k+1)}=\Pi_{l=0}^k(1-l_{\kappa,\kappa'}^{i,j,(l)}){\bf A}_{\kappa,\kappa'}^{i,j,n+\theta}$. This defines an iterative algorithm that does not require one to apply the numerical scheme: at a given iteration $k\geq0$ we only need to compute the new limiting coefficients $l_{\kappa,\kappa'}^{i,j,(k)}$ and scale the antidiffusive fluxes by $(1-l_{\kappa,\kappa'}^{i,j,(k)})$ to obtain the next ones.

The algorithm flowchart of the iterative limiter is described in \Cref{algo:iterative_IDP_limiter}, where the iterative loop is stopped when a maximum number of iterations or a given tolerance on the relative distance between two consecutive limited solutions is reached. The following results provide some information about convergence of the present iterative limiter.

\begin{proposition}[Existence of a fixed point]
 The iterative limiter \cref{eq:iterative-limiter-b} has at least one fixed point.
\end{proposition}
 
\begin{proof}
 Every iteration of the limiter satisfies $\|{\bf U}_{\kappa,\HO}^{i,n+1} - {\bf U}_{\kappa}^{i,n+1,(k)}\|\leq \tfrac{\Delta t^{(n)}}{M_\kappa^i}\sum_{(\kappa',j)}\|{\bf A}_{\kappa,\kappa'}^{i,j,n+\theta}\|$ and therefore remains in a finite dimensional ball equipped with the Euclidean norm and centered on ${\bf U}_{\kappa,\HO}^{i,n+1}$. By continuity of the mapping \cref{eq:iterative-limiter-b} between two successive limited solutions, there exists at least one fixed point according to Brouwer's fixed point theorem in finite dimension.  
\end{proof}

\begin{proposition}[Recovering IDP HO solutions]\label{th:recover_IDP_HO_sol}
 Assume that ${\bf U}_{\kappa,\HO}^{i,n+1}\in\D$ and that there exists $\epsilon>0$ such that the ball ${\cal B}({\bf U}_{\kappa,\LO}^{i,n+1},\epsilon)\subset\D$ and that all $k_{\kappa'}^{j,(k)}$ in \cref{eq:tentative_limited_DOF_contrib} are chosen to satisfy ${\cal B}({\bf U}_{\kappa}^{i,n+1,(k)} + k_{\kappa'}^{j,(k)}\frac{\Delta t^{(n)}}{M_\kappa^i}{\bf A}_{\kappa,\kappa'}^{i,j,n+\theta},\epsilon) \subset\D$ at every iteration $k\geq0$ of the limiter. Then, the iterative limiter \cref{eq:iterative-limiter-b} converges to a unique fixed point satisfying $\lim_{k\rightarrow\infty}{\bf U}_{\kappa}^{i,n+1,(k)}={\bf U}_{\kappa,\HO}^{i,n+1}$.
\end{proposition}
 
\begin{proof}
 First, we proceed by induction to show that ${\cal B}({\bf U}_{\kappa}^{i,n+1,(k)},\epsilon) \subset\D$ holds for every $k\geq0$. This is true for $k=0$ by assumption: ${\bf U}_{\kappa}^{i,n+1,(0)}={\bf U}_{\kappa,\LO}^{i,n+1}$. Then, using \cref{eq:generic-limited-scheme-cv-comb,eq:choice-for-lkappa-coeffs} at iteration $k$ we rewrite \cref{eq:iterative-limiter-b} as 
 
 \begin{align*}
 {\bf U}_{\kappa}^{i,n+1,(k+1)} =& {\bf U}_{\kappa}^{i,n+1,(k)} + \frac{\Delta t^{(n)}}{M_\kappa^i} \sum_{(\kappa',j)\in{\cal N}_\kappa^i} l_{\kappa,\kappa'}^{i,j,(k)}{\bf A}_{\kappa,\kappa'}^{i,j,n+\theta,(k)} \\
 =& \sum_{(\kappa',j)\in{\cal N}_\kappa^i}\alpha_{\kappa'}^{j,(k)} \Big( {\bf U}_{\kappa}^{i,n+1,(k)} + \frac{l_{\kappa,\kappa'}^{i,j,(k)}}{\alpha_{\kappa'}^{j,(k)}}\frac{\Delta t^{(n)}}{M_\kappa^i}{\bf A}_{\kappa,\kappa'}^{i,j,n+\theta,(k)}\Big) \\
 =& \sum_{(\kappa',j)\in{\cal N}_\kappa^i}\alpha_{\kappa'}^{j,(k)} \bigg( \Big(1-\frac{l_{\kappa,\kappa'}^{i,j,(k)}}{\alpha_{\kappa'}^{j,(k)}k_{\kappa'}^{j,(k)}}\Big){\bf U}_{\kappa}^{i,n+1,(k)} \\ &\hspace{2.25cm}+ \frac{l_{\kappa,\kappa'}^{i,j,(k)}}{\alpha_{\kappa'}^{j,(k)}k_{\kappa'}^{j,(k)}}\Big({\bf U}_{\kappa}^{i,n+1,(k)}+k_{\kappa'}^{j,(k)}\frac{\Delta t^{(n)}}{M_\kappa^i}{\bf A}_{\kappa,\kappa'}^{i,j,n+\theta,(k)}\Big) \bigg)
\end{align*}

\noindent which is indeed a convex combination due to \cref{eq:choice-for-lkappa-coeffs} and confirms the induction. Therefore, there exists $\epsilon_k$ such that $0<\epsilon_k\leq k_{\kappa'}^{j,(k)}$, from which we deduce $\tfrac{\epsilon_k}{{\cal N}_{max}}\leq l_{\kappa,\kappa'}^{i,j,(k)}<1$ by \cref{eq:choice-for-lkappa-coeffs,eq:optimal-choice-for-alpha}, where ${\cal N}_{max}:=\max(\#{\cal N}_{\kappa\in\Omega_h}^{1\leq i\leq N_p})$. We therefore have $0<\Pi_{l=0}^{k-1}(1-l_{\kappa,\kappa'}^{i,j,(l)})\leq(1-\tfrac{\epsilon_k}{{\cal N}_{max}})^{k+1}$ and we conclude by noting that ${\bf A}_{\kappa,\kappa'}^{i,j,n+\theta,(k+1)}=\Pi_{l=0}^{k-1}(1-l_{\kappa,\kappa'}^{i,j,(l)}){\bf A}_{\kappa,\kappa'}^{i,j,n+\theta}\rightarrow_k0$.
\end{proof}

\begin{corollary}[Maximum convergence speed]
 Under the assumptions of \cref{th:recover_IDP_HO_sol} and further assuming that no limiting is needed, i.e. every $k_{\kappa'}^{j,(k)}=1$, together with a uniform stencil, $\#{\cal N}_{\kappa\in\Omega_h}^{1\leq i\leq N_p}={\cal N}_{max}$, the iterative limiter is a contraction mapping with $1-\tfrac{1}{{\cal N}_{max}}$ as Lipschitz coefficient:
 \begin{equation*}
  \|{\bf U}_{\kappa,\HO}^{i,n+1} - {\bf U}_{\kappa}^{i,n+1,(k+1)}\|=\big(1-\tfrac{1}{{\cal N}_{max}}\big)\|{\bf U}_{\kappa,\HO}^{i,n+1} - {\bf U}_{\kappa}^{i,n+1,(k)}\|.
 \end{equation*}
\end{corollary}

\begin{remark}\label{rk:iterative_algo}
 Some remarks are in order. First, the assumption of \cref{th:recover_IDP_HO_sol} is easily ensured by reducing the limiting coefficients $k_{\kappa'}^{j,(k)}$ from their theoretical values (see \cref{sec:IDP_constraints} for their practical evaluation), which is usually applied in practice to ensure that ${\bf U}_{\kappa}^{i,n+1,(k)} + k_{\kappa'}^{j,(k)}\frac{\Delta t^{(n)}}{M_\kappa^i}{\bf A}_{\kappa,\kappa'}^{i,j,n+\theta}$ is included in the open convex set $\D$. Then, the convergence speed decreases as ${\cal N}_{max}$ increases and may reach low values for HO methods with large stencils. This motivates the slight modification of \cref{algo:iterative_IDP_limiter} in \cref{sec:acc_iterative_limiter} to accelerate its convergence, as well as the gathering of antidiffusive flux contributions at the same interface for DG schemes in \cref{sec:DG_schemes,sec:DGSEM_IDP_limiter} to lower ${\cal N}_{max}$.
\end{remark}

\begin{algorithm}
\caption{Algorithm flowchart of the iterative IDP limiter, where the user-defined parameters $k_{max}\geq1$ and $0<\TOL\ll 1$ denote the maximum number of iterations and the tolerance, respectively.}\label{algo:iterative_IDP_limiter}
\begin{algorithmic}[1]
\REQUIRE ${\bf u}_h^n$, $k_{max}$, $\TOL$
\STATE{compute ${\bf u}_{\HO}^{n+1}$ from \cref{eq:HO-generic-scheme}}
\STATE{compute ${\bf u}_{\LO}^{n+1}$ from \cref{eq:LO-generic-scheme}}
\STATE{compute ${\bf A}_{\kappa,\kappa'}^{i,j,n+\theta}$ defined in \cref{eq:HO-LO-generic-scheme}}
\STATE{$r_0\gets\|{\bf u}_{\HO}^{n+1}-{\bf u}_{\LO}^{n+1}\|_{L^2(\Omega_h)}^2$}
\STATE{$r\gets r_0$}
\STATE{$k\gets$ 0}
\WHILE{$k\leq k_{max}$ \AND $r>\TOL\,r_0$} 
    \STATE{compute the limiting coefficients $l_{\kappa,\kappa'}^{i,j}$ from \cref{eq:choice-for-lkappa-coeffs,eq:optimal-choice-for-alpha}}
    \STATE{${\bf A}_{\kappa,\kappa'}^{i,j,n+\theta}\gets(1-l_{\kappa,\kappa'}^{i,j}){\bf A}_{\kappa,\kappa'}^{i,j,n+\theta}$}
    \STATE{compute ${\bf u}_h^{n+1}$ from \cref{eq:HO-LO-generic-limited-scheme}}
    \STATE{$r\gets \|{\bf u}_{h}^{n+1}-{\bf u}_{\LO}^{n+1}\|_{L^2(\Omega_h)}^2$ computed from \cref{eq:iterative-limiter-b}}
    \STATE{${\bf u}_{\LO}^{n+1}\gets {\bf u}_h^{n+1}$}
    \STATE{$k\gets k+1$}
\ENDWHILE
\RETURN{${\bf u}_h^{n+1}$}
\end{algorithmic}
\end{algorithm} 

\subsection{Accelerating the iterative process}\label{sec:acc_iterative_limiter}

We here propose a heuristic to accelerate the convergence of the iterative algorithm. \Cref{algo:iterative_IDP_limiter} is modified by scaling the antidiffusive coefficients with a factor $\beta \ge 1$. For $\beta > 1$, the iterative solution is updated using potentially larger limiting coefficients, accelerating the convergence of the iterative limiter. In practice, the limiting coefficients in step 8 of \cref{algo:iterative_IDP_limiter} are modified through  \cref{eq:tentative_limited_DOF_contrib,eq:choice-for-lkappa-coeffs}: we now look for $0\leq k_{\kappa'}^j\leq 1$ such that

\begin{equation*}
 {\bf U}_{\kappa,\LO}^{i,n+1} + \beta k_{\kappa'}^j\frac{\Delta t^{(n)}}{M_\kappa^i}{\bf A}_{\kappa,\kappa'}^{i,j,n+\theta} \in\D
\end{equation*}

\noindent and then set 

\begin{equation*}
 l_{\kappa,\kappa'}^{i,j}=l_{\kappa',\kappa}^{j,i}:=\beta\min(\alpha_{\kappa'}^jk_{\kappa'}^j,\alpha_\kappa^ik_\kappa^i).
\end{equation*}

A schematic example is shown in \cref{fig:schematic_iterative_loop}, where the case ${\bf u}_{\HO}^{n+1} \in \mathcal{D}$ is considered. Using an acceleration factor $\beta = 2$, the convergence toward the HO solution is achieved more rapidly than with the standard choice $\beta = 1$, as evidenced by the faster decay of the update step size measured in both $L^1$ and $L^2$ norms over the IDP iterations. For completeness, a second case is also presented in which ${\bf u}_{\HO}^{n+1} \notin \mathcal{D}$ and $\beta = 1$. In this situation, the same convergence behavior is observed for all $\beta \ge 1$, due to the effect of the limiting coefficients $k_{\kappa}^{\pm}$.

\begin{figure}[H]
    \centering
    \subfloat[${\bf u}_{\HO}^{n+1} \in \D$, $\beta=1$]{
        \begin{minipage}{0.3\linewidth}
            \centering
            \includegraphics[trim={0.15cm 0.15cm 0.5cm 0.15cm},clip,width=\linewidth]{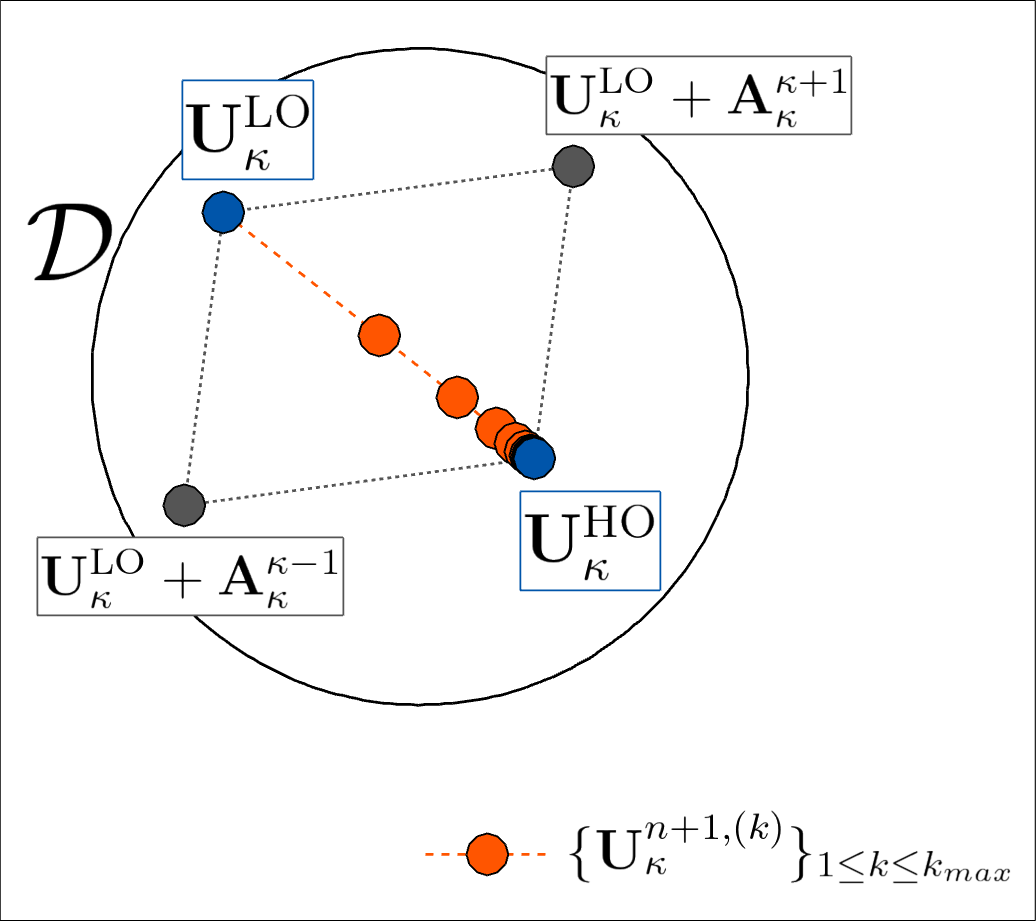}
            \includegraphics[trim={0.15cm 0.15cm 0.15cm 0.15cm},clip,width=\linewidth]{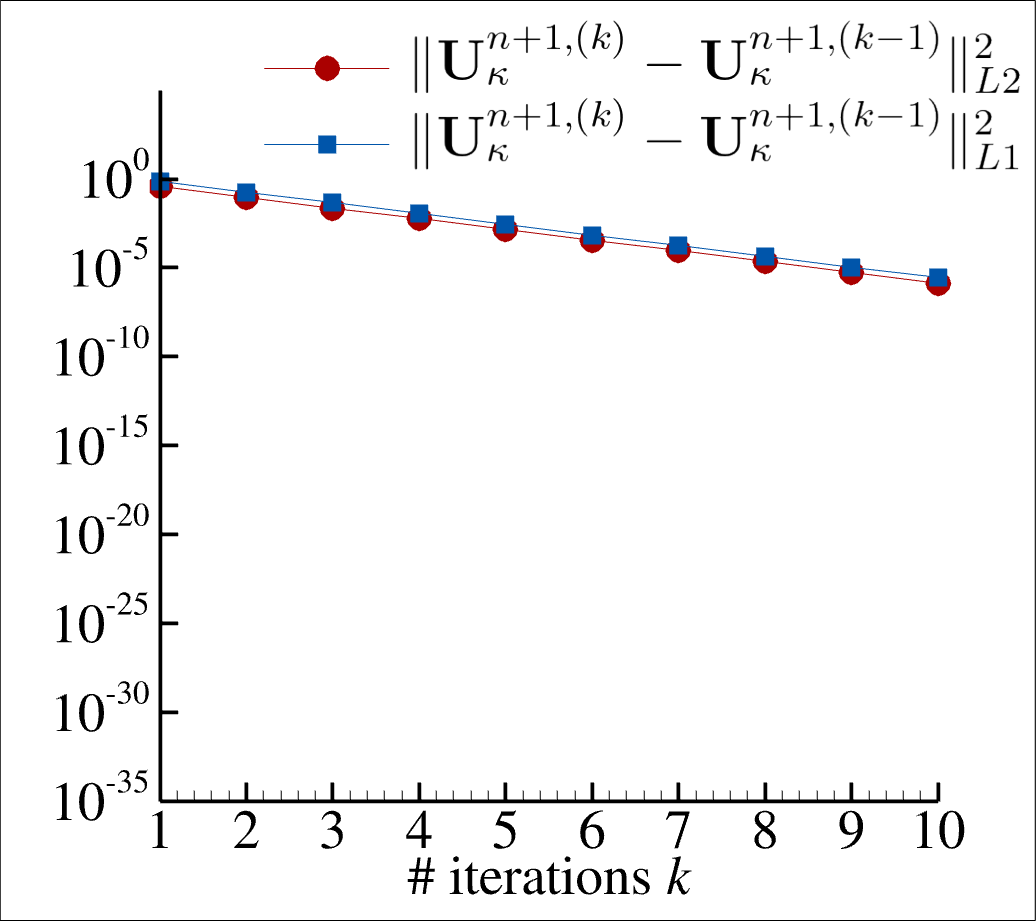}
        \end{minipage}
    }
    \hfill
    \subfloat[${\bf u}_{\HO}^{n+1} \in \D$, $\beta=2$]{
        \begin{minipage}{0.3\linewidth}
            \centering
            \includegraphics[trim={0.15cm 0.15cm 0.5cm 0.15cm},clip,width=\linewidth]{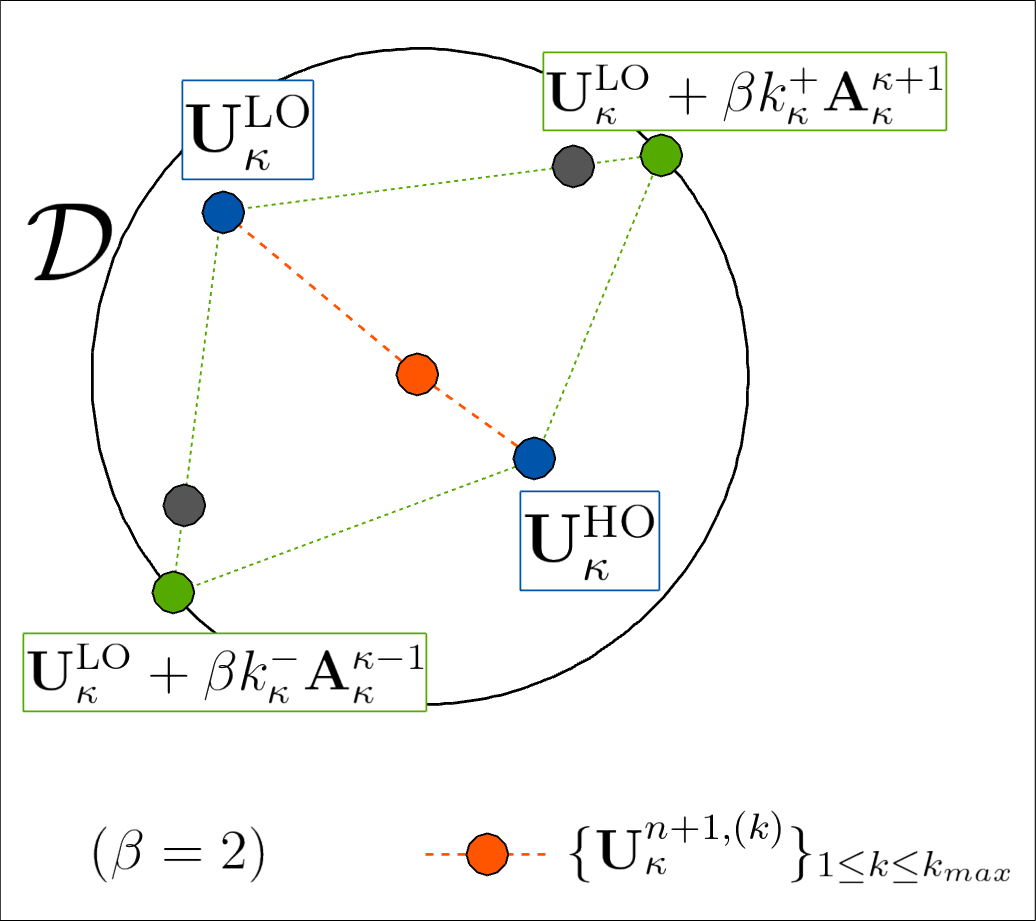}
            \includegraphics[trim={0.15cm 0.15cm 0.15cm 0.15cm},clip,width=\linewidth]{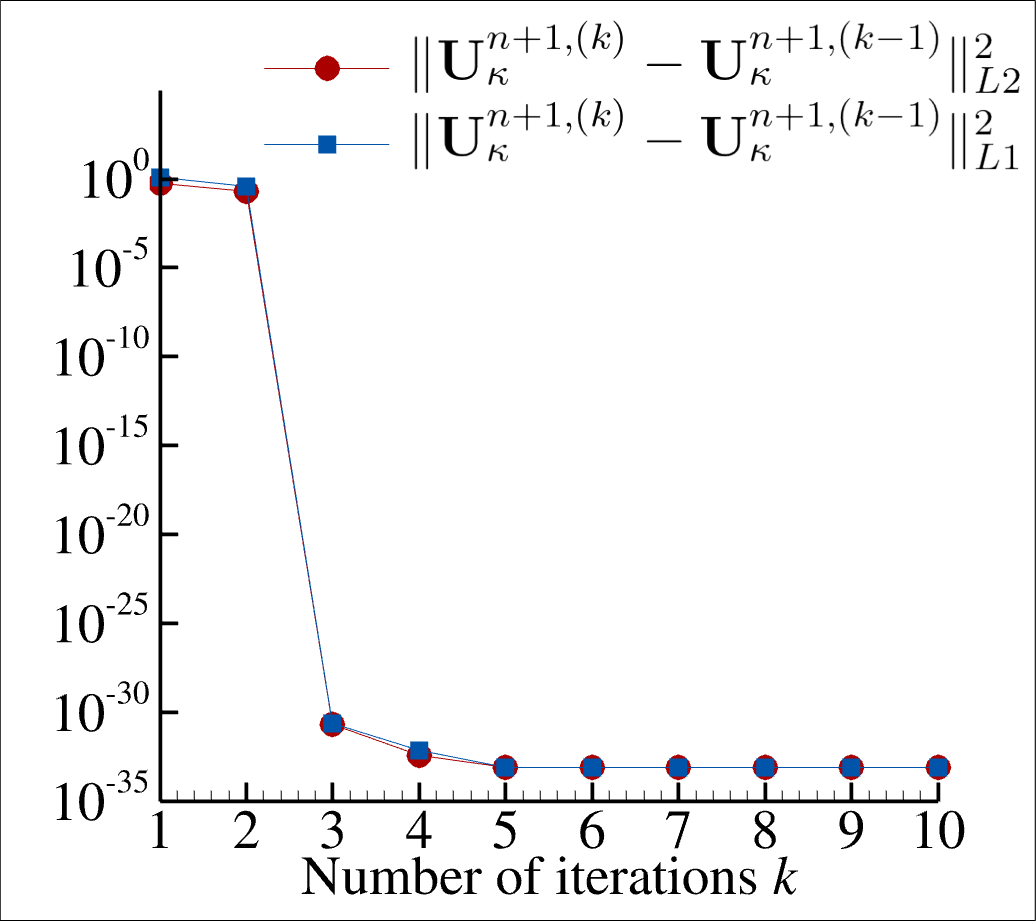}
        \end{minipage}
    }
    \hfill
    \subfloat[${\bf u}_{\HO}^{n+1} \notin \D$, $\beta=1$]{
        \begin{minipage}{0.3\linewidth}
            \centering
            \includegraphics[trim={0.15cm 0.15cm 0.15cm 0.15cm},clip,width=\linewidth]{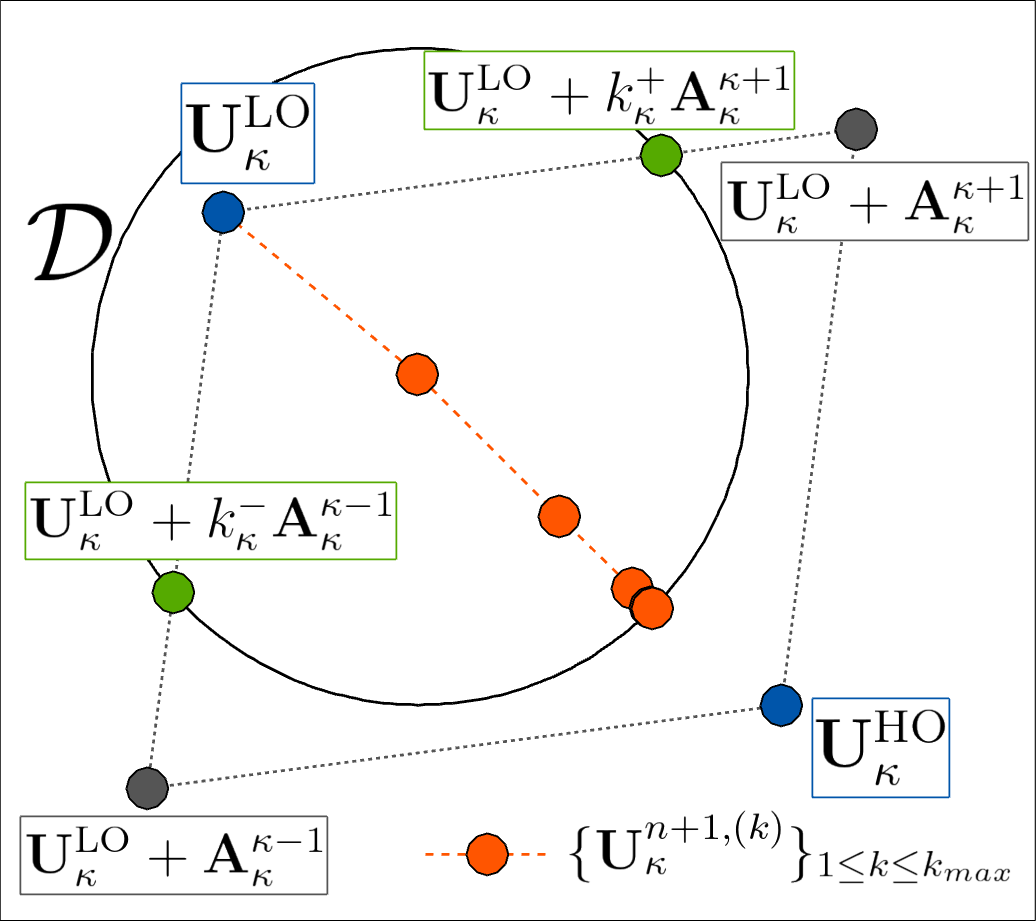}
            \includegraphics[trim={0.15cm 0.15cm 0.15cm 0.15cm},clip,width=\linewidth]{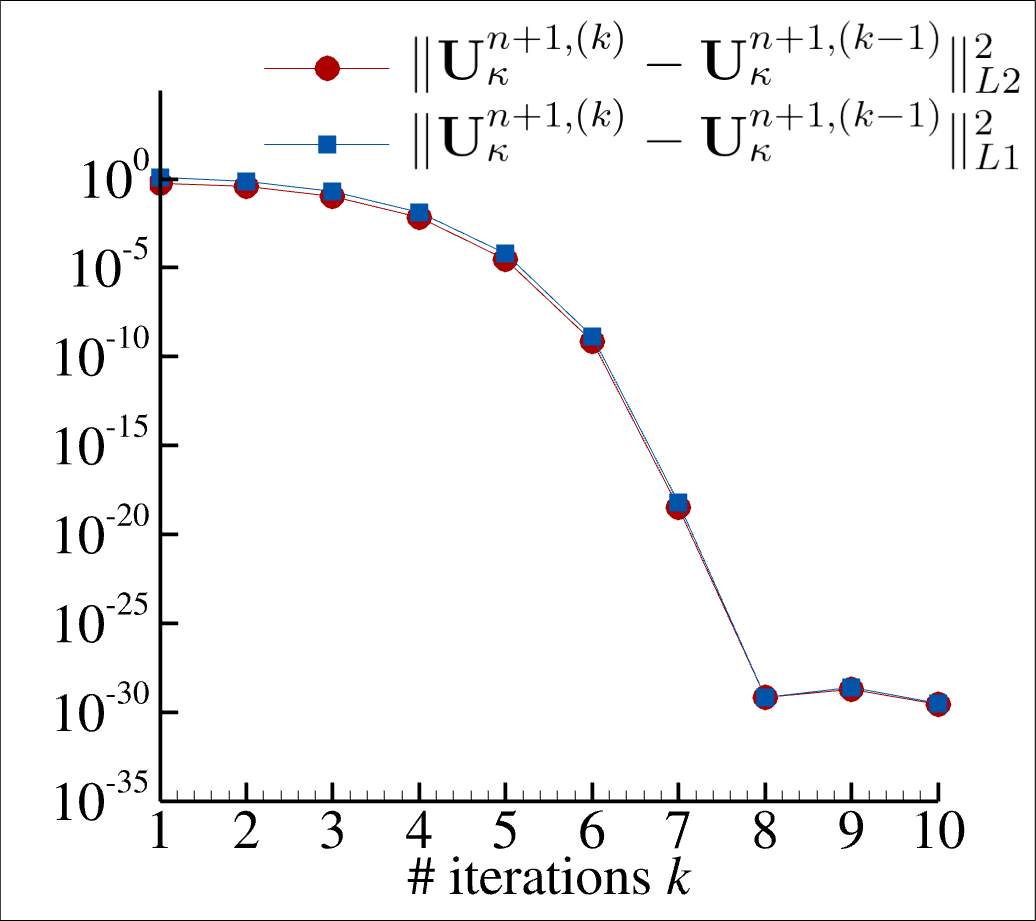}
        \end{minipage}
    }

    \caption{Top row: Schematic illustration of \cref{algo:iterative_IDP_limiter} in a simplified setting with $d=1$, $n_{eq}=2$ and $\#{\cal N}_\kappa^i=2$, where the convex invariant domain $\D$ is a disk delimited by the circle. Bottom row: Evolution of the $L^1$ and $L^2$ distances between the limited solutions at iterations $k$ and $k-1$.}
    \label{fig:schematic_iterative_loop}
\end{figure}
\color{black}

\subsection{IDP constraints}\label{sec:IDP_constraints}

Here we discuss constraints in \cref{eq:tentative_limited_DOF_contrib} and how to compute the $k_\kappa^i$ coefficients. The constraints we discuss are not new and we recall them for the sake of completeness of the description of the limiter and of the numerical results. We, however, emphasize that we leverage the iterative character of the limiter to avoid line-searches to evaluate the $k_\kappa^i$ and rather use explicit formulas that may be less accurate, but also less expansive. For the sake of clarity, we here consider a generic constraint and look for $0\leq k\leq 1$ such that

\begin{equation*}
 {\bf U}_{\LO}+k{\bf A}=(1-k){\bf U}_{\LO}+k({\bf U}_{\LO}+{\bf A})\in\D.
\end{equation*}

By $m$ and $M$ in $\R$, we denote some generic lower and upper bounds. We will consider the following constraints in the numerical experiments in \cref{sec:num-xp}:

\begin{itemize}
 \item Bounds on some component of ${\bf u}$: $m<u<M$. Examples are the maximum principle for nonlinear scalar equations (see \cref{ex:scalar_eq}) with $m=\essinf_{{\bf x}\in\Omega}u_0({\bf x})$ and $M=\esssup_{{\bf x}\in\Omega}u_0({\bf x})$, or positivity of the density $u=\rho$ with $M=\infty$ and $0<m\ll1$ for the compressible Euler equations (see \cref{ex:Euler_eq}, we typically use $m=10^{-8}$ in the numerical experiments of \cref{sec:num-xp}). In this case, we impose
 \begin{equation*}
  k = \left\{\begin{array}{ll}
      \frac{m-U_{\LO}}{A_u} & \text{if } U_{\LO}+A_u<m, \\
      \frac{M-U_{\LO}}{A_u} & \text{if } M<U_{\LO}+A_u, \\
      1 & \text{else,}
  \end{array}\right.
 \end{equation*}

\noindent with $A_u$ the component in ${\bf A}$ corresponding to $u$ in ${\bf u}$.

 \item Lower bound on concave functions of ${\bf u}$: $f({\bf u})>m$, where $\partial_{{\bf u}{\bf u}}^2f({\bf u})$ is negative semi-definite. An example is the internal energy for the compressible Euler equations (see \cref{ex:Euler_eq}). Using concavity of $f(\cdot)$, it is enough to impose $f({\bf U}_{\LO}+k{\bf A})\geq (1-k)f({\bf U}_{\LO})+kf({\bf U}_{\LO}+{\bf A})=m$ for what we obtain
 \begin{equation*}
  k = \max\Big(\min\Big(\frac{f({\bf U}_{\LO})-m}{f({\bf U}_{\LO})-f({\bf U}_{\LO}+{\bf A})}, 1\Big), 0\Big),
 \end{equation*}

\noindent and proceed the same way to impose an upper bound to a convex function, $f({\bf u})<M$.

 \item Bounds on specific entropies: given a convex entropy in \cref{eq:entropy_inequ} of the form $\eta({\bf u})$ with ${\bf u}=(\rho,\rho{\bf y}^\top)^\top$ and $\rho>0$ a positive state variable (e.g., the density for the compressible Euler or traffic flow equations, the water height in the shallow water equations), it is a classical matter that $\s({\bf w})=\s(\tau,{\bf y}^\top)=-\tau\eta(\tfrac{1}{\tau},\tfrac{1}{\tau}{\bf y}^\top)$ is a concave function. Likewise, the mapping ${\bf u}={\bf u}({\bf w})$ is one-to-one and for any given $0\leq \vartheta\leq1$, there exists a unique $0\leq k\leq1$ such that ${\bf U}_{\LO}+k{\bf A}={\bf u}\big((1-\vartheta){\bf w}({\bf U}_{\LO})+\vartheta{\bf w}({\bf U}_{\LO}+{\bf A})\big)$. An example is the specific entropy $(\tfrac{1}{\rho},{\bf w}^\top,E)\mapsto\s(\tfrac{1}{\rho},E-\tfrac{1}{2}{\bf w}\cdot{\bf w})$ for the compressible Euler equations \cite[Sec.~II]{godlewski_raviart_book1991} (see \cref{ex:Euler_eq}), thus giving ${\bf w}({\bf u})=(\tfrac{1}{\rho},{\bf v}^\top, E)^\top$ with ${\bf u}$ defined in \cref{eq:Euler}. Assuming that the LO scheme satisfies a minimum entropy principle, $\s\circ{\bf w}({\bf U}_{\LO})\geq m$, we compute $0\leq \vartheta\leq 1$ from 
  \begin{equation*}
  \vartheta = \max\Big(\min\Big(\frac{\s\circ{\bf w}({\bf U}_{\LO})-m}{\s\circ{\bf w}({\bf U}_{\LO})-\s\circ{\bf w}({\bf U}_{\LO}+{\bf A})}, 1\Big), 0\Big),
 \end{equation*}

 \noindent and then set $k=\tfrac{\vartheta\rho_{\LO}}{\rho_{\LO}+(1-\vartheta)A_\rho}$ with $A_\rho$ the first component in ${\bf A}$ corresponding to $\rho$ in ${\bf u}$. Note that the same results hold for scalar equations with $\rho\equiv\tau\equiv1$ and ${\bf u}\equiv{\bf v}\equiv u$.

 \item Discrete entropy inequalities: assume that the LO scheme satisfies some discrete entropy inequality for a given convex entropy $\eta({\bf u})$ of the form

\begin{equation*}
 \eta({\bf U}_{\LO}) \leq \eta({\bf U}_\kappa^{i,n}) - \frac{\Delta t^{(n)}}{M_\kappa^i}{\bf Q}_{\LO}({\bf u}_{\LO}) =: m \quad \forall \kappa\in\Omega_h, 1\leq i\leq N_p,
\end{equation*}

\noindent where ${\bf Q}_{\LO}$ corresponds to some conservative discretization of $\nabla\cdot{\bf q}({\bf u})$ in \cref{eq:entropy_inequ}. If the HO scheme fails to satisfy the same kind of entropy inequality, it is possible to use the RHS in the above equation as a bound by defining:

 \begin{equation*}
  k = \max\Big(\min\Big(\frac{\eta({\bf U}_{\LO})-m}{\eta({\bf U}_{\LO})-\eta({\bf U}_{\LO}+{\bf A})}, 1\Big), 0\Big),
 \end{equation*}

 \noindent which guarantees that $\eta({\bf U}_{\LO}+k{\bf A})\leq m$ and is always possible since $\eta({\bf U}_{\LO})\leq m$.

  \item Multiple bounds: it is sometimes beneficial to impose multiple bounds to the HO solution. This is done in a sequential way so as to ensure that all bounds can be effectively imposed. For instance, for the compressible Euler equations, one may first limit density, then limit the internal energy and finally impose the maximum principle on entropy. In this sequence, one needs a positive density for the internal energy to be concave, then one needs positive density and internal energy for the physical entropy to be well defined.
 
\end{itemize}

%
%
\section{Time integration schemes}\label{sec:time_discretization}

We here describe how the IDP limiter introduced in \cref{sec:IDP_limiter} can be adapted to HO explicit and implicit time integration schemes. We pay attention to lower the cost of the limiter by limiting the computation of the LO scheme to a time integration with a single step. The precise definition of the antidiffusive fluxes depends on the choice of the space discretization and will be given in \cref{sec:FV-DG_IDP_limiters}.

\subsection{Runge-Kutta schemes}\label{sec:RK_schemes}

A general ${\cal S}$-stage RK method in Butcher form for the HO time integration of \cref{eq:generic-semi-discr-scheme} reads

\begin{subequations}\label{eq:RK_scheme}
\begin{align}
 {\bf u}_h^{(n,k)} &= {\bf u}_h^{n} - \Delta t^{(n)} {\bf M}^{-1}\sum_{l=1}^{k+\theta-1} a_{kl}{\bf R}_h({\bf u}_h^{(n,l)}) \quad \forall 1\leq k\leq {\cal S}, \label{eq:RK_scheme-a}\\
 {\bf u}_h^{n+1} &= {\bf u}_h^{n} - \Delta t^{(n)} {\bf M}^{-1}\sum_{l=1}^{{\cal S}} b_{l}{\bf R}_h({\bf u}_h^{(n,l)}),  \label{eq:RK_scheme-b}
\end{align}
\end{subequations}

\noindent with $\sum_{l=1}^{{\cal S}} b_{l}=1$ for time consistency. We here consider the cases of explicit RK methods for $\theta=0$ and diagonally implicit Runge-Kutta (DIRK) methods for $\theta=1$. Note that  by \cref{eq:RK_scheme-a} we have ${\bf u}_h^{(n,1)} = {\bf u}_h^{n}$ in the explicit case $\theta=0$. We define 

\begin{equation}\label{eq:def-ck-coeff-in-RKM}
 c_k:=\sum_{l=1}^{k+\theta-1}a_{kl}\geq0 \quad \forall 1\leq k\leq {\cal S},
\end{equation}

\noindent and assume that they are nonnegative without loss of generality, this property being satisfied by most of the RK methods. 

Let us consider the limitation of the $k$th stage \cref{eq:RK_scheme-a}, the limitation of \cref{eq:RK_scheme-b} being similar. Given ${\bf u}_h^{n}$ and ${\bf u}_h^{(n,l)}$ in $\D$ for $1\leq l< k$, we aim at limiting the solution ${\bf u}_{\HO}^{(n,k)}$ of the HO scheme

\begin{equation}\label{eq:RK_stage_HO_scheme}
 {\bf u}_{\HO}^{(n,k)} = {\bf u}_h^{n} - \Delta t^{(n)} {\bf M}^{-1}\sum_{l=1}^{k+\theta-1} a_{kl}{\bf R}_{\HO}({\bf u}_h^{(n,l)}).
\end{equation}

\noindent with ${\bf u}_h^{(n,k)}={\bf u}_{\HO}^{(n,k)}$ when $\theta=1$. 

We now define a single-step LO scheme to reduce the cost of the limiter. In the explicit case, $\theta=0$, we apply the forward-Euler scheme, while we use the backward-Euler scheme in the implicit case, $\theta=1$, to define the quantity

\begin{equation}\label{eq:def-global-uLO-RKM}
  {\bf u}_{\LO}^{n+1} = {\bf u}_h^{n} - c_\infty \Delta t^{(n)} {\bf M}^{-1}{\bf R}_{\LO}({\bf u}_{\LO}^{n+\theta})
\end{equation}

\noindent with $c_\infty:=\max(c_{1\leq k\leq {\cal S}},1)$ and ${\bf u}_{\LO}^{n}={\bf u}_h^{n}$ when $\theta=0$. We assume that the LO scheme satisfies ${\bf u}_{\LO}^{n+1}\in\D$. Note that the time step should be limited in the explicit case and must satisfy $c_\infty\Delta t^{(n)}\leq\Delta t^{(n)}_{FE}$ where $\Delta t^{(n)}_{FE}$ corresponds to the CFL condition for \cref{eq:def-global-uLO-RKM} with $c_\infty=1$ and $\theta=0$. This may also be the case in the implicit case and depends on the choice of LO scheme and on the PDEs as discussed in \cref{sec:FV-DG_IDP_limiters}.


We then use the quantity ${\bf u}_{\LO}^{n+1}$ in \cref{eq:def-global-uLO-RKM} to evaluate the LO solution at the $k$th stage as follows

\begin{align}
 {\bf u}_{\LO}^{(n,k)}  &:= {\bf u}_h^{n} - \Delta t^{(n)} {\bf M}^{-1}\sum_{l=1}^{k+\theta-1} a_{kl}{\bf R}_{\LO}({\bf u}_{\LO}^{n+\theta}) \label{eq:RK_stage_LO_scheme}\\
 &\overset{\cref{eq:def-ck-coeff-in-RKM}}{=} {\bf u}_h^{n} - c_{k}\Delta t^{(n)} {\bf M}^{-1} {\bf R}_{\LO}({\bf u}_{\LO}^{n+\theta}) \nonumber\\
 &= \Big(1-\frac{c_k}{c_\infty}\Big){\bf u}_h^{n} + \frac{c_k}{c_\infty}\Big({\bf u}_h^{n} - c_\infty \Delta t^{(n)} {\bf M}^{-1}{\bf R}_{\LO}({\bf u}_{\LO}^{n+\theta})\Big) \nonumber\\ 
 &\overset{\cref{eq:def-global-uLO-RKM}}{=} \Big(1-\frac{c_k}{c_\infty}\Big){\bf u}_h^{n} + \frac{c_k}{c_\infty} {\bf u}_{\LO}^{n+1}, \nonumber
\end{align}

\noindent which is a convex combination of quantities in $\D$ and thus preserves all the invariant domains. Using $\sum_{l=1}^{{\cal S}} b_{l}=1$, the LO scheme for \cref{eq:RK_scheme-b} reads

\begin{equation*}
  \tilde{\bf u}_{\LO}^{n+1} = {\bf u}_h^{n} - \Delta t^{(n)} {\bf M}^{-1}{\bf R}_{\LO}({\bf u}_{\LO}^{n+\theta}) = \Big(1-\frac{1}{c_\infty}\Big){\bf u}_h^{n} + \frac{1}{c_\infty} {\bf u}_{\LO}^{n+1}\in\D.
\end{equation*}

Subtracting \cref{eq:RK_stage_LO_scheme} from \cref{eq:RK_stage_HO_scheme}, we obtain the equivalent form to \cref{eq:HO-LO-generic-scheme} with the antidiffusive fluxes:

\begin{equation*}
  \frac{M_\kappa^i}{\Delta t^{(n)}}\big({\bf U}_{\kappa,\HO}^{i,(n,k)} - {\bf U}_{\kappa,\LO}^{i,(n,k)}\big) = \sum_{(\kappa',j)\in{\cal N}_\kappa^i} {\bf A}_{\kappa,\kappa'}^{i,j,(n,k)} = \sum_{l=1}^{k+\theta-1}a_{kl}\big({\bf R}_{\LO}({\bf u}_{\LO}^{n+\theta})-{\bf R}_{\HO}({\bf u}_{\HO}^{(n,l)})\big).
\end{equation*}

In both explicit and implicit cases, we need only one computation of the LO solution \cref{eq:def-global-uLO-RKM} independently of the number of stages ${\cal S}$ and one evaluation of the LO residuals \cref{eq:RK_stage_LO_scheme}.

\subsection{Multistep methods}\label{sec:multistep_schemes}

We now consider Adams–Bashforth and Adams–Moulton ${\cal S}$-step methods of the form

\begin{equation}\label{eq:multistep_scheme}
 {\bf u}_h^{n+1} = {\bf u}_h^{n} - \Delta t^{(n)} {\bf M}^{-1}\sum_{l=0}^{\cal S} b_{l}{\bf R}_h({\bf u}_h^{n+1-l}),
\end{equation}

\noindent with $\sum_{l=0}^{\cal S} b_{l}=1$. The scheme is explicit when $b_0=0$ and corresponds to Adams–Bashforth methods, while it is implicit when $b_0\neq0$ and corresponds to Adams–Moulton methods.

The HO scheme reads

\begin{equation*}
 {\bf u}_{\HO}^{n+1} = {\bf u}_h^{n} - \Delta t^{(n)} {\bf M}^{-1}\Big(b_0{\bf R}_{\HO}({\bf u}_{\HO}^{n+1})+\sum_{l=1}^{\cal S} b_{l}{\bf R}_{\HO}({\bf u}_h^{n+1-l})\Big),
\end{equation*}

\noindent 

When $b_0=0$, we define the LO solution from the forward-Euler method

\begin{equation*}
 {\bf u}_{\LO}^{n+1} = {\bf u}_h^{n} - \Delta t^{(n)} {\bf M}^{-1}{\bf R}_{\LO}({\bf u}_h^{n}),
\end{equation*}

\noindent subject to the constraint $\Delta t^{(n)}\leq\Delta t^{(n)}_{FE}$, while we define it from a backward-Euler method in the implicit case $b_0\neq0$:

\begin{equation*}
 {\bf u}_{\LO}^{n+1} = {\bf u}_h^{n} - \Delta t^{(n)} {\bf M}^{-1}{\bf R}_{\LO}({\bf u}_{\LO}^{n+1}).
\end{equation*}

In both cases, we use $\sum_{l=0}^{\cal S} b_{l}=1$ to rewrite the LO scheme as 

\begin{equation*}
 {\bf u}_{\LO}^{n+1} = {\bf u}_h^{n} - \Delta t^{(n)} {\bf M}^{-1}{\bf R}_{\LO}({\bf u}_{\LO}^{n+\theta}) = {\bf u}_h^{n} - \Delta t^{(n)} {\bf M}^{-1}\sum_{l=0}^{\cal S} b_{l}{\bf R}_{\LO}({\bf u}_{\LO}^{n+\theta}),
\end{equation*}

\noindent with $\theta=0$ (resp., $\theta=1$) in the explicit (resp., implicit) case $b_0=0$ (resp., $b_0\neq0$). This way, we do not need to store multiple additional residuals and the equivalent form to \cref{eq:HO-LO-generic-scheme} reads

\begin{align*}
  \frac{M_\kappa^i}{\Delta t^{(n)}}\big({\bf U}_{\kappa,\HO}^{i,(n,k)} - {\bf U}_{\kappa,\LO}^{i,(n,k)}\big) = \sum_{(\kappa',j)\in{\cal N}_\kappa^i} {\bf A}_{\kappa,\kappa'}^{i,j,(n,k)} &= b_0\Big({\bf R}_{\LO}({\bf u}_{\LO}^{n+\theta})-{\bf R}_{\HO}({\bf u}_{\HO}^{n+1})\Big) \\ &+\sum_{l=1}^{\cal S} b_{l}\Big({\bf R}_{\LO}({\bf u}_{\LO}^{n+\theta})-{\bf R}_{\HO}({\bf u}_h^{n+1-l})\Big)
\end{align*}

\subsection{Discontinuous Galerkin method in time}\label{sec:time_DG_discretization}

Let us apply a discontinuous Galerkin (DG) discretization in time of the semi-discrete scheme \cref{eq:generic-semi-discr-scheme}. We look for a numerical solution in the function space of piecewise polynomials of degree $q\geq1$ in time:

\begin{equation*}
 {\bf u}_h(\cdot,t) \equiv \sum_{r=0}^{q}\phi_n^r(t) {\bf u}_h^{(n,r)}(\cdot) \quad \forall t\in(t^{(n)},t^{(n+1)}),
\end{equation*}

\noindent with the $\phi_n^{0\leq r\leq q}(t)$ spanning the space of polynomials of degree $q$ over $(t^{(n)},t^{(n+1)})$. The HO time DG discretization of \cref{eq:generic-semi-discr-scheme} reads \cite{vegt_ven_st_DG_2002,Barth_st_DG_2000}:

\begin{multline*}
 \phi_n^k(t^{(n+1)^-}){\bf u}_{\HO}(t^{(n+1)^-}) - \phi_n^k(t^{(n)^+}){\bf u}_h(t^{(n)^-}) - \Delta t^{(n)}\sum_{r=0}^q \omega_rd_t\phi_n^k(t_n^r){\bf u}_{\HO}(t_n^r) = \\ -\Delta t^{(n)}{\bf M}^{-1}\sum_{r=0}^{q}\omega_r\phi_n^k(t_n^r){\bf R}_{\HO}\big({\bf u}_{\HO}(t_n^r)\big) \quad \forall 0\leq k\leq q,
\end{multline*}

\noindent where we have used an upwind numerical flux at time interfaces $t^{(n)}$ and $t^{(n+1)}$ to satisfy causality, while we have approximated the integrals over $(t^{(n)},t^{(n+1)})$ with some quadrature rule over $[0,1]$ with nodes $(t_n^{0\leq r\leq q})$ and weights $(\omega_{0\leq r\leq q})$ satisfying $\sum_{r=0}^{q}\omega_r=1$. We again assume that the final solution at preceding time step is IDP: ${\bf u}_h(t^{(n)^-})\in\D$.

Setting $\phi_n^k\equiv1_{(t^{(n)},t^{(n+1)})}$ into the above equation, we obtain

\begin{equation*}
 {\bf u}_{\HO}(t^{(n+1)^-}) - {\bf u}_h(t^{(n)^-}) = -\Delta t^{(n)}{\bf M}^{-1}\sum_{r=0}^{q}\omega_r{\bf R}_{\HO}\big({\bf u}_{\HO}(t_n^r)\big).
\end{equation*}

Once again, we define the LO scheme as the solution to the backward-Euler scheme

\begin{align*}
 {\bf u}_{\LO}^{n+1} &= {\bf u}_h(t^{(n)^-}) -\Delta t^{(n)}{\bf M}^{-1}{\bf R}_{\LO}\big({\bf u}_{\LO}^{n+1}\big) \\ 
 &= {\bf u}_h(t^{(n)^-}) -\Delta t^{(n)}{\bf M}^{-1}\sum_{r=0}^{q}\omega_r{\bf R}_{\LO}\big({\bf u}_{\LO}^{n+1}\big),
\end{align*}

\noindent and the antidiffusive fluxes can be defined from the following equivalent form to \cref{eq:HO-LO-generic-scheme}

\begin{align*}
 \frac{1}{\Delta t^{(n)}}{\bf M}\big({\bf u}_{\HO}(t^{(n+1)^-}) - {\bf u}_{\LO}^{n+1}\big) &= \sum_{(\kappa',j)\in{\cal N}_\kappa^i} {\bf A}_{\kappa,\kappa'}^{i,j,n+1} \\
 &= \sum_{r=0}^{q}\omega_r\Big({\bf R}_{\LO}\big({\bf u}_{\LO}^{n+1}\big)-{\bf R}_{\HO}\big({\bf u}_{\HO}(t_n^r)\big)\Big).
\end{align*}

%
%
\section{Derivation of the antidiffusive fluxes}\label{sec:FV-DG_IDP_limiters}

The definition of antidiffusive fluxes depends on the choice of the discretization method. We here describe the antidiffusive fluxes for FV and DG space discretizations that we will use in the numerical experiments of \cref{sec:num-xp} and refer the reader to \cite{kuzmin_turek_FCT_02,KUZMIN_FCT_09,Guermond_etal_IDP_conv_lim_18,Guermond_etal_IDP_conv_lim_19} for more general derivations and the application to other discretization schemes. We restrict ourself to first-order time stepping for the sake of clarity and refer to \cref{sec:time_discretization} for the application of the limiter to various time discretization methods. As in \cref{sec:IDP_limiter}, $\theta=0$ (resp., $\theta=1$) will refer to the forward (resp., backward) Euler time integration and ${\bf u}_{\HO}^{n}={\bf u}_{\LO}^{n}={\bf u}_{h}^{n}$.


%
\subsection{Finite volume schemes}\label{sec:FV_schemes}

We consider unstructured HO FV schemes of the form

\begin{align}
 \frac{|\kappa|}{\Delta t^{(n)}}\big({\bf U}_{\kappa,\HO}^{n+1}-{\bf U}_{\kappa}^{n}\big) + {\bf R}_{\kappa,\HO}({\bf u}_{\HO}^{n+\theta}) &= 0, \label{eq:HO_FV_scheme}\\
 {\bf R}_{\kappa,\HO}({\bf u}_h) &= \sum_{e\in\partial\kappa}\sum_{k=1}^{N_q^e} \omega_e^k|e| {\bf h}\big({\bf u}_h^-({\bf x}_e^k,t),{\bf u}_h^+({\bf x}_e^k,t),{\bf n}_e\big), \nonumber
\end{align}

\noindent where ${\bf U}_{\kappa}^{n}$ approximates the averaged solution of the exact PDE solution ${\bf u}$ in the cell $\kappa$ at time $t^{(n)}$, ${\bf n}_e$ is the unit outward normal vector on the interface $e$ in $\partial\kappa$, and $\kappa_e^+$ the neighboring cell sharing the interface $e$ (see \cref{fig:stencil-FV-2D}). By $|e|$ and $|\kappa|$ we denote $d-1$ and $d$ dimensional measures of $e$ and $\kappa$, respectively, and $(\omega_e^k,{\bf x}_e^k)_{1\leq k\leq N_q^e}$ with $\sum_{k=1}^{N_q^e}\omega_e^k=1$ denotes some quadrature rule that may be required for the HO integration of the numerical flux over faces \cite{barth1990higher,gooch_vanaltena_HO_LSQ_02,haider_etal_k_exact_stab_09,pont_etal_k_exact_RANS_17}. We thus have $N_p=1$ and $M_\kappa^1=|\kappa|$ in \cref{eq:HO-generic-scheme}.

\begin{figure}[ht]
 \begin{center}
  \begin{tikzpicture}[scale=0.7]
   \draw (1.43,1.47) node {$\kappa$};
\draw (5.,1.65)   node {$\kappa_e^+$};
\draw (2.8,4.0)   node[above] {$e$};

\draw [>=stealth,->] (2.94,0.90) -- (3.9,0.95) ;
\draw (3.9,0.95) node[below right] {${\bf n}_e$};
		     
\draw [>=stealth,-] (-1.5,1.4) -- (3.,-1.) ;
\draw [>=stealth,-] (3.,-1.) -- (2.8,4.0) ;
\draw [>=stealth,-] (2.8,4.0) -- (-1.5,1.4) ;
\draw [>=stealth,-] (3.,-1.) -- (6.5,-0.5) ;
\draw [>=stealth,-] (6.5,-0.5) -- (7.5,4.) ;
\draw [>=stealth,-] (7.5,4.) -- (2.8,4.0) ;
  \end{tikzpicture}
  \caption{Notations for the FV mesh for $d=2$.}
  \label{fig:stencil-FV-2D}
 \end{center}
\end{figure}
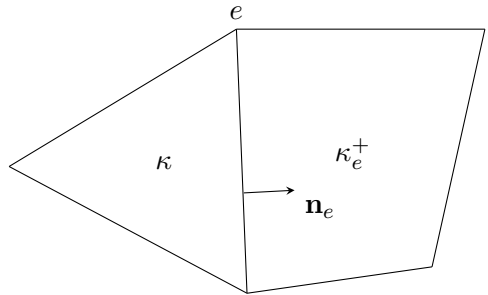

The LO scheme \cref{eq:HO-generic-scheme} is a first-order scheme:

\begin{equation}\label{eq:LO_FV_scheme}
 \frac{|\kappa|}{\Delta t^{(n)}}\big({\bf U}_{\kappa,\LO}^{n+1}-{\bf U}_{\kappa}^{n}\big) + {\bf R}_{\kappa,\LO}({\bf u}_{\LO}^{n+\theta}) = 0, \quad 
 {\bf R}_{\kappa,\LO}({\bf u}_h) = \sum_{e\in\partial\kappa}|e| {\bf h}\big({\bf U}_{\kappa},{\bf U}_{\kappa_e^+},{\bf n}_e\big).
\end{equation}

Using the fact that $\sum_{k=1}^{N_q^e}\omega_e^k=1$, we define the antidiffusive fluxes in \cref{eq:HO-LO-generic-scheme} as

\begin{subequations}\label{eq:AD_fluxes_FVM}
\begin{align}
 \frac{|\kappa|}{\Delta t^{(n)}}\big({\bf U}_{\kappa,\HO}^{n+1}&-{\bf U}_{\kappa,\LO}^{n+1}\big) = \sum_{e\in\partial\kappa} {\bf A}_{\kappa,\kappa_{e}^+}^{e,n+\theta} \\
 {\bf A}_{\kappa,\kappa_{e}^+}^{e} &= \sum_{k=1}^{N_q^e}\omega_e^k|e|\Big( {\bf h}({\bf U}_{\kappa,\LO},{\bf U}_{\kappa_e^+,\LO},{\bf n}_e) - {\bf h}\big({\bf u}_{\HO}^-({\bf x}_e^k,t),{\bf u}_{\HO}^+({\bf x}_e^k,t),{\bf n}_e\big) \Big)=-{\bf A}_{\kappa_{e}^+,\kappa}^{e},
\end{align}
\end{subequations}

\noindent where we have adapted the notations for the antidiffusive fluxes to those of the FV schemes, but without ambiguity and where the skew symmetry property follows from conservation of the numerical fluxes: ${\bf h}({\bf u}^-,{\bf u}^+,{\bf n})=-{\bf h}({\bf u}^+,{\bf u}^-,-{\bf n})$.

The limited solution ${\bf u}_h^{(n+1)}$ in \cref{eq:HO-LO-generic-limited-scheme} finally reads

\begin{equation}\label{eq:IDP_limiter_FVM}
 \frac{|\kappa|}{\Delta t^{(n)}}\big({\bf U}_{\kappa}^{n+1}-{\bf U}_{\kappa,\LO}^{n+1}\big) = \sum_{e\in\partial\kappa} l_{\kappa,\kappa_{e}^+}^{e}{\bf A}_{\kappa,\kappa_{e}^+}^{e,n+\theta},
\end{equation}

\noindent $l_{\kappa,\kappa_{e}^+}^{e}$ the limiting coefficients defined in \cref{sec:limiting-coeffs}. 

The IDP property is satisfied by numerous time explicit, or some time implicit first-order FV schemes for systems of conservation laws under a suitable CFL condition on the time step \cite{Frid_idp_LF_01,Hoff_idp_85,TangXu00positivity} and unconditionally by time implicit monotone FV schemes for scalar equations \cite[\S~26]{EYMARD2000713}. However, proving the IDP property for time implicit discretizations for general systems without any condition on the time step remains an open problem and is beyond the scope of the present work. Past and present numerical experiments, however, show very good robustness and stability properties of first-order FV schemes.

\subsection{Discontinuous Galerkin schemes}\label{sec:DG_schemes}

We again consider a partition $\Omega_h$ of $\Omega$, composed of non-overlapping and non-empty elements $\kappa$, and look for approximate solutions in the function space of piecewise polynomials

\begin{equation*}
 {\cal V}_h^p = \{\phi\in L^2(\Omega_h):\;\phi|_{\kappa}\circ{\bf x}_\kappa\in{\cal P}^p(\hat{K})\; \forall\kappa\in \Omega_h\},
\end{equation*}

\noindent where ${\cal P}^{p}(\hat{K})$ is some polynomial space of polynomials of degree lower or equal to $p$ over a reference element $\hat{K}$ and $N_p$ is its dimension.  By $(\phi_\kappa^1,\dots,\phi_\kappa^{N_p})$, we denote a basis of ${\cal V}_h^p$ restricted onto $\kappa$. Each physical element $\kappa$ is the image of $\hat{K}$ through the mapping ${\bf x}={\bf x}_\kappa(\bxi)$, where $\bxi=(\xi_1,\dots,\xi_d)$ are coordinates in the reference element. Likewise, each face $e$ in the mesh is the image of a reference face $\hat{E}$ through the mapping ${\bf x}={\bf x}_e(\xi_1,\dots,\xi_{d-1})$. We further define Jacobians of the transformations by $J_\kappa({\bf x})=\big|{\bf x}_\kappa'(\bxi)\big|$ and $J_e({\bf x})=\big|{\bf x}_e'(\xi_1,\dots,\xi_{d-1})\big|$.

The approximate solution is sought under the form

\begin{equation}\label{eq:DG_num_sol}
 {\bf u}_h({\bf x},t) = \sum_{i=1}^{N_p} \phi_\kappa^{i}({\bf x}){\bf U}_\kappa^{i}(t) \quad \forall{\bf x}\in\kappa,\, \kappa\in \Omega_h,\,  t\geq0,
\end{equation}

\noindent and we define the cell-averaged solution from some quadrature rule $(\omega_\kappa^j,{\bf x}_{\kappa}^j)_{1\leq j\leq N_q^\kappa}$ over the cells 

\begin{equation*}
\langle {\bf u}_{h}\rangle_\kappa(t) = \frac{1}{|\kappa|}\sum_{i=1}^{N_q^\kappa} \omega_\kappa^i J_{\kappa}^i{\bf u}_h({\bf x}_{\kappa}^i,t),
\end{equation*}

\noindent where $J_\kappa^i=J_\kappa({\bf x}_\kappa^i)$. Using an orthonormal basis with respect to the internal product defined by this quadrature rule over each element, the first-order in time and HO in space DG scheme for the discretization of \cref{eq:hyp_sys_cons_laws-a} reads

\begin{subequations}\label{eqn:fully-discr_DG}
\begin{equation}
 M_\kappa^i \frac{{\bf U}_{\kappa,\HO}^{i,n+1}-{\bf U}_\kappa^{i,n}}{\Delta t^{(n)}} + {\bf R}_{\kappa,\HO}^{i}({\bf u}_{\HO}^{n+\theta}) = 0 \quad \forall \kappa\in \Omega_h, \; 1\leq i\leq N_p, \; n\geq 0,
\end{equation}

\noindent where ${\bf U}_\kappa^{i,n}={\bf U}_\kappa^{i}(t^{n})$ and ${\bf u}_{\HO}^{n}\equiv{\bf u}_{h}^{n}$ when $\theta=0$. The HO space discretization reads

\begin{equation}\label{eqn:weakDG}
 {\bf R}_{\kappa,\HO}^{i}({\bf u}_h) = 
    - \sum_{j=1}^{N_q^\kappa} \omega_\kappa^j J_{\kappa}^j {\bf f}\big({\bf u}_h({\bf x}_{\kappa}^j,t)\big)\cdot \nabla \phi_\kappa^i({\bf x}_{\kappa}^j) + \sum_{e\in\partial\kappa}\sum_{k=1}^{N_q^e} \phi_\kappa^i({\bf x}_e^k)\omega_e^k J_e^k {\bf h}\big({\bf u}_h^-({\bf x}_e^k,t),{\bf u}_h^+({\bf x}_e^k,t),{\bf n}_e^k\big),
\end{equation}
\end{subequations}

\noindent with ${\bf n}_e^k={\bf n}_e({\bf x}_e^k)$ the outward normal, ${\bf u}_h^\mp$ the inner and outer traces of the numerical solution at ${\bf x}_e^k$, $J_e^k=J_e({\bf x}_e^k)$, and $(\omega_e^k,{\bf x}_e^k)_{1\leq k\leq N_q^e}$ some quadrature rule over the face $e$ (see \cref{fig:stencil_2D_DGSEM}). For $\phi_\kappa^1\equiv1_\kappa$, the first sum in \cref{eqn:weakDG} vanishes and we obtain

\begin{equation}\label{eq:fully-discr_DG_average}
 \langle{\bf u}_{\HO}^{n+1}\rangle_\kappa = \langle{\bf u}_h^{n}\rangle_\kappa - \frac{\Delta t^{(n)}}{|\kappa|}\sum_{e\in\partial\kappa}\sum_{k=1}^{N_q^e} \omega_e^k J_e^k {\bf h}\big({\bf u}_{\HO}^-({\bf x}_e^k,t^{(n+\theta)}),{\bf u}_{\HO}^+({\bf x}_e^k,t^{(n+\theta)}),{\bf n}_e^k\big).
\end{equation}

When the test functions $(\phi_\kappa^{1\leq k\leq N_p})$ satisfy the partition of unity properties within elements and on faces and when using a Rusanov two-point numerical flux at interfaces, there exists a general way to derive the antidiffusive fluxes by rewriting the scheme in skew-symmetry form \cite[Sec.~4.3]{Guermond_etal_IDP_conv_lim_18}. 
We now propose a general procedure to derive antidiffusive fluxes for every conservative DG scheme that is also independent of the choice of the numerical flux. It is based on the approaches proposed in \cite[Sec.~4.3]{MRR_BEDGSEM_23} and \cite[Sec.~4]{renac_mpp_dgsem_nlsca_24}, where one defines antidiffusive fluxes for the cell-averaged scheme \cref{eq:fully-discr_DG_average}. The resulting antidiffusive fluxes are however restricted to limit the cell-averaged solution, but they can be associated with the linear scaling limiter introduced in \cite{zhang2010positivity,zhang_shu_10a} to impose these properties to the full solution within elements as will be described in \cref{sec:DGSEM_IDP_limiter} below.

We illustrate the procedure by using the first-order FV scheme \cref{eq:LO_FV_scheme} as the LO scheme (see \cref{sec:DGSEM_IDP_limiter} for another LO scheme using the same function space as the HO scheme), so the antidiffusive fluxes are similar to \cref{eq:AD_fluxes_FVM} for the FV scheme:

\begin{subequations}\label{eq:AD_fluxes_DGM}
\begin{align}
 \frac{|\kappa|}{\Delta t^{(n)}}\big(\langle{\bf u}_{\HO}^{n+1}\rangle_\kappa&-{\bf U}_{\kappa,\LO}^{n+1}\big) = \sum_{e\in\partial\kappa} {\bf A}_{\kappa,\kappa_{e}^+}^{e,n+\theta} \label{eq:AD_fluxes_DGM-a}\\
 {\bf A}_{\kappa,\kappa_{e}^+}^{e} &= \sum_{k=1}^{N_q^e}\omega_e^kJ_e^k\Big( {\bf h}({\bf U}_{\kappa,\LO},{\bf U}_{\kappa_e^+,\LO},{\bf n}_e) - {\bf h}\big({\bf u}_{h}^-({\bf x}_e^k,\cdot),{\bf u}_{h}^+({\bf x}_e^k,\cdot),{\bf n}_e^k\big) \Big)=-{\bf A}_{\kappa_{e}^+,\kappa}^{e}, \label{eq:AD_fluxes_DGM-b}
\end{align}
\end{subequations}

\noindent with ${\bf n}_e=\tfrac{1}{|e|}\sum_{k=1}^{N_q^e}\omega_e^kJ_e^k{\bf n}_e^k$ and $|e|:=\sum_{k=1}^{N_q^e}\omega_e^kJ_e^k$, which defines a FV grid with closed straight-sided elements: $\sum_{e\in\partial\kappa}|e|{\bf n}_e=\sum_{e\in\partial\kappa}\sum_{k=1}^{N_q^e}\omega_e^kJ_e^k{\bf n}_e^k=0$.


%
\subsection{Discontinuous Galerkin spectral element method}\label{sec:DGSEM_IDP_limiter}

We will focus on the DGSEM which is a particular case of DG method introduced above that uses Lagrange interpolation polynomials associated to Gauss-Lobatto quadrature nodes over quadrangles and hexahedra. We here give details on the scheme in 2D, where the solution is sought under the form

\begin{equation}\label{eq:DGSEM_num_sol}
 {\bf u}_h({\bf x},t)=\sum_{0\leq i,j\leq p}\phi_\kappa^{n_{ij}}({\bf x}){\bf U}_\kappa^{n_{ij}}(t) \quad \forall{\bf x}\in\kappa,\, \kappa\in \Omega_h,\, t\geq0,
\end{equation}

\noindent with $({\bf U}_\kappa^{n_{ij}})_{0\leq i,j\leq p}$ the DOFs in the element $\kappa$ with index $n_{ij}=1+i+jp$. The subset $(\phi_\kappa^{n_{ij}})_{0\leq i,j\leq p}$ is an elementwise basis of ${\cal V}_h^p$ of size $N_p=(p+1)^2$. Let $(\omega_p^i, \xi_p^i)_{0\leq i\leq p}$ be the $(p+1)$-point Gauss-Lobatto quadrature rule over $[-1,1]$ and let $(\ell_k)_{0\leq k\leq p}$ be the Lagrange interpolation polynomials associated to the Gauss-Lobatto quadrature nodes: $\ell_k(\xi_p^l)=\delta_{kl}$ for all $0\leq k,l \leq p$, with  $\delta_{kl}$ the Kronecker delta. We use tensor products of these polynomials, $\phi_\kappa^{n_{ij}}({\bf x})=\ell_i(\xi_1)\ell_j(\xi_2)$, therefore the DOFs correspond to the point values of the solution at the quadrature nodes ${\bf x}_\kappa^{n_{ij}}$: ${\bf U}_\kappa^{n_{ij}}(t)={\bf u}_h({\bf x}_\kappa^{n_{ij}},t)$ (see \cref{fig:stencil_2D_DGSEM}).

%
\begin{figure}[ht]
 \begin{center}
  \includegraphics[width=9cm]{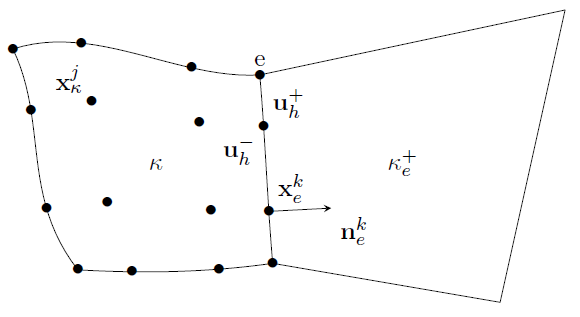}
  \caption{Notations for the DGSEM scheme for $d=2$: inner and outer elements, $\kappa^-$ and $\kappa_e^+$; definitions of traces ${\bf u}_h^\pm$ on the interface $e$ and of the unit outward normal vector ${\bf n}_e^k={\bf n}_e({\bf x}_e^k)$; positions of quadrature points in $\kappa$ and on $e$ for $p=3$ (bullets $\bullet$).}
  \label{fig:stencil_2D_DGSEM}
 \end{center}
\end{figure}

The scheme also uses symmetric entropy conservative two-point fluxes ${\bf h}_{ec}$ in place of the physical flux ${\bf f}({\bf u})$ in the space residuals \cref{eqn:weakDG} for ensuring semi-discrete entropy stability \cite{fisher_carpenter_13,gassner_13}, while the metric terms satisfy the metric identities \cite{kopriva_metric_id_06} to preserve uniform states. The space residuals in \cref{eqn:fully-discr_DG} read

\begin{align}\label{eq:semi-discr_DGSEM-res}
 {\bf R}_{\kappa,\HO}^{n_{ij}}({\bf u}_h)  &= 2\omega_p^i\omega_p^j\Big(\sum_{k=0}^p D_{ik}{\bf h}_{ec}\big({\bf U}_\kappa^{n_{ij}},{\bf U}_\kappa^{n_{kj}},{\bf n}_\kappa^{(i,k)j}\big) + \sum_{k=0}^p D_{jk}{\bf h}_{ec}\big({\bf U}_\kappa^{n_{ij}},{\bf U}_\kappa^{n_{ik}},{\bf n}_\kappa^{i(j,k)}\big)\Big) \nonumber \\
   &+ \sum_{e\in\partial\kappa}\sum_{k=0}^p \phi_\kappa^{n_{ij}}({\bf x}_e^k)\omega_p^kJ_e^k\Big({\bf h}\big({\bf U}_\kappa^{n_{ij}},{\bf u}_h^{+}({\bf x}_e^{k},t),{\bf n}_e^k\big)-{\bf f}({\bf U}_\kappa^{n_{ij}})\cdot{\bf n}_e^k\Big),
\end{align}

\noindent with $D_{ik}=\ell_k'(\xi_p^i)$ and 
\begin{equation*}
 {\bf n}_\kappa^{(i,k)j}=\frac{1}{2}\big(J_\kappa^{n_{ij}}\nabla\xi_1({\bf x}_\kappa^{n_{ij}})+J_\kappa^{n_{kj}}\nabla\xi_1({\bf x}_\kappa^{n_{kj}})\big), \quad {\bf n}_\kappa^{i(j,k)}=\frac{1}{2}\big(J_\kappa^{n_{ij}}\nabla\xi_2({\bf x}_\kappa^{n_{ij}})+J_\kappa^{n_{ik}}\nabla\xi_2({\bf x}_\kappa^{n_{ik}})\big).
\end{equation*}

In the numerical experiments of \cref{sec:num-xp_DG_RK}, we will add graph viscosity \cite{jameson_LED_95,guermond_popov_GV_16,Guermond_etal_IDP_conv_lim_18} to build a LO scheme as advocated in \cite{PAZNER_idg_DGSEM20211,MRR_BEDGSEM_23,renac_mpp_dgsem_nlsca_24} for the DGSEM. We thus use the following LO scheme instead of \cref{eq:LO_FV_scheme}:

\begin{subequations}\label{eqn:fully-discr_LO_DG}
\begin{equation}
 M_\kappa^{n_{ij}} \frac{{\bf U}_{\kappa,\LO}^{n_{ij},n+1}-{\bf U}_\kappa^{n_{ij},n}}{\Delta t^{(n)}} + {\bf R}_{\kappa,\HO}^{n_{ij}}({\bf u}_{\LO}^{n+\theta}) + {\bf V}_{\kappa}^{n_{ij}}({\bf u}_{\LO}^{n+\theta}) = 0 \quad \forall \kappa\in \Omega_h, \; 1\leq i\leq N_p, \; n\geq 0,
\end{equation}

\noindent with $M_\kappa^{n_{ij}}=\omega_p^i\omega_p^j J_\kappa^{n_{ij}}$, ${\bf u}_{\LO}^{n}\equiv{\bf u}_{h}^{n}$ when $\theta=0$, and

\begin{equation}
 {\bf V}_{\kappa}^{n_{ij}}({\bf u}_h) = 
    d_\kappa \omega_p^i\omega_p^j\sum_{k=0}^p\frac{\omega_p^k}{2}|{\bf n}_\kappa^{(ik)j}|({\bf U}_\kappa^{n_{ij}}-{\bf U}_\kappa^{n_{kj}})+\frac{\omega_p^k}{2}|{\bf n}_\kappa^{i(jk)}|({\bf U}_\kappa^{n_{ij}}-{\bf U}_\kappa^{n_{ik}}),
\end{equation}
\end{subequations}

\noindent with $d_\kappa\geq0$ the grap viscosity coefficient, so the antidiffusive fluxes in \cref{eq:AD_fluxes_DGM} read

\begin{subequations}\label{eq:AD_fluxes_DGSEM}
\begin{align}
 \frac{|\kappa|}{\Delta t^{(n)}}\big(\langle{\bf u}_{\HO}^{n+1}\rangle_\kappa&-\langle{\bf u}_{\LO}^{n+1}\rangle_\kappa\big) = \sum_{e\in\partial\kappa} {\bf A}_{\kappa,\kappa_{e}^+}^{e,n+\theta} \label{eq:AD_fluxes_DGM-a}\\
 {\bf A}_{\kappa,\kappa_{e}^+}^{e} &= \sum_{k=0}^{p}\omega_p^kJ_e^k\Big( {\bf h}\big({\bf u}_{\LO}^-({\bf x}_e^k,\cdot),{\bf u}_{\LO}^+({\bf x}_e^k,\cdot),{\bf n}_e^k\big) - {\bf h}\big({\bf u}_{\HO}^-({\bf x}_e^k,\cdot),{\bf u}_{\HO}^+({\bf x}_e^k,\cdot),{\bf n}_e^k\big) \Big). \label{eq:AD_fluxes_DGSEM-b} 
\end{align}
\end{subequations}

This strategy guarantees the cell-averaged solution to preserve the invariant domains. This is achieved by modifying the DOFs in the following way: first, observe that summing the following relation over $0\leq i,j\leq p$ gives \cref{eq:AD_fluxes_DGSEM}

\begin{equation*}
 \frac{M_\kappa^{n_{ij}}}{\Delta t^{(n)}}\big({\bf U}_{\kappa,\HO}^{n_{ij},n+1}-{\bf U}_{\kappa,\LO}^{n_{ij},n+1}\big) = \frac{M_\kappa^{n_{ij}}}{|\kappa|} \sum_{e\in\partial\kappa}  {\bf A}_{\kappa,\kappa_{e}^+}^{e,n+\theta},
\end{equation*}

\noindent then the DOFs of the limited solution ${\bf u}_h^{(n+1)}$ in \cref{eq:HO-LO-generic-limited-scheme} read

\begin{equation}\label{eq:IDP_limiter_DGSEM}
 \frac{|\kappa|}{\Delta t^{(n)}}\big({\bf U}_{\kappa}^{n_{ij},n+1}-{\bf U}_{\kappa,\HO}^{n_{ij},n+1}\big) = \sum_{e\in\partial\kappa} (l_{\kappa,\kappa_{e}^+}^{e}-1) {\bf A}_{\kappa,\kappa_{e}^+}^{e,n+\theta},
\end{equation}

\noindent with $l_{\kappa,\kappa_{e}^+}^{e}$ the limiting coefficients defined in \cref{sec:limiting-coeffs}. 

As proposed in \cite{MRR_BEDGSEM_23,renac_mpp_dgsem_nlsca_24}, one finally applies the linear scaling limiter introduced in \cite{zhang2010positivity,zhang_shu_10a} to impose these properties to the solution at nodal values within elements:

\begin{equation}\label{eq:pos_limiter}
 \tilde{\bf U}_\kappa^{n_{ij},n+1} = \theta_\kappa{\bf U}_\kappa^{n_{ij},n+1} + (1-\theta_\kappa)\langle{\bf u}_h^{(n+1)}\rangle_\kappa, \quad 0\leq i,j\leq p, \quad \kappa \in  \Omega_h,
\end{equation}

\noindent where there always exists $0\leq\theta_\kappa\leq1$ such that all the DOFs $\tilde{\bf U}_\kappa^{n_{ij},n+1}$ are in $\D$ since $\langle{\bf u}_h^{(n+1)}\rangle_\kappa\in\D$ after the IDP limiter \cref{eq:IDP_limiter_DGSEM}.


%
%
\section{Numerical experiments}\label{sec:num-xp}

In this section we present numerical experiments on problems in one and two space dimensions. The objective is here to illustrate the properties of the limiter introduced in \cref{sec:IDP_limiter,sec:time_discretization,sec:FV-DG_IDP_limiters} when applied to several numerical schemes in space and time. For that purpose, we consider finite volume schemes with explicit or implicit Runge-Kutta time marching in \cref{sec:num-xp_FV_RK} and DGSEM schemes with implicit Runge-Kutta time marching in \cref{sec:num-xp_DG_RK}, or DGSEM as time integration in \cref{sec:num-xp_stDG}. 

Without stated otherwise, we use the iterative limiter as described in \cref{algo:iterative_IDP_limiter} with a low tolerance $\TOL=10^{-8}$ and a maximum number of iterations $k_{max}=50$ (see \cref{sec:acc_iterative_limiter,sec:num-xp_FV_RK} for numerical tests about the convergence of the iterative limiter). The nonlinear algebraic systems resulting from the space discretizations described in \cref{sec:FV-DG_IDP_limiters} of nonlinear problems are solved by using a quasi-Newton iterative algorithm. The full description of the numerical procedure is given below for each numerical scheme. The DG schemes and they limiter were implemented in the CFD code Aghora developed at ONERA  \cite{renac_etal15}.

\subsection{Finite volume schemes with Runge-Kutta integration}\label{sec:num-xp_FV_RK}

We here consider a FV scheme for the discretization of the 1D compressible Euler equations with a polytropic ideal gas law $\p = (\gamma-1)\rho e$, where $\gamma=\tfrac{7}{5}$ is the ratio of specific heats. The simulations are performed on a 1D uniform grid within the FV framework outlined in \cref{sec:FV_schemes}. We use a Rusanov two-point flux at cell interfaces \cite{Rusanov1961}. In the HO scheme \cref{eq:HO_FV_scheme}, the left and right traces of the solution at interfaces are reconstructed using a MUSCL-type approximation based on a piecewise parabolic reconstruction \cite[Sec.~10.4]{lohner_CFDbook_2008}, which ensures third-order spatial accuracy in smooth regions. To guarantee nonlinear stability and reduce spurious oscillations near discontinuities, a superbee limiter is applied to the reconstructed slopes. Explicit time integration is performed using the three-stage strong-stability-preserving Runge--Kutta scheme (ERK3) proposed by Shu and Osher \cite{shu-osher88}. For implicit time stepping, the three-stage diagonally implicit Runge--Kutta method (DIRK33) introduced by Alexander \cite{Alexander_DIRK_77} is employed, which is strongly $S$-stable, and consequently both $A$-stable and $S$-stable. In contrast, the LO scheme \cref{eq:LO_FV_scheme} utilizes either an explicit forward Euler scheme (ERK1), or a linearized backward-Euler implicit scheme. Interface states are evaluated without spatial reconstruction, which yields a robust first-order accurate scheme. The time step is computed so as to satisfy $\Delta t^{(n)}\tfrac{\lambda(\langle{\bf u}_{h}^n\rangle_\kappa)}{\text{diam}\,\kappa}\leq\CFL$ in all cells $\kappa\in\Omega_h$ with $\lambda({\bf u})=|{\bf v}|+\sqrt{\gamma\p/\rho}$. The limited scheme \cref{eq:AD_fluxes_FVM} is applied to ensure positivity of density and internal energy as well as a minimum principle on the specific entropy (see \cref{ex:Euler_eq}).

We consider the following Riemann problem with strong waves proposed in \cite{toro_book}:

\begin{equation}\label{eq:RP_Toro_pb3}
 {\bf u}(x,0)=\left\{ \begin{array}{ll} \begin{pmatrix} 1 & 0 & 2.5\times10^{3}\end{pmatrix}^\top & \text{if } x>0,  \\
    \begin{pmatrix} 1 & 0 & 2.5\times10^{-2}\end{pmatrix}^\top  & \text{if } x<0.
 \end{array} \right.
\end{equation}

Even with the superbee limiter, the HO scheme fails to complete the computation up to the final time due to the presence of a strong shock, which induces spurious oscillations and negative pressure values. In contrast, the limited scheme resolves all waves sharply when the time step is of the order of the wave speeds ($\CFL=1$), for both explicit and implicit time integration, as observed in \cref{fig:toro_pb3_FV_ERK,fig:toro_pb3_FV_DIRK_cfl1.0}. The scheme maintains robustness at larger time steps ($\CFL=5$), but results in reduced resolution and increased spurious oscillations, as highlighted in \cref{fig:toro_pb3_FV_DIRK_cfl5.0}.

\begin{figure}
\centering
\captionsetup[subfigure]{labelformat=empty}
\subfloat{\begin{picture}(0,0) \put(-15,30){\rotatebox{90}{LO scheme}} \end{picture}}
\subfloat{\includegraphics[width=4.5cm]{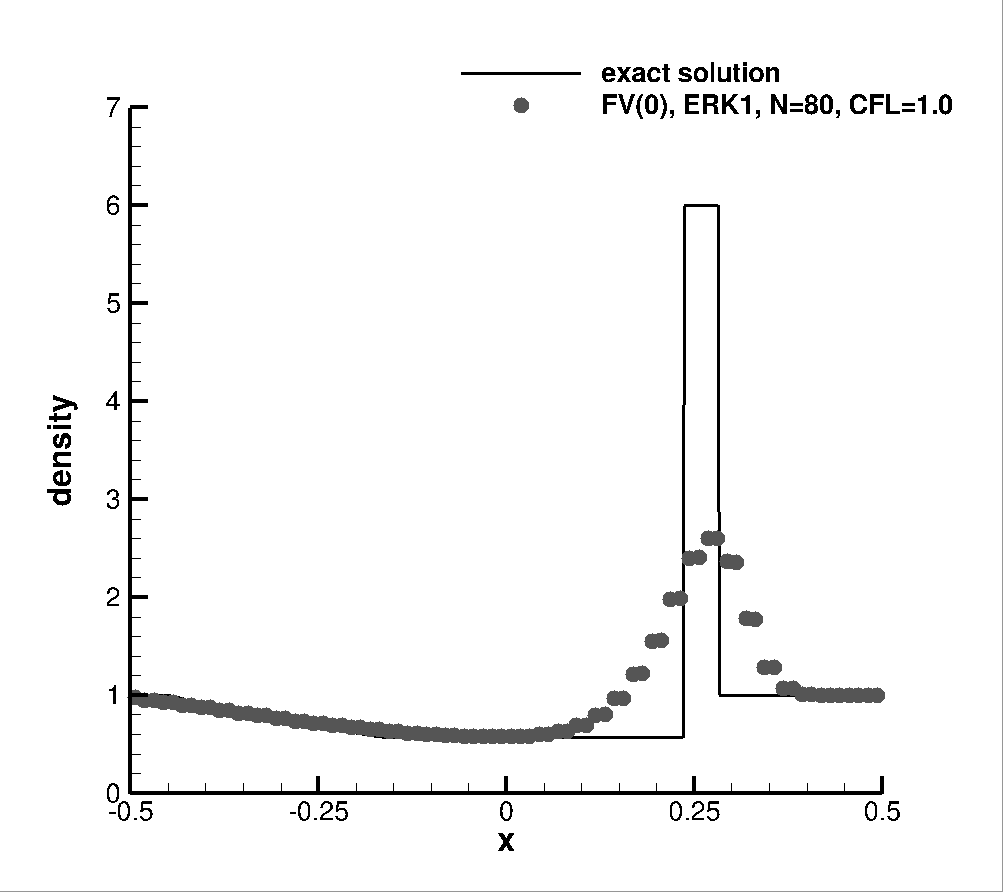}}
\subfloat{\includegraphics[width=4.5cm]{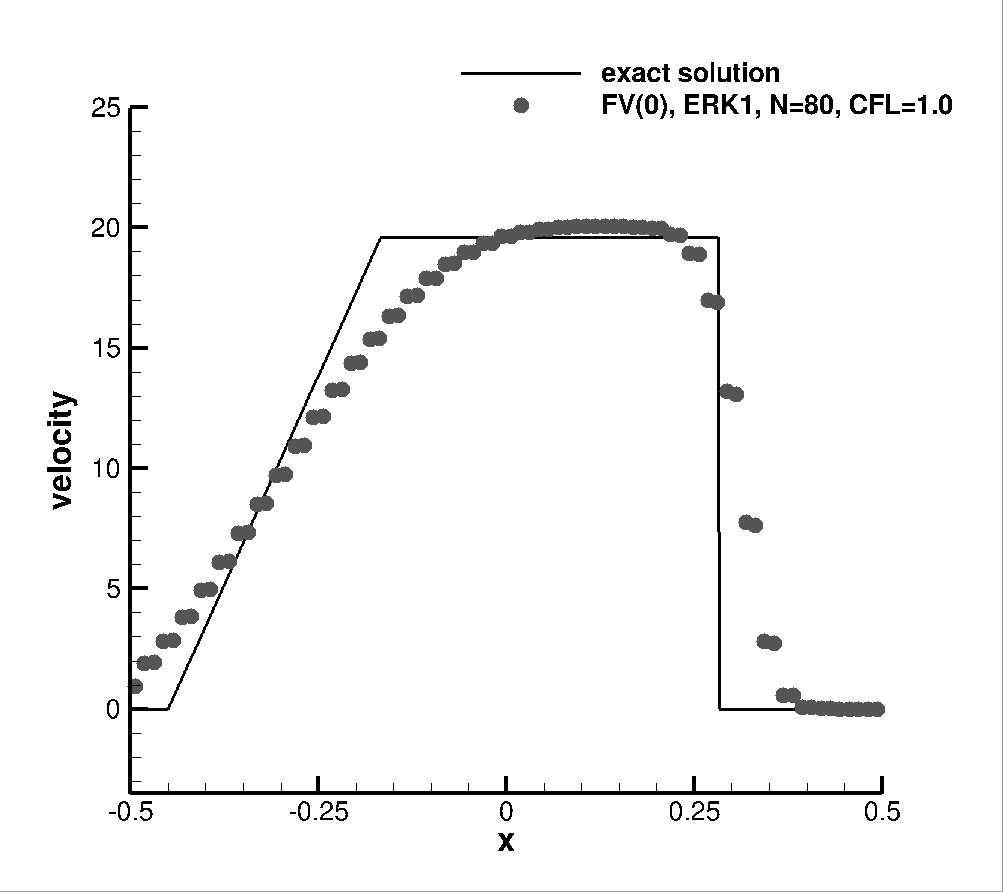}}
\subfloat{\includegraphics[width=4.5cm]{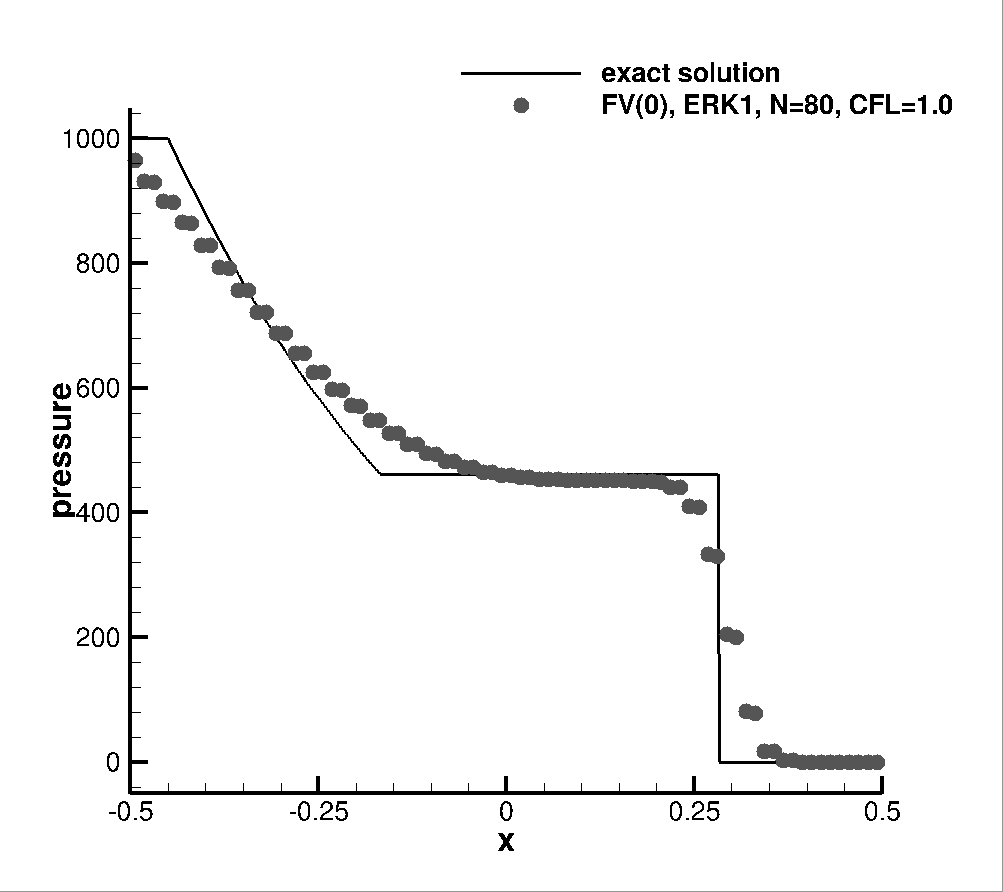}} \\
\subfloat{\begin{picture}(0,0) \put(-15,20){\rotatebox{90}{limited scheme}} \end{picture}}
\subfloat[density]{\includegraphics[width=4.5cm]{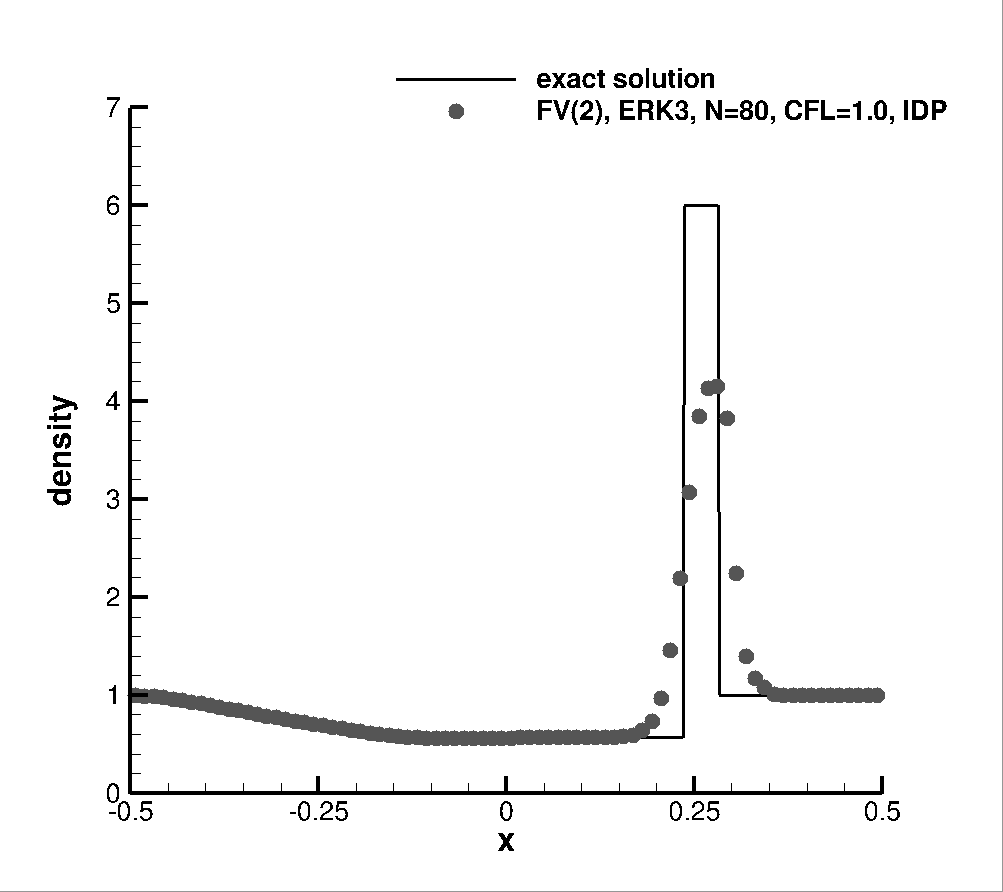}}
\subfloat[velocity]{\includegraphics[width=4.5cm]{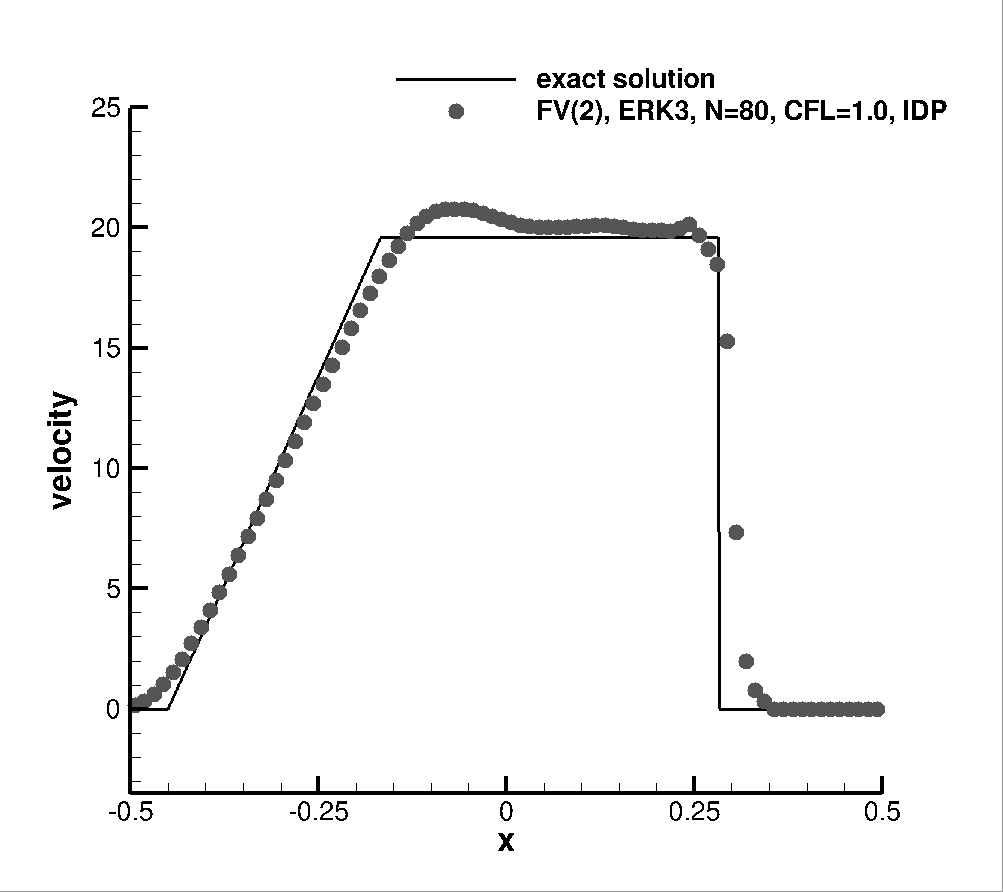}}
\subfloat[pressure]{\includegraphics[width=4.5cm]{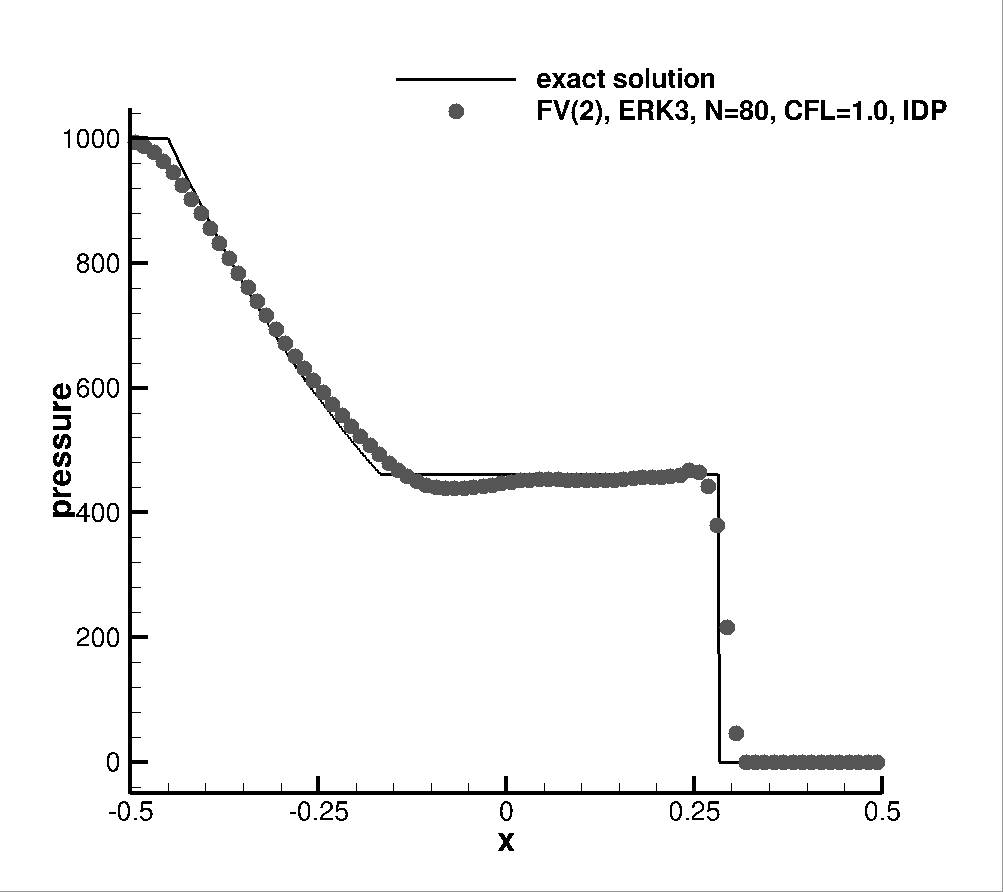}}
\caption{Riemann problem with initial data \cref{eq:RP_Toro_pb3}: explicit LO FV(0)-ERK1 (top) and limited FV(2)-ERK3 (bottom) solutions at $t=0.012$ on a uniform mesh with $N=80$ cells and $\CFL=1$. The DOFs (bullets) are compared to the exact solution (lines).}
\label{fig:toro_pb3_FV_ERK}
\end{figure}
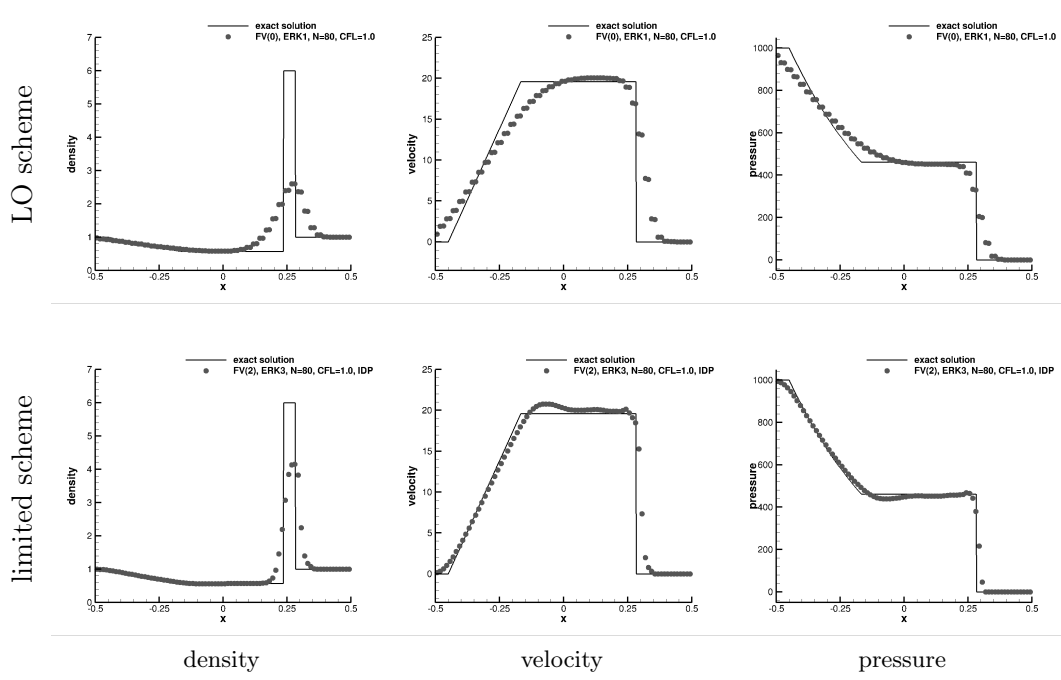 

\begin{figure}
\centering
\captionsetup[subfigure]{labelformat=empty}
\subfloat{\begin{picture}(0,0) \put(-15,30){\rotatebox{90}{LO scheme}} \end{picture}}
\subfloat{\includegraphics[width=4.5cm]{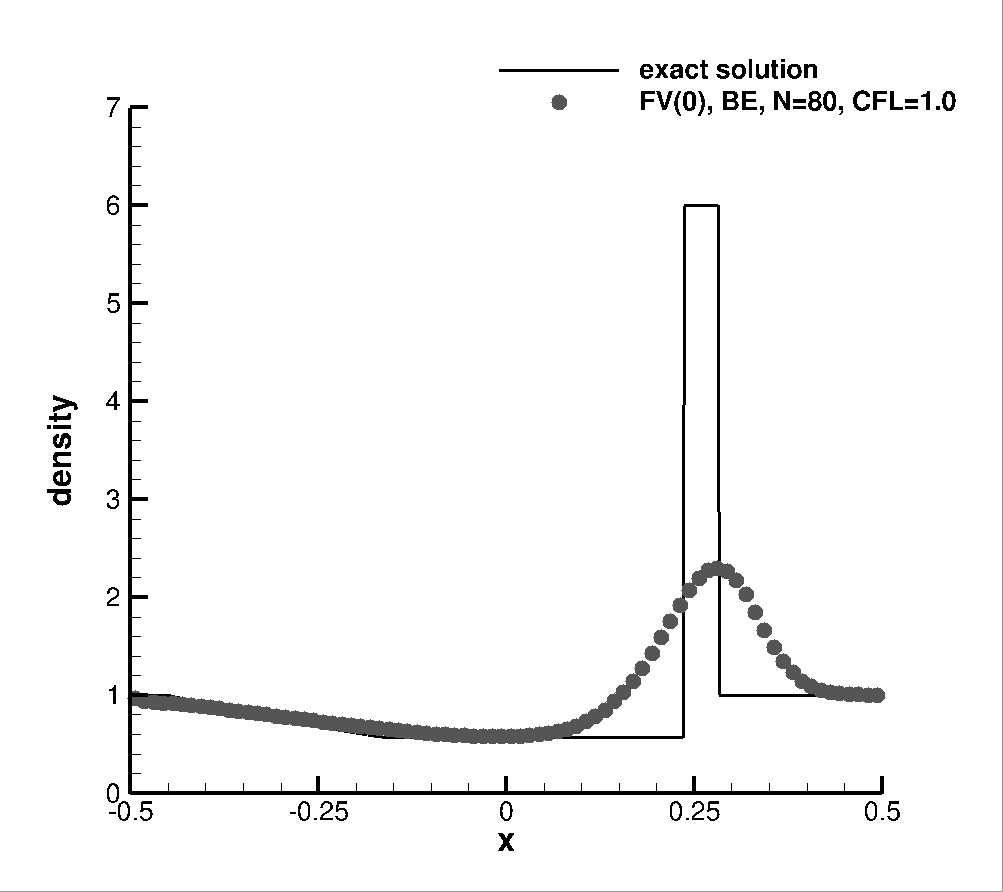}}
\subfloat{\includegraphics[width=4.5cm]{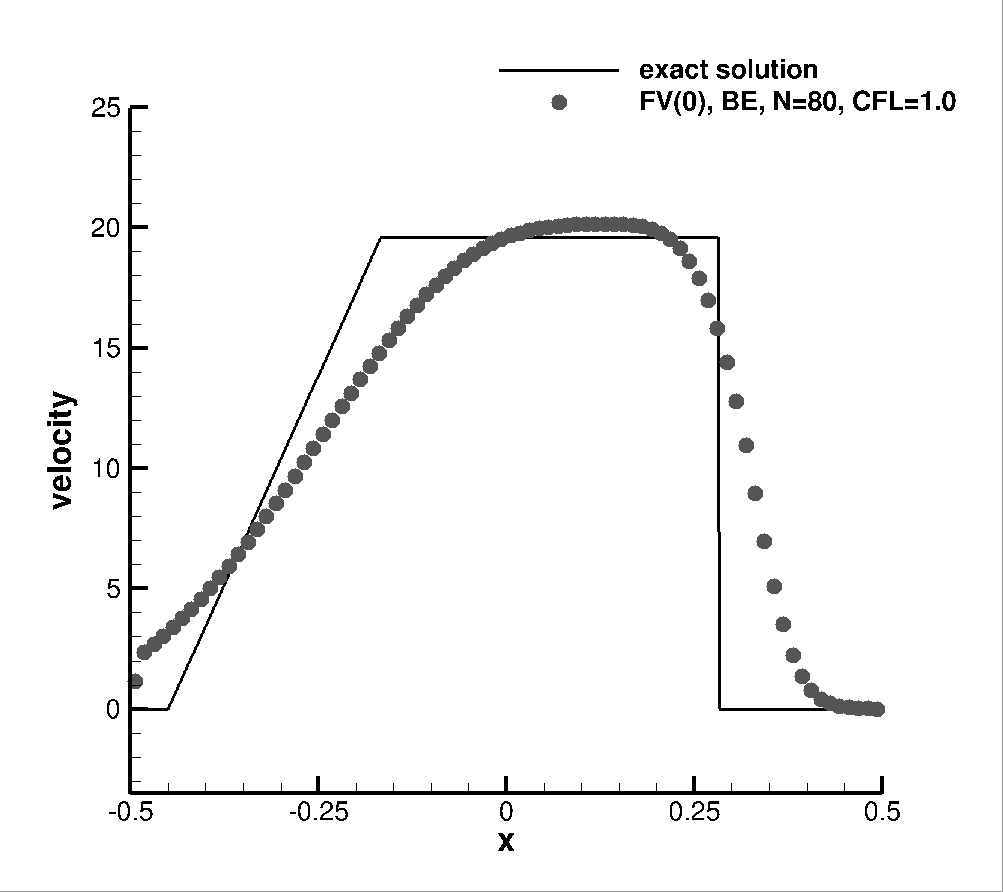}}
\subfloat{\includegraphics[width=4.5cm]{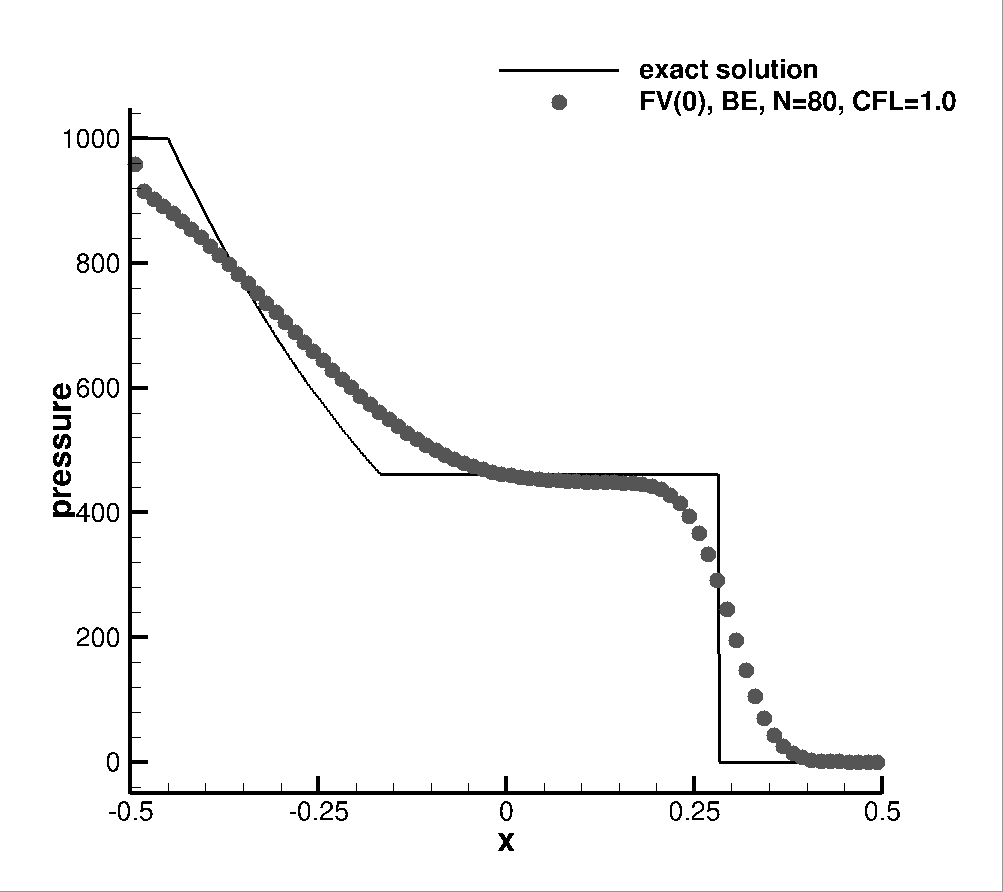}} \\
\subfloat{\begin{picture}(0,0) \put(-15,20){\rotatebox{90}{limited scheme}} \end{picture}}
\subfloat[density]{\includegraphics[width=4.5cm]{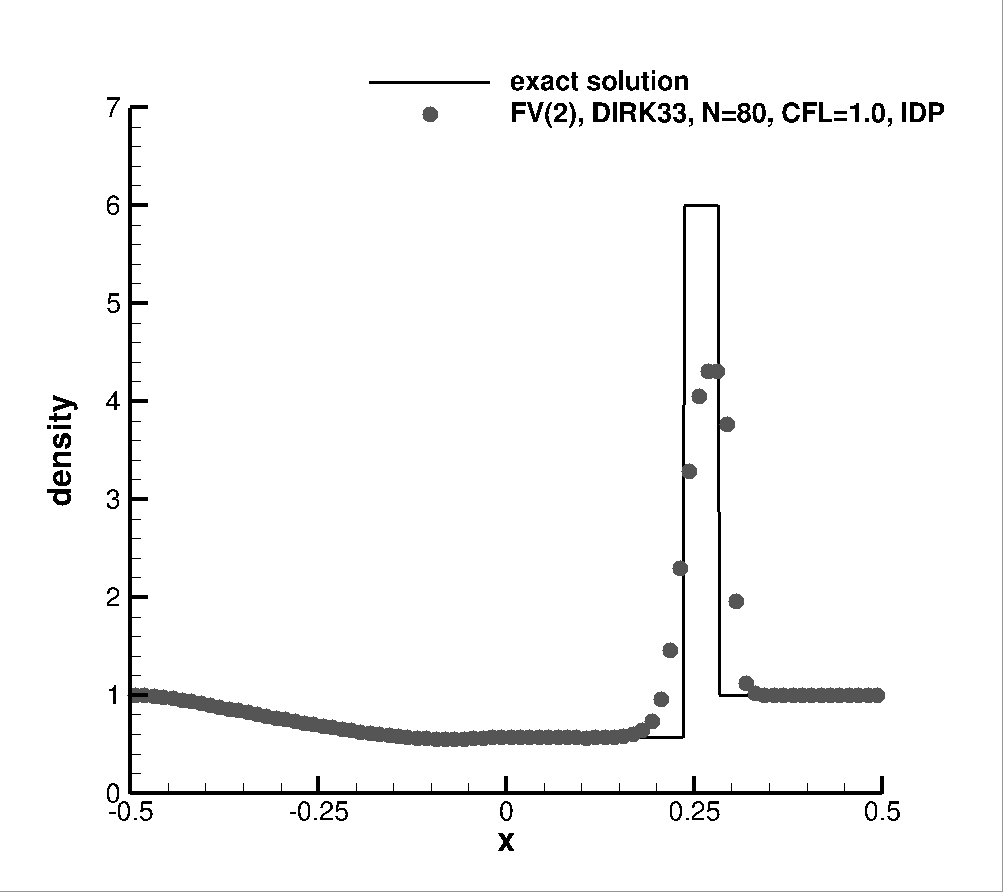}}
\subfloat[velocity]{\includegraphics[width=4.5cm]{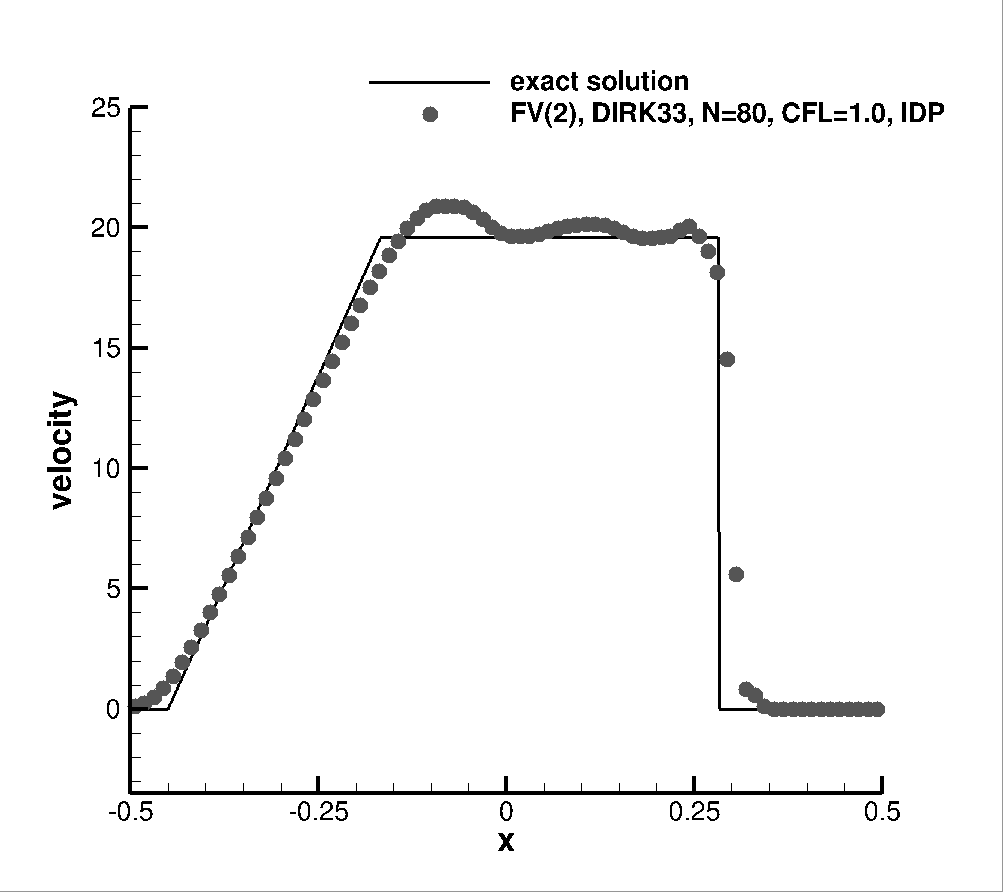}}
\subfloat[pressure]{\includegraphics[width=4.5cm]{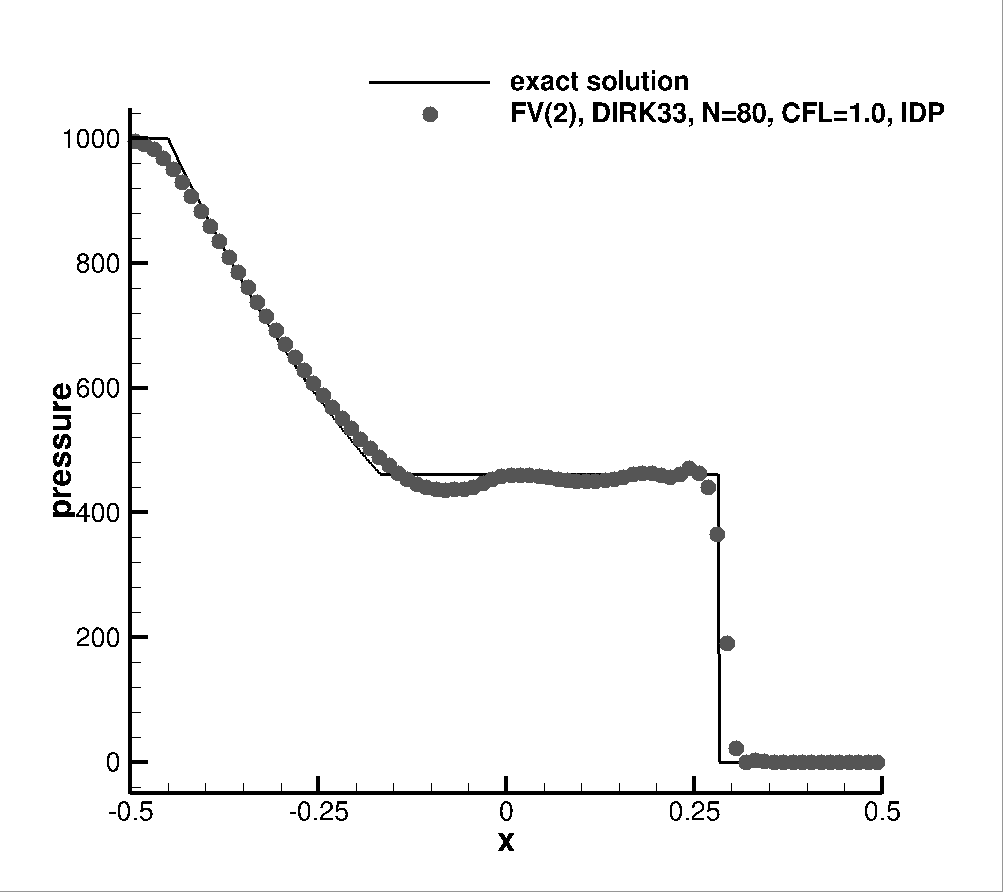}}
\caption{Riemann problem with initial data \cref{eq:RP_Toro_pb3}: implicit LO FV(0)-BE (top) and limited FV(2)-DIRK33 (bottom) solutions at $t=0.012$ on a uniform mesh with $N=80$ cells and $\CFL=1$. The DOFs (bullets) are compared to the exact solution (lines).}
\label{fig:toro_pb3_FV_DIRK_cfl1.0}
\end{figure}

\begin{figure}
\centering
\captionsetup[subfigure]{labelformat=empty}
\subfloat{\begin{picture}(0,0) \put(-15,30){\rotatebox{90}{LO scheme}} \end{picture}}
\subfloat{\includegraphics[width=4.5cm]{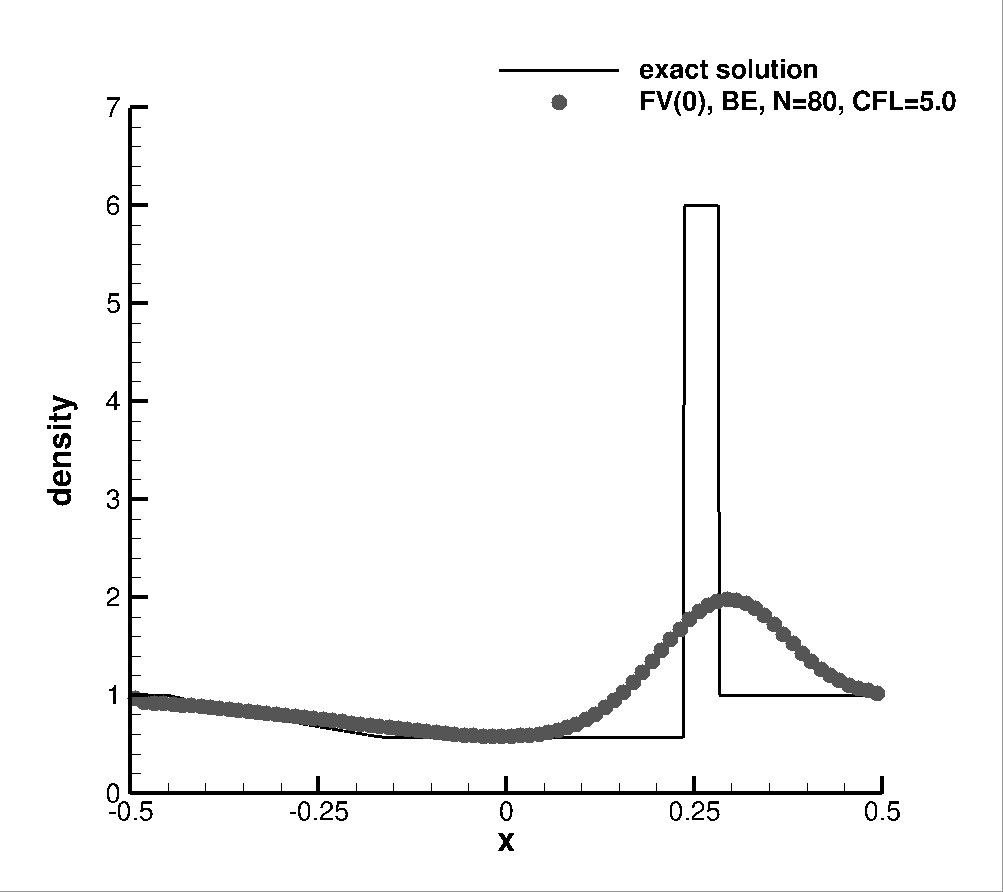}}
\subfloat{\includegraphics[width=4.5cm]{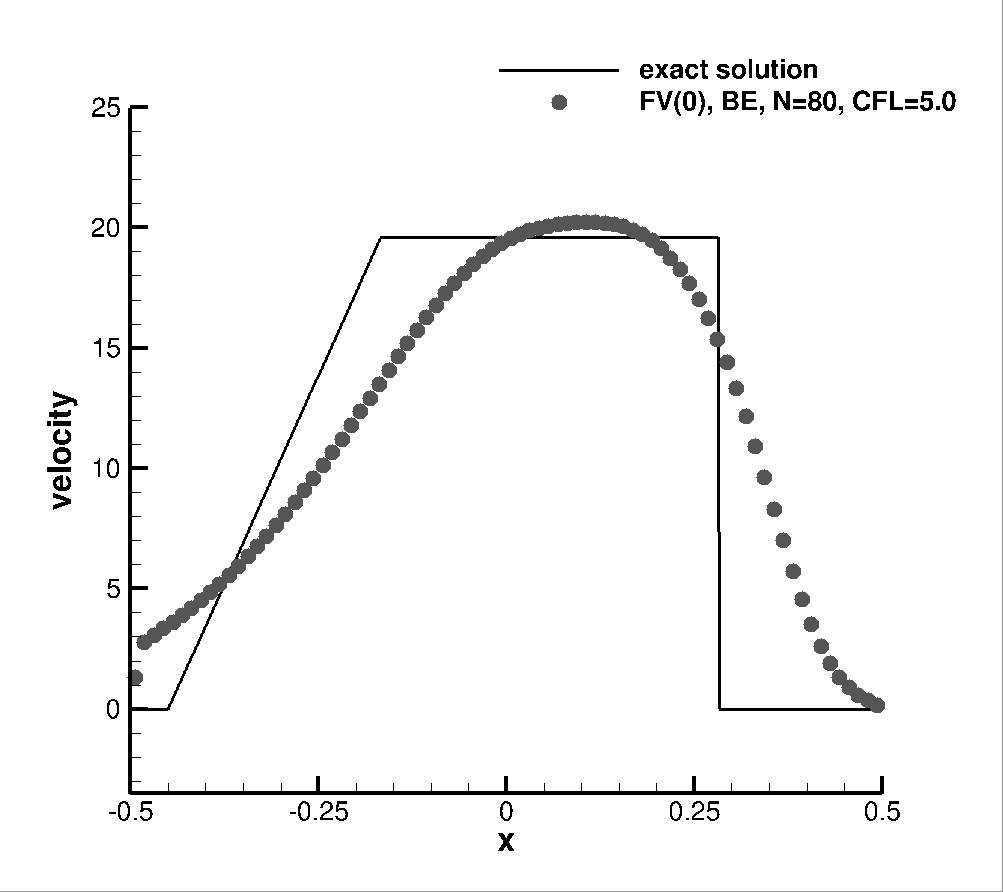}}
\subfloat{\includegraphics[width=4.5cm]{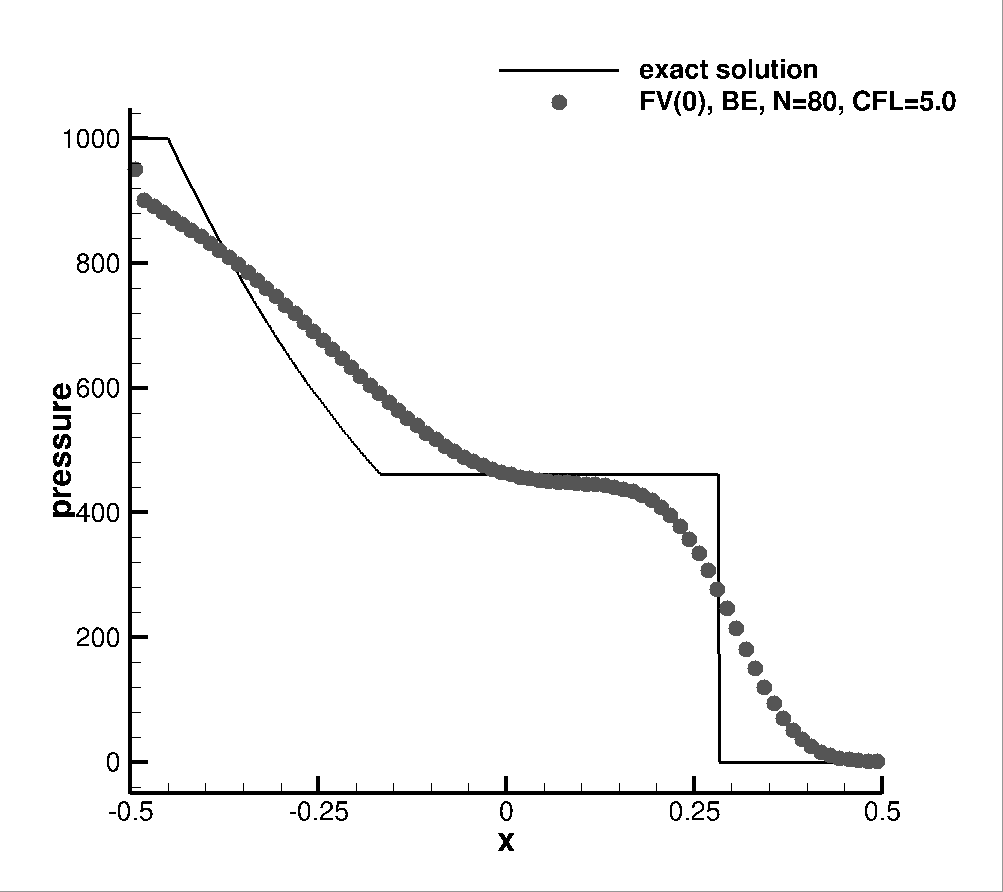}} \\
\subfloat{\begin{picture}(0,0) \put(-15,20){\rotatebox{90}{limited scheme}} \end{picture}}
\subfloat[density]{\includegraphics[width=4.5cm]{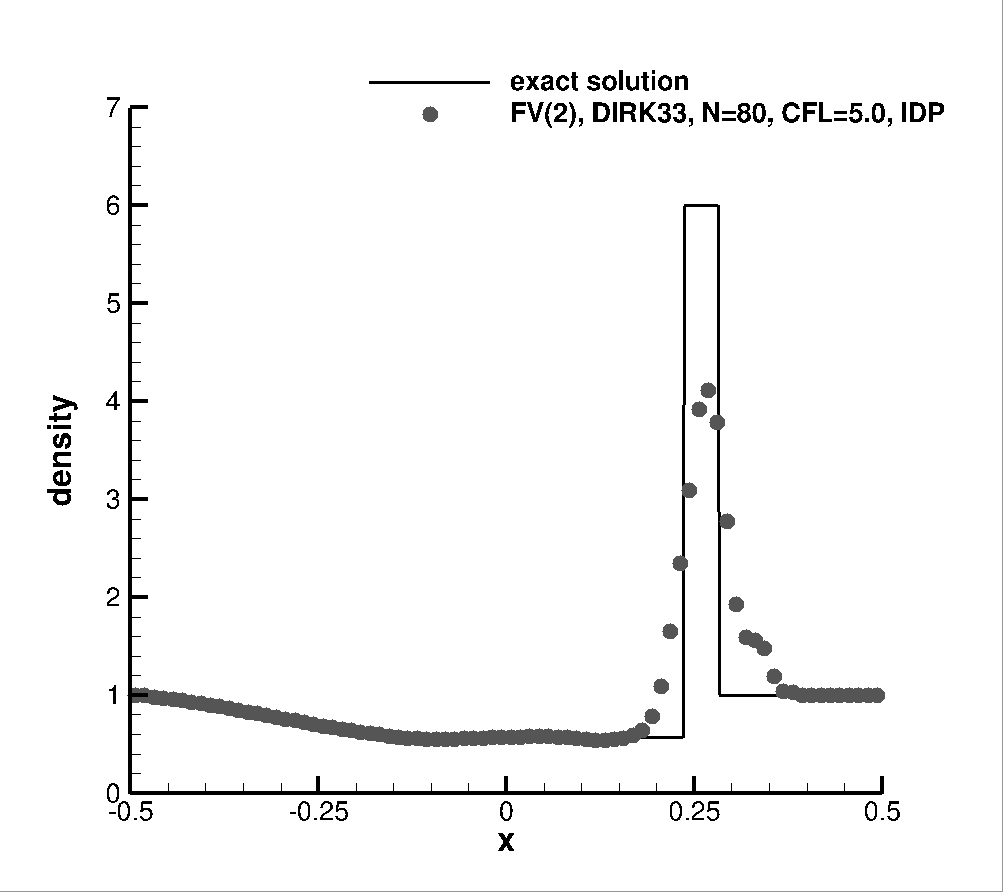}}
\subfloat[velocity]{\includegraphics[width=4.5cm]{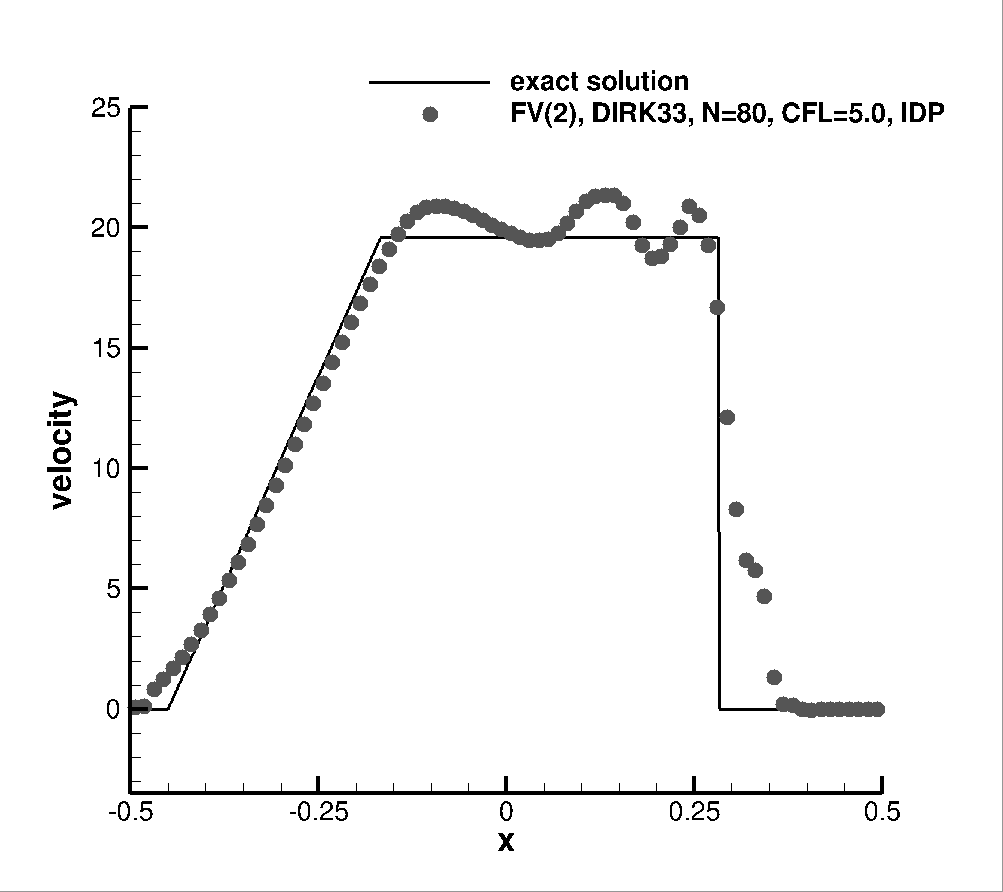}}
\subfloat[pressure]{\includegraphics[width=4.5cm]{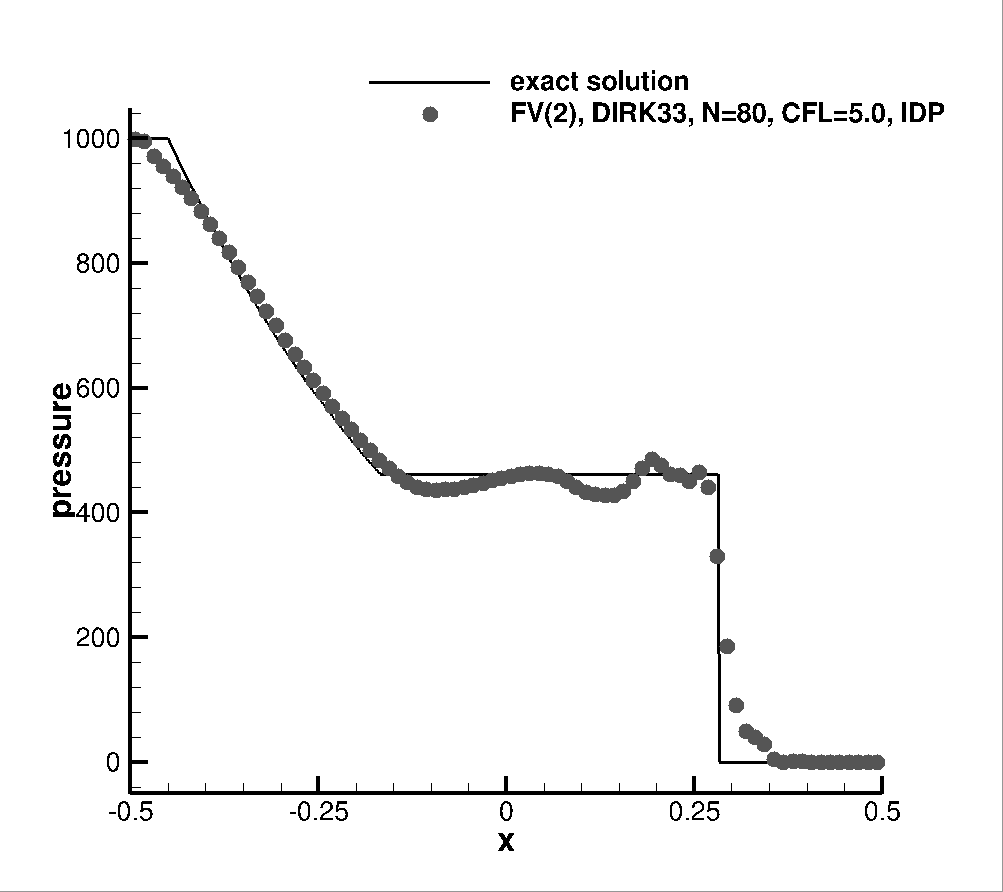}}
\caption{Riemann problem with initial data \cref{eq:RP_Toro_pb3}: implicit LO FV(0)-BE (top) and limited FV(2)-DIRK33 (bottom) solutions at $t=0.012$ on a uniform mesh with $N=80$ cells and $\CFL=5$. The DOFs (bullets) are compared to the exact solution (lines).}
\label{fig:toro_pb3_FV_DIRK_cfl5.0}
\end{figure} 

The effects of the iterative limiter described in \cref{algo:iterative_IDP_limiter} (see \cref{sec:acc_iterative_limiter}) and the acceleration factor $\beta$ (see \cref{sec:iterative_limiter}) are shown in \cref{fig:toro_pb3_FV_ERK_beta12}. With $\beta=1$, the limited solution converges slowly toward the HO solution. Increasing the maximum number of limiter iterations brings the limited solution closer to the unlimited HO solution, while preserving invariant domains. Setting $\beta=2$ strongly accelerates the convergence toward the HO solution; after the first limiter iteration, the limited solution is identical to the tenth iteration obtained with $\beta=1$.

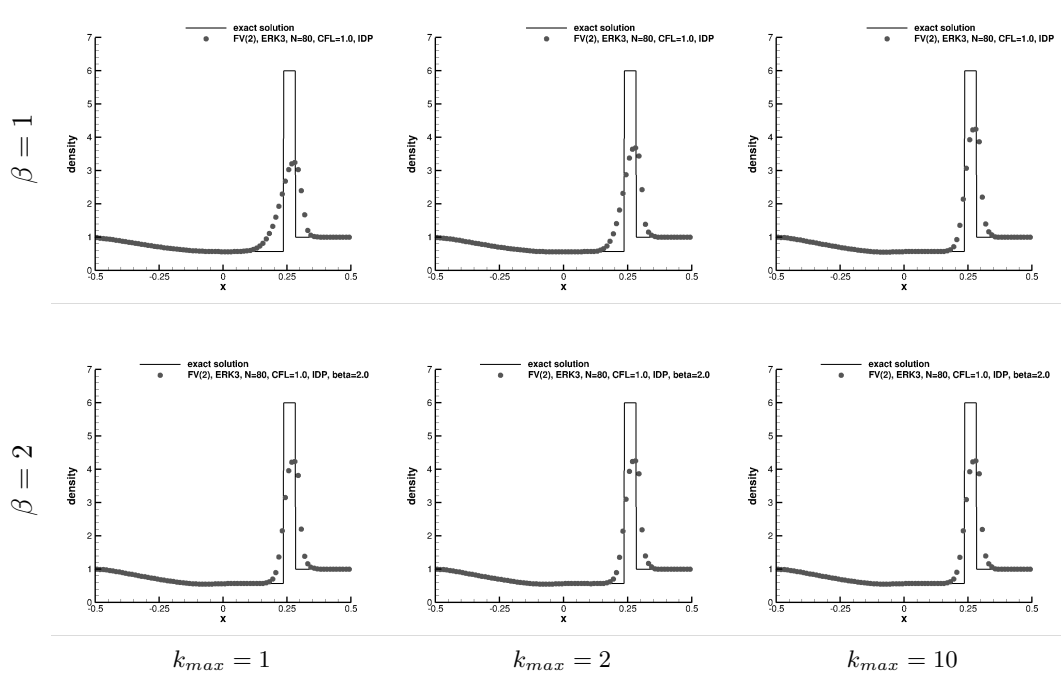
\begin{figure}
\centering
\captionsetup[subfigure]{labelformat=empty}
\subfloat{\begin{picture}(0,0) \put(-15,45){\rotatebox{90}{$\beta=1$}} \end{picture}}
\subfloat{\includegraphics[width=4.5cm]{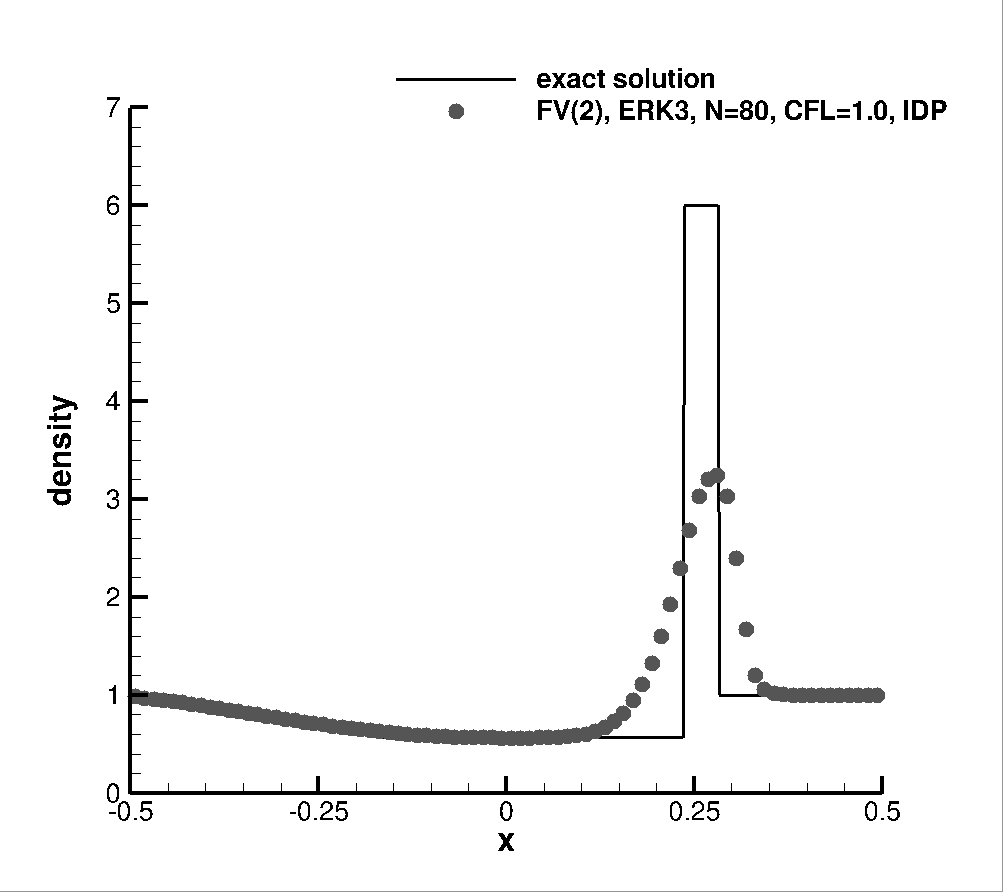}}
\subfloat{\includegraphics[width=4.5cm]{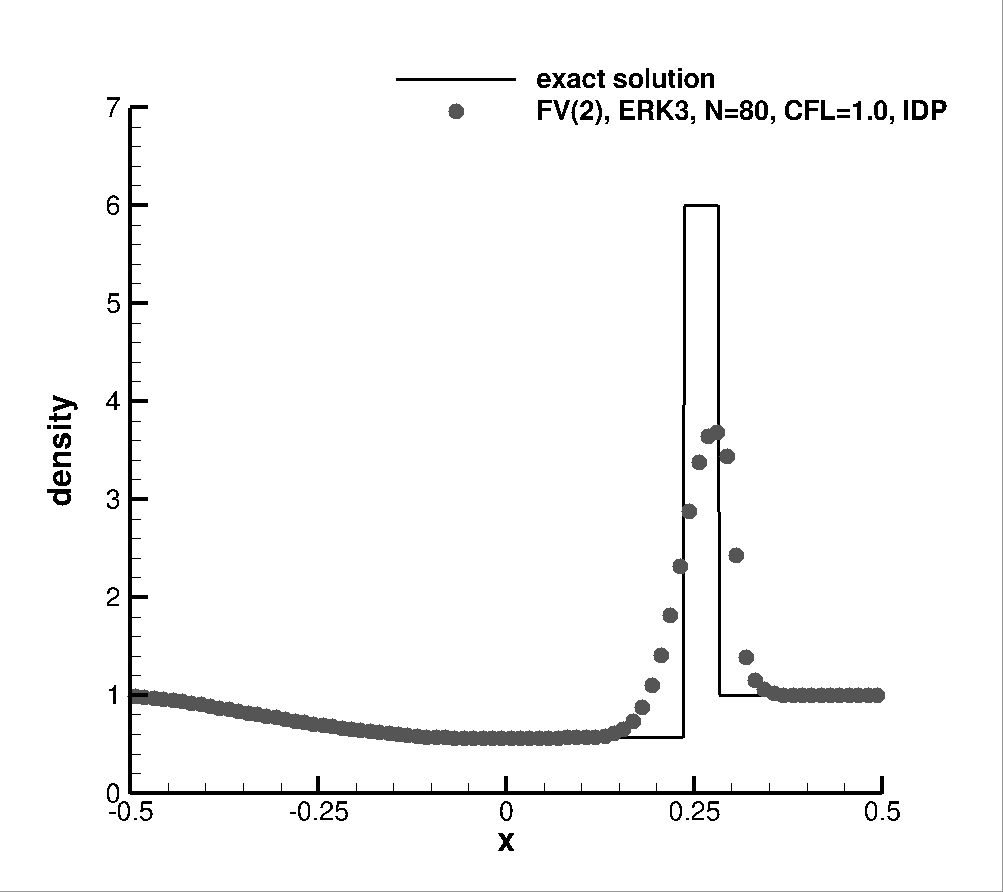}}
\subfloat{\includegraphics[width=4.5cm]{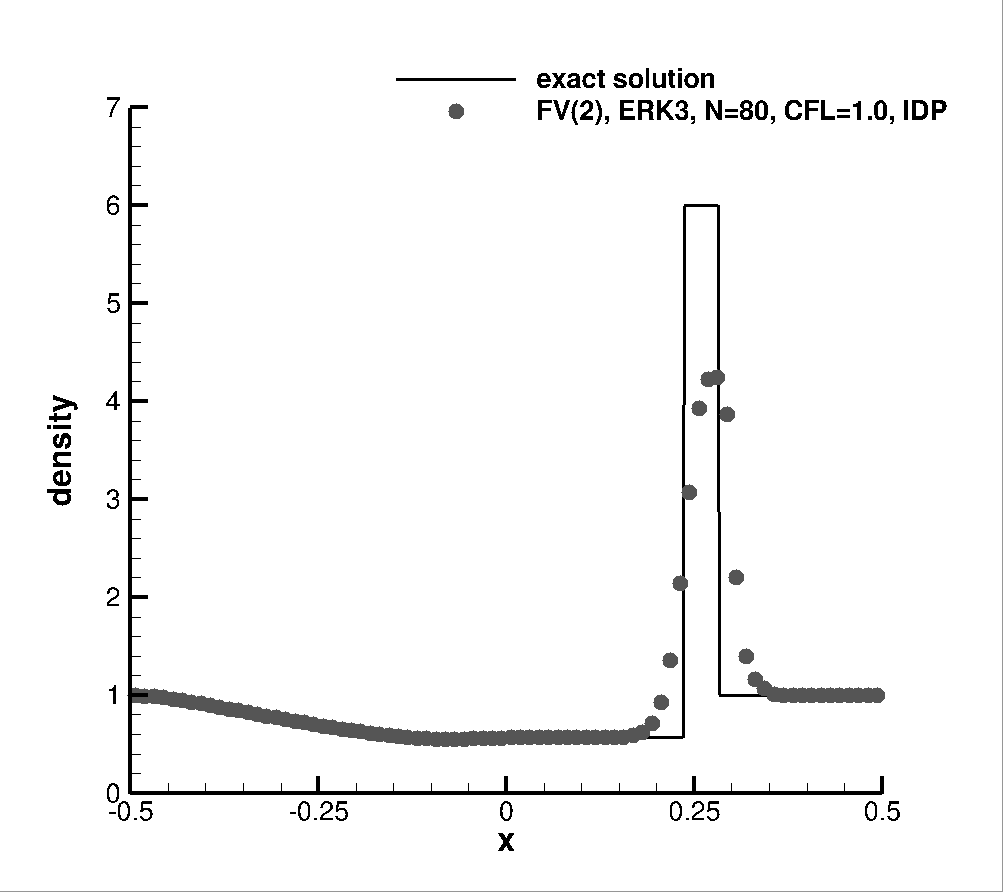}} \\
\subfloat{\begin{picture}(0,0) \put(-15,45){\rotatebox{90}{$\beta=2$}} \end{picture}}
\subfloat[$k_{max}=1$]{\includegraphics[width=4.5cm]{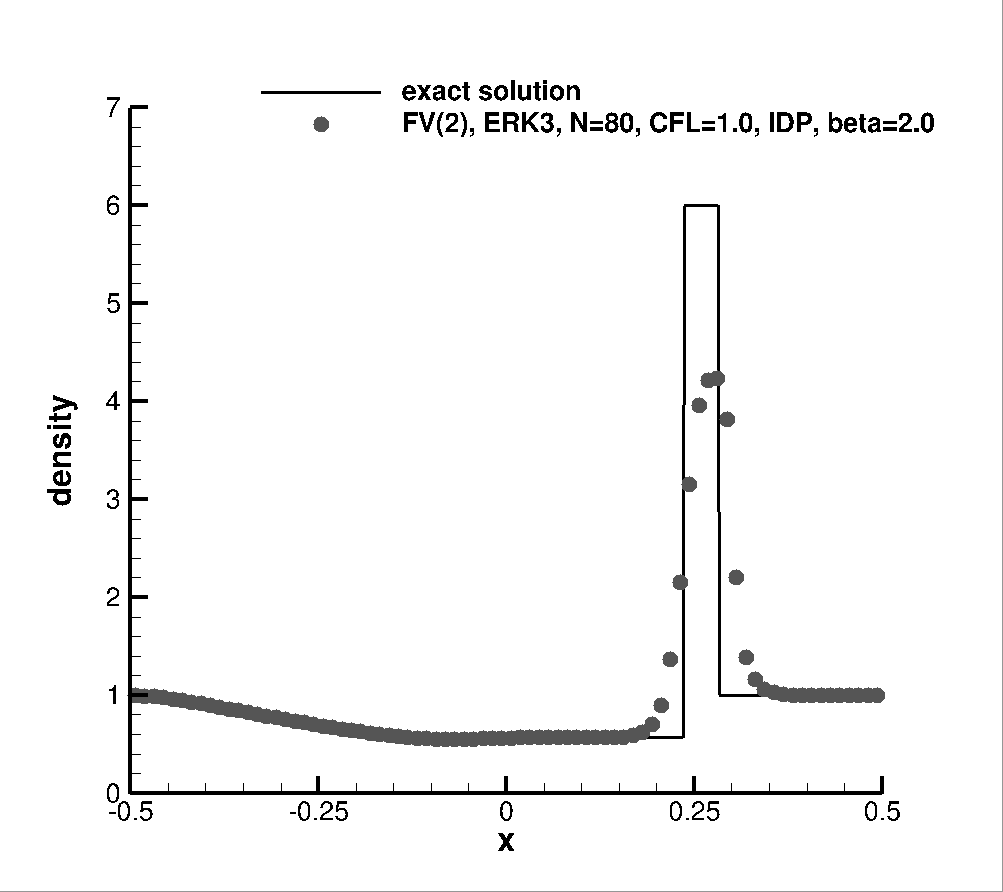}}
\subfloat[$k_{max}=2$]{\includegraphics[width=4.5cm]{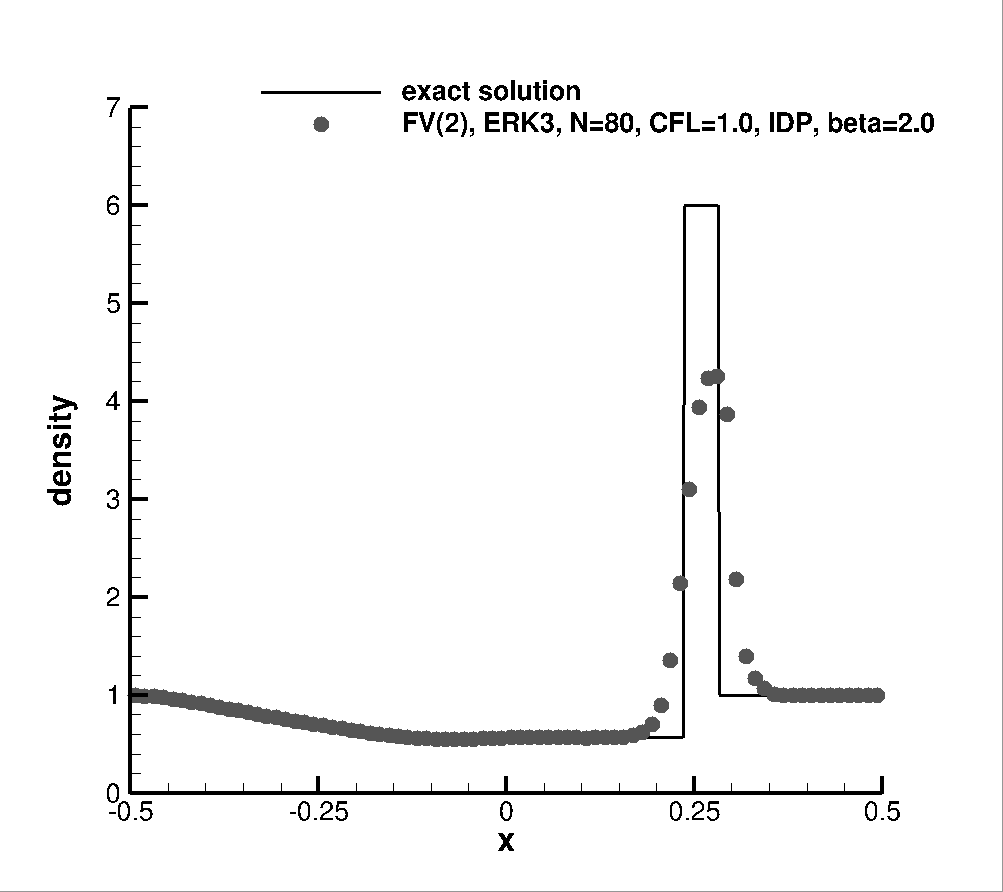}}
\subfloat[$k_{max}=10$]{\includegraphics[width=4.5cm]{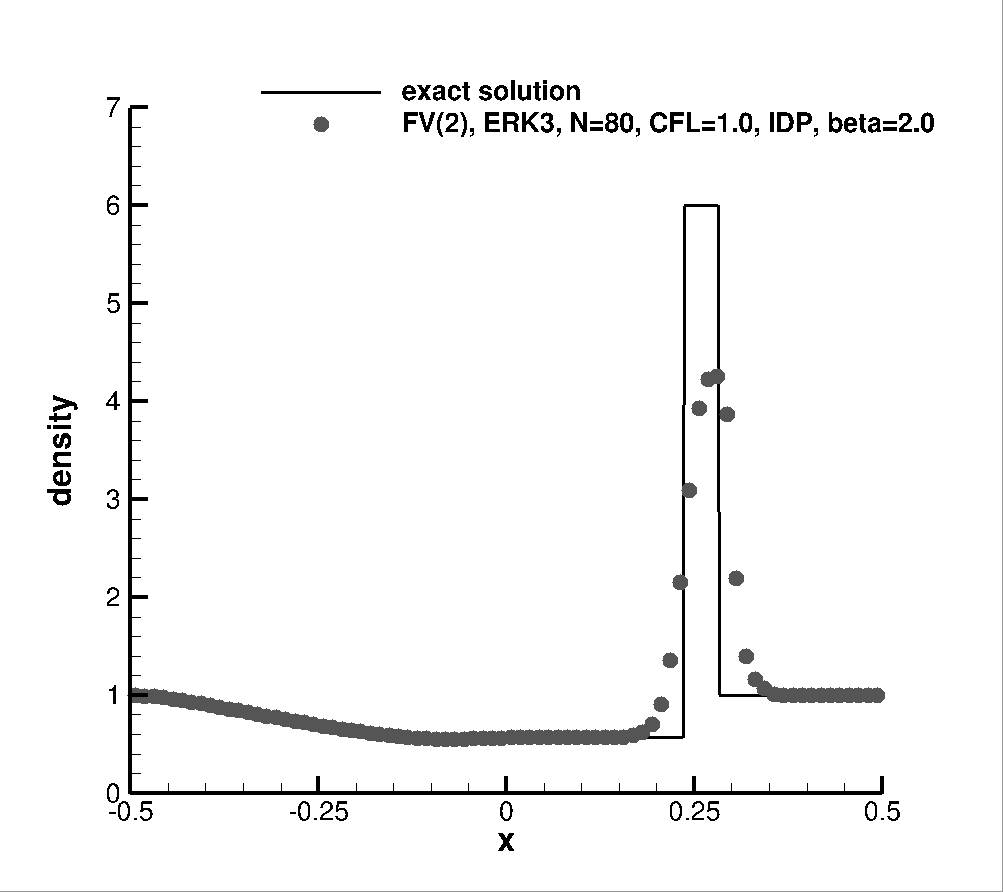}}
\caption{Riemann problem with initial data \cref{eq:RP_Toro_pb3}: The top and bottom rows correspond to the limited FV(2)-ERK3 scheme with acceleration factors $\beta=1$ and $\beta=2$, respectively (see \cref{sec:iterative_limiter}). Computations are performed at $t=0.012$ on a uniform mesh with $N=80$ cells, $\CFL=1.0$, and different iterations of the iterative limiter $k_{max}$ in \cref{algo:iterative_IDP_limiter}. The DOFs values (bullets) are compared to the exact solution (lines).}
\label{fig:toro_pb3_FV_ERK_beta12}
\end{figure} 


\color{black}

%
\subsection{Discontinuous Galerkin schemes with Runge-Kutta integration}\label{sec:num-xp_DG_RK}

Here we consider the fourth-order DGSEM \cref{eqn:fully-discr_DG,eq:semi-discr_DGSEM-res} with $p=3$ as HO scheme for the discretization of the 1D and 2D compressible Euler equations with a polytropic ideal gas law as in the previous section. We use the DGSEM scheme with graph viscosity \cref{eqn:fully-discr_LO_DG} as LO scheme, while the limited scheme is \cref{eq:IDP_limiter_DGSEM,eq:pos_limiter} and is applied to ensure positivity of density and internal energy. We use the Chandrashekar's EC fluxes \cite{chandrashekar13} within cells and a Rusanov flux with local upwinding at interfaces in \cref{eq:semi-discr_DGSEM-res}. The scheme is integrated in time by using the strongly $S$-stable three-stage DIRK33 method from Alexander \cite{Alexander_DIRK_77}. The graph viscosity coefficient in \cref{eqn:fully-discr_LO_DG} is defined from the analyses in \cite[Lemma~4.1]{MRR_BEDGSEM_23} and \cite[Th.~2.2]{renac_mpp_dgsem_nlsca_24} in the scalar case. We use  $d_\kappa=2\max_{k,l}\frac{D_{kl}}{\omega_p^l}\sup_{{\bf x}\in\Omega_h}\lambda({\bf u}_0({\bf x}))$ with $2\max_{k,l}\frac{D_{kl}}{\omega_p^l}\simeq9.7$ for $p=3$ and $\lambda({\bf u})=|{\bf v}|+\sqrt{\gamma\p/\rho}$, which was seen to successfully stabilize and keep robustness of the computations. The time step is computed so as to satisfy $\Delta t^{(n)}\tfrac{\lambda(\langle{\bf u}_{h}^n\rangle_\kappa)}{\text{diam}\,\kappa}\leq\CFL$ in all cells $\kappa\in\Omega_h$.

We first consider the problem of a 2D Kelvin-Helmholtz instability \cite{chan2022entropy} in a square domain $[-1,1]^2$ with periodic conditions and initial data as $\rho_0({\bf x})=\tfrac{1}{2}+\tfrac{3}{4}B({\bf x})$, $u_0({\bf x})=\tfrac{1}{2}(B({\bf x})-1)$, $v_0({\bf x})=\tfrac{1}{10}\sin(2\pi x)$ and $\p_0({\bf x})=1$ where $B({\bf x})=\tanh(15y+ 7.5)-\tanh(15y-7.5)$. Although there are strong variations in
the initial density and velocity fields that usually challenge the robustness and stability of the methods \cite{chan2022entropy,carlier_renac_IDP_22}, the present time implicit HO scheme already preserves positivity even at large time steps. We here rather evaluate the capability of the limiter to preserve the resolution of the HO scheme in presence of a wide range of scales in the solution. \Cref{fig:KHI} displays snapshots of density fields obtained on a Cartesian mesh. The large diffusion of the LO scheme damps all the small scales, but the limiter allows to recover the high resolution of the HO scheme, with the small differences attributable to the sensitivity of this test case to small perturbations.

\begin{figure}
\centering
\captionsetup[subfigure]{labelformat=empty}
\subfloat[HO scheme]{\includegraphics[width=4.5cm]{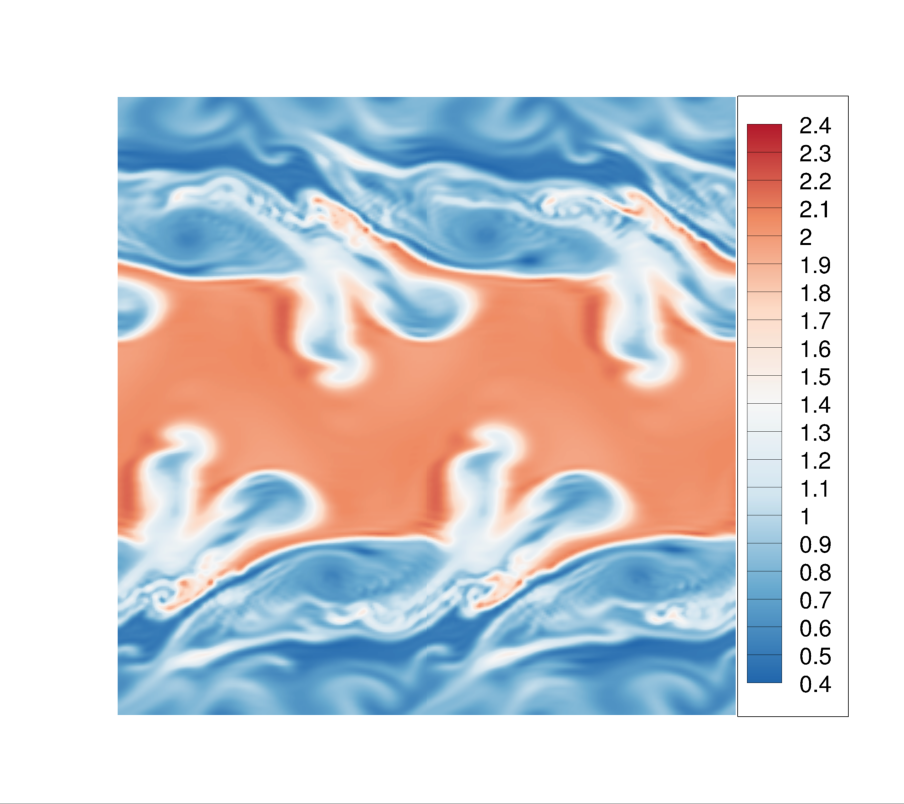}}
\subfloat[LO scheme]{\includegraphics[width=4.5cm]{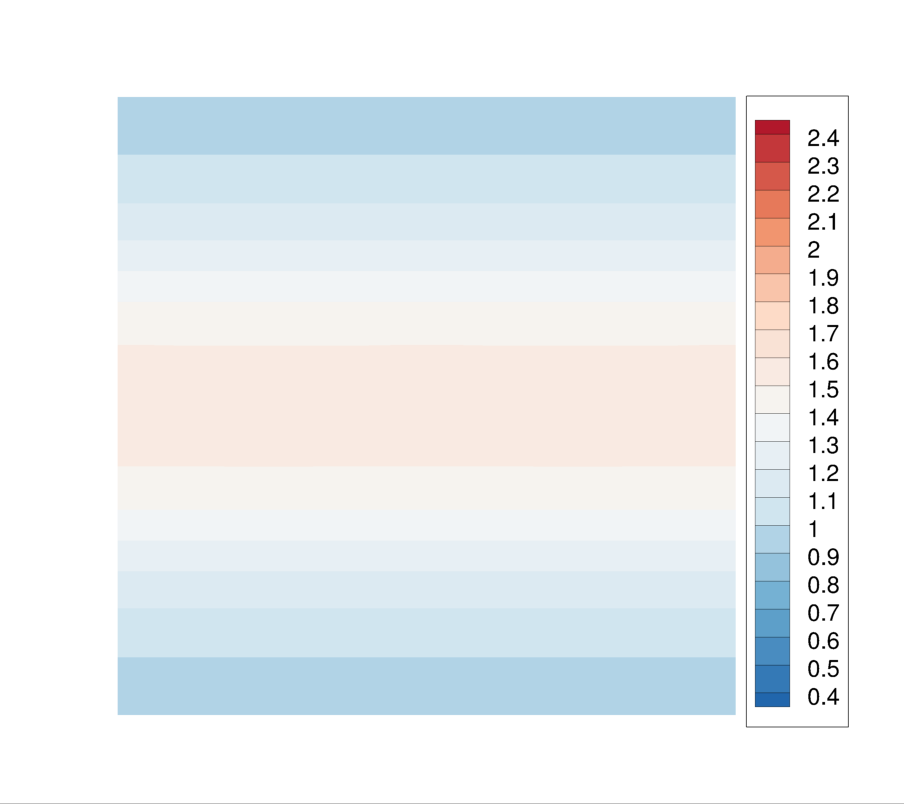}}
\subfloat[limited scheme]{\includegraphics[width=4.5cm]{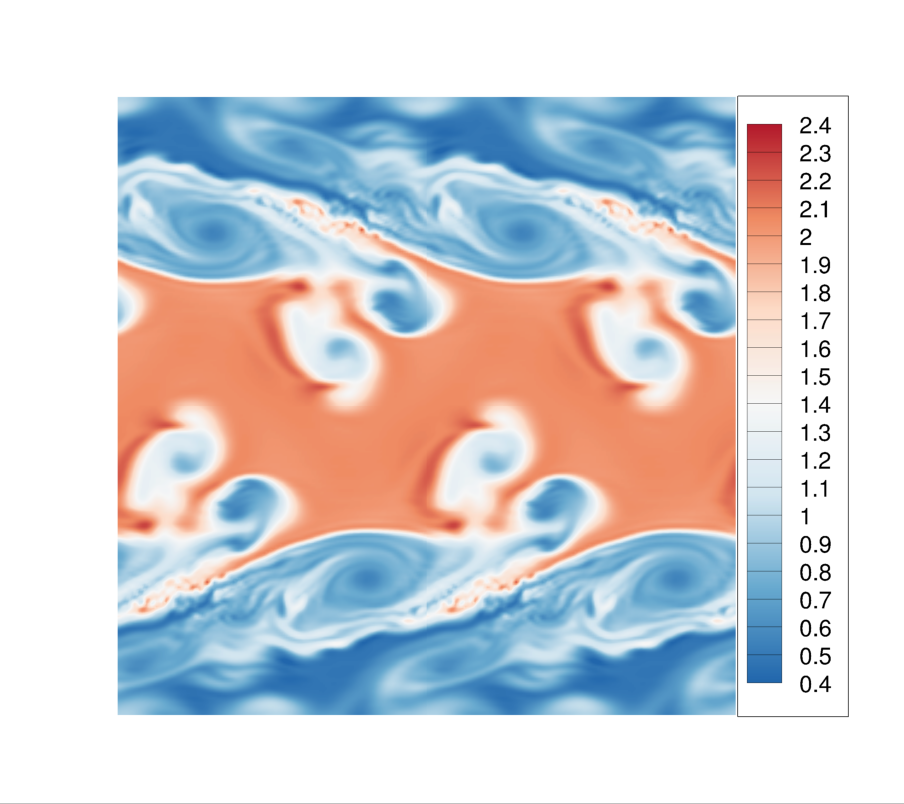}}
\caption{Two-dimensional Kelvin--Helmholtz instability: density field at $t = 10$ obtained with the DGSEM($3$) and DIRK33 schemes on a mesh with $N=64\times64$ elements and a large time step with $\CFL=10$.}
\label{fig:KHI}
\end{figure}

We now consider once again the Riemann problem with strong waves \cref{eq:RP_Toro_pb3} from \cite{toro_book}, which we solve with the DGSEM(3) scheme in space on a uniform mesh with $N=80$ cells and the DIRK33 method in time and two different time steps corresponding to CFL values $\CFL=1$ and $\CFL=5$. The HO scheme without limiter is unable to run the computation up to the final time due to the presence of the strong shock. As observed in \cref{fig:toro_pb3_DGSEM_DIRK}, the limited scheme resolves sharply all the waves when the time step is of the order of the wave speeds ($\CFL=1$), while it keeps robustness of the computation at larger time step ($\CFL=5$) but results in lower resolution and larger spurious oscillations.

\begin{figure}
\centering
\captionsetup[subfigure]{labelformat=empty}
\subfloat{\begin{picture}(0,0) \put(-15,45){\rotatebox{90}{$\CFL=1$}} \end{picture}}
\subfloat{\includegraphics[width=4.5cm]{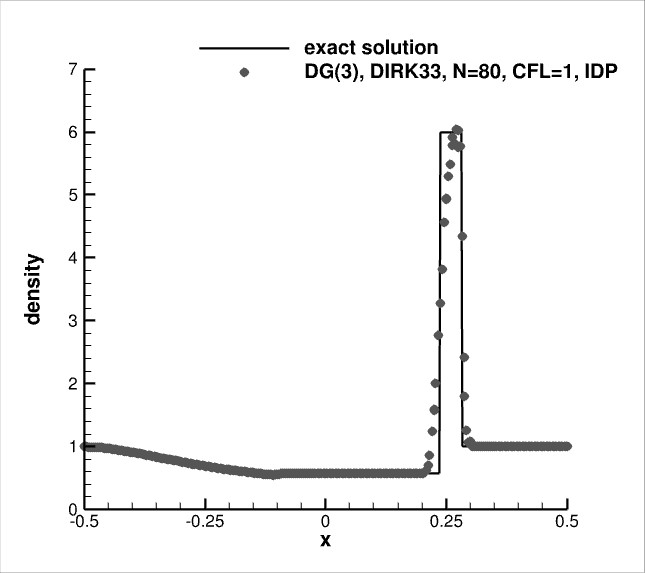}}
\subfloat{\includegraphics[width=4.5cm]{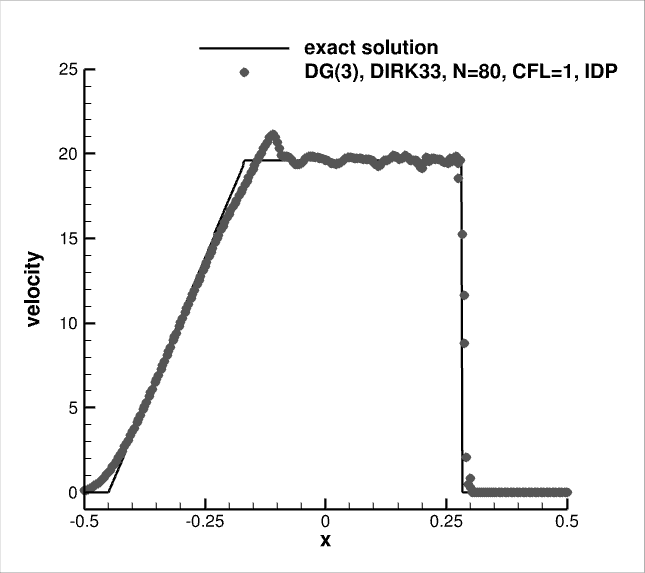}}
\subfloat{\includegraphics[width=4.5cm]{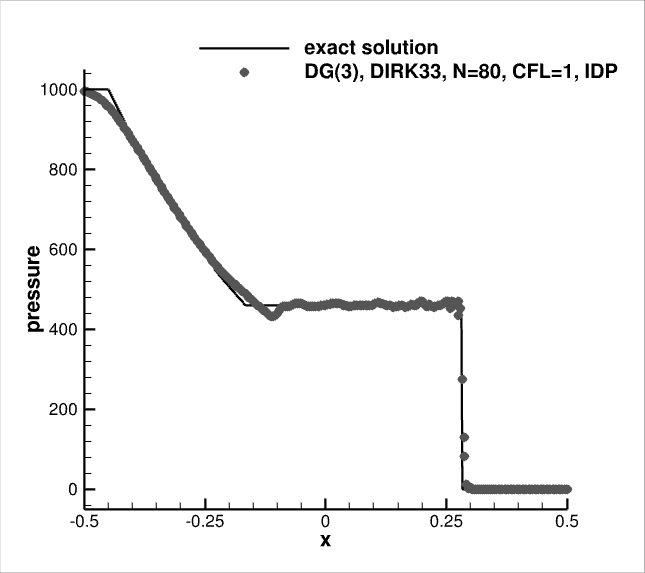}} \\
\subfloat{\begin{picture}(0,0) \put(-15,45){\rotatebox{90}{$\CFL=5$}} \end{picture}}
\subfloat[density]{\includegraphics[width=4.5cm]{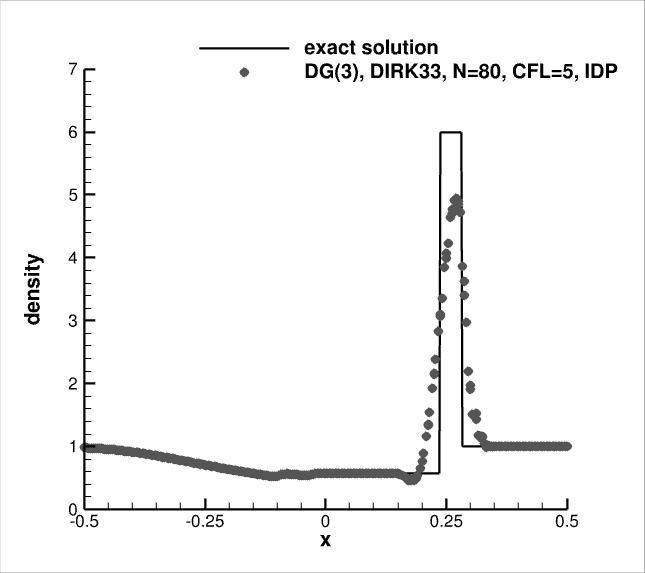}}
\subfloat[velocity]{\includegraphics[width=4.5cm]{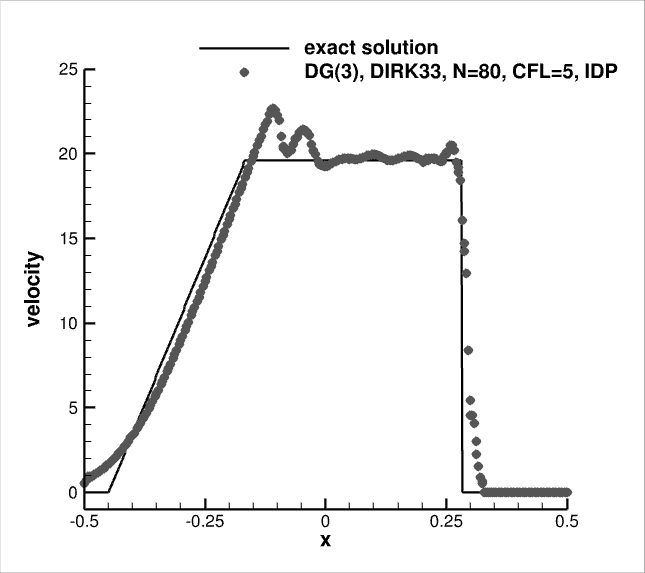}}
\subfloat[pressure]{\includegraphics[width=4.5cm]{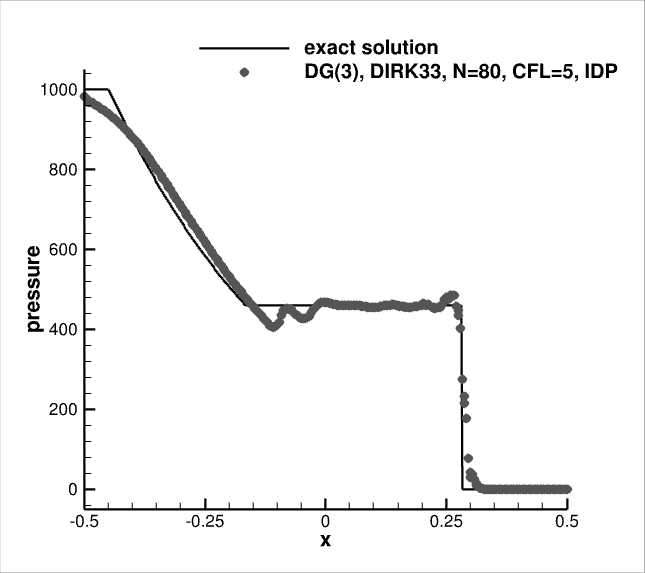}}
\caption{Riemann problem with initial data \cref{eq:RP_Toro_pb3}: solutions obtained with the limited DGSEM(3) and DIRK33 scheme at time $t=0.012$ with $\CFL=1$ (top row) and $\CFL=5$ (bottom row) on a uniform mesh with $N=80$ cells. The $p+1=4$ DOFs per mesh element are displayed (bullets) and compared to the exact solution (lines).}
\label{fig:toro_pb3_DGSEM_DIRK}
\end{figure}
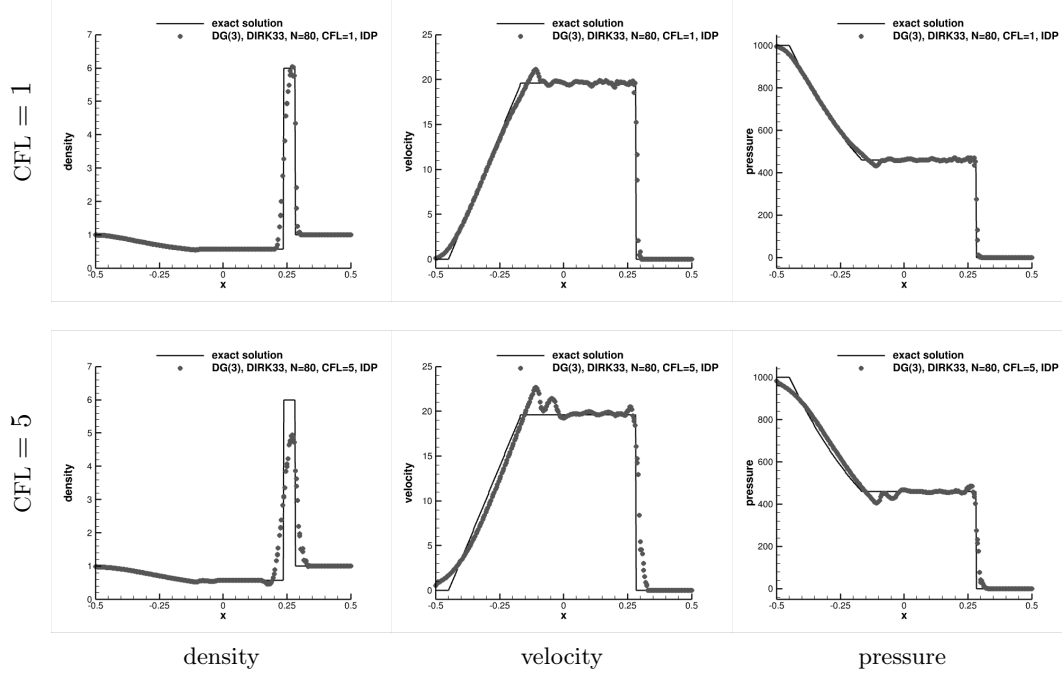 

We also consider the 2D problem defined by the diffraction of a shock wave at a corner. The computational domain is the union of $[0,1] \times [6,11]$ and $[1,13] \times [0,11]$. The initial condition is a pure right-moving shock of Mach $5.09$, initially located at $x = 0.5$ and $6 \leq y \leq 11$, moving into undisturbed air ahead of the shock with a density of $1.4$ and pressure of $1$. The boundary conditions are inflow at $x = 0$, $6 \leq y \leq 11$, outflow at $x = 13$, $0 \leq y \leq 11$, and $1 \leq x \leq 13$, $y = 0$, and reflective at the walls $0 \leq x \leq 1$, $y = 6$ and $x = 1$, $0 \leq y \leq 6$, and top boundary $0 \leq x \leq 13$, $y = 11$. We use a coarse uniform Cartesian mesh with $N=8768$ cells corresponding to a cell size $\text{diam}\,\kappa=\frac{1}{16}$ and low and large time steps corresponding to $\CFL=1$ and $\CFL=10$, respectively. In \cref{fig:shock_diffraction_pb} we plot the density contours at final time, where we compare the limited scheme obtained for both CFL values with the LO scheme at low CFL. Once again, we cannot run the simulation with the HO scheme due to nonphysical solutions, so we also provide results obtained with the HO scheme in space and an explicit time stepping with the three-stage third-order SSP Runge-Kutta method from \cite{shu-osher88}. These results are stabilized with a low enough time step so as to allow to keep positivity of the cell-averaged solution and the linear scaling limiter to obtain a positive solution \cite{zhang2010positivity}. Such solution may be considered as a reference in space and time resolution levels and we observe that the LO scheme indeed presents large diffusion, but the limited scheme at $\CFL=1$ offers similar resolution as the explicit integration, thus indicating that the limiter effectively allows to recover most of the resolution capabilities of the HO scheme, while maintaining positivity. The limited scheme at $\CFL=10$ also offers better resolution than the LO scheme at $\CFL=1$, but also exhibits some spurious oscillations.

%
\begin{figure}
\centering
\captionsetup[subfigure]{labelformat=empty}
\subfloat[ERK3 (ZS limiter \cite{zhang2010positivity}, $\CFL=\frac{1}{4}$)]{\includegraphics[width=6cm]{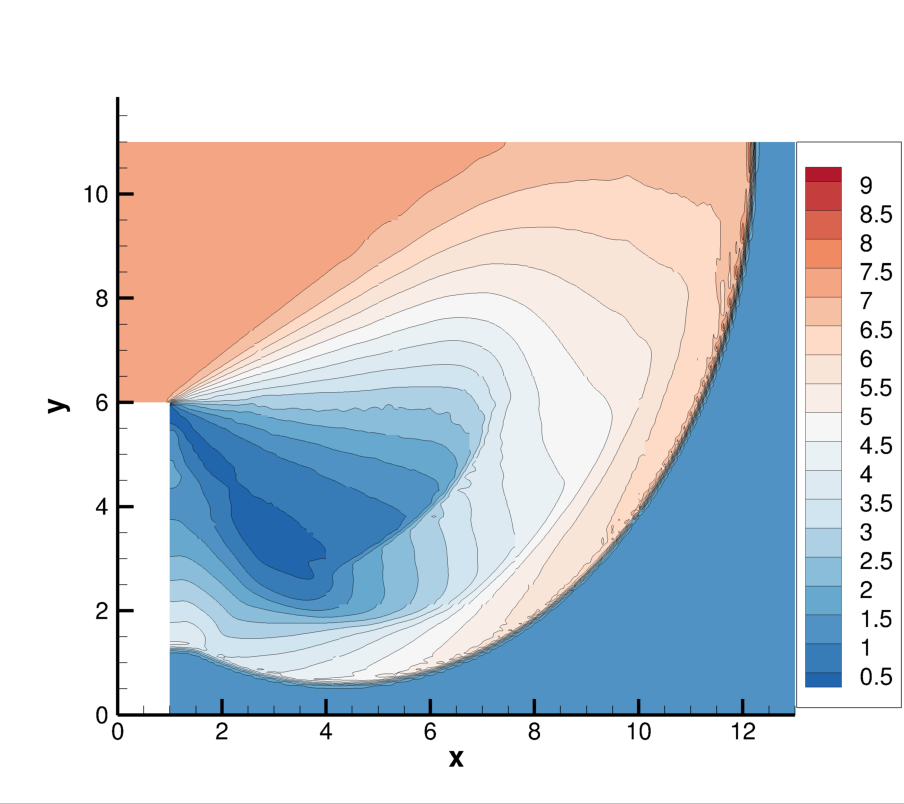}}
\subfloat[LO scheme ($\CFL=1$)]{\includegraphics[width=6cm]{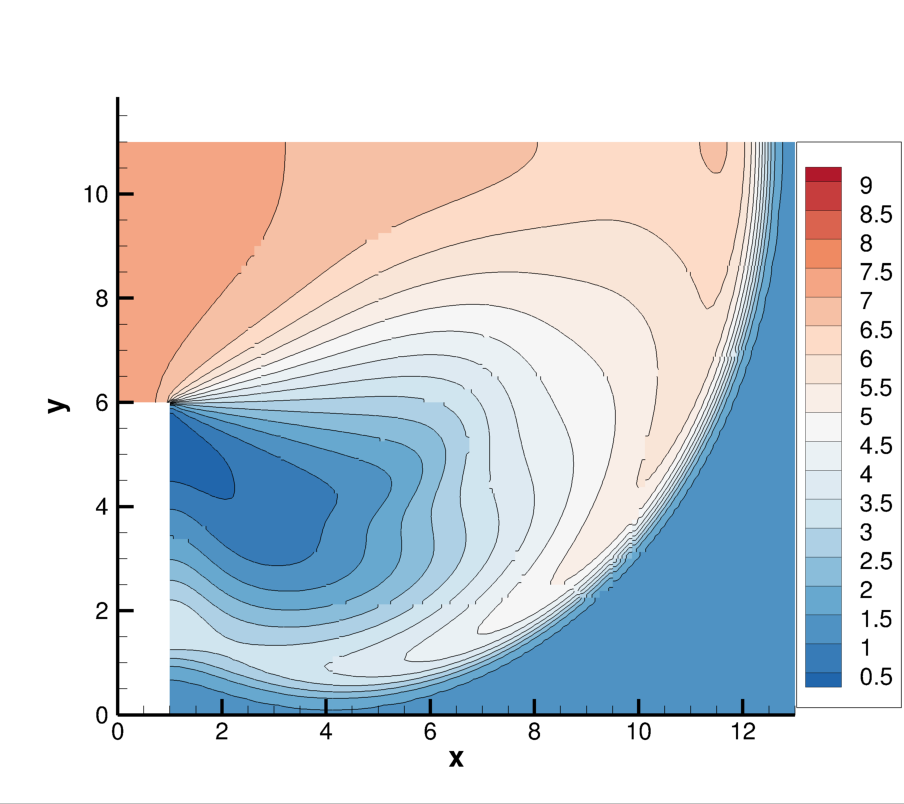}} \\
\subfloat[limited scheme ($\CFL=1$)]{\includegraphics[width=6cm]{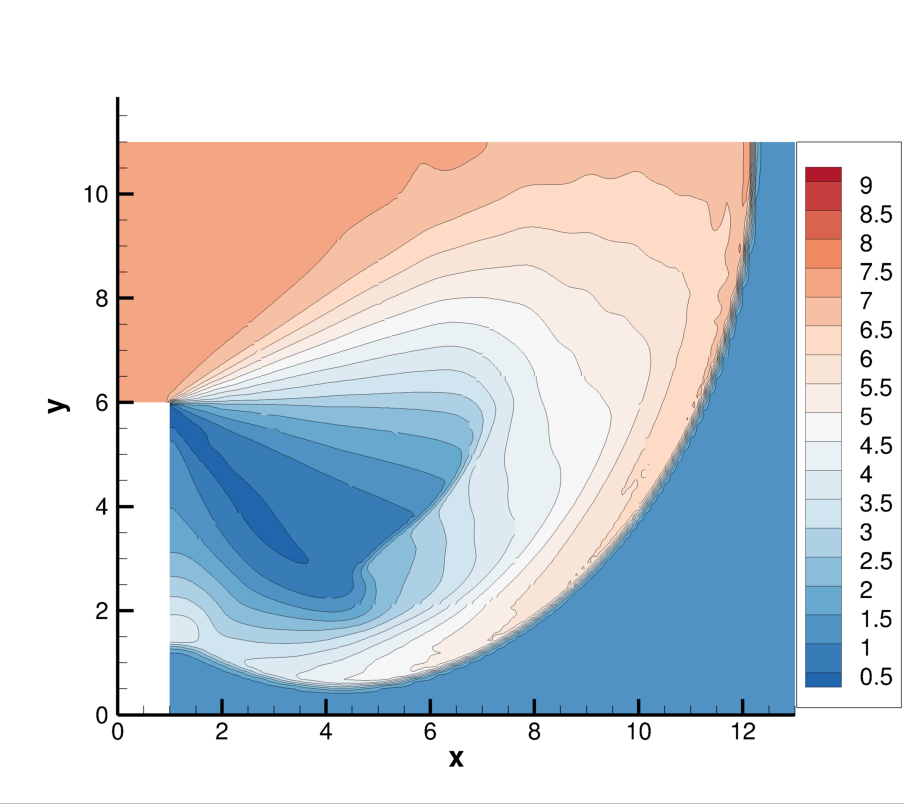}}
\subfloat[limited scheme ($\CFL=10$)]{\includegraphics[width=6cm]{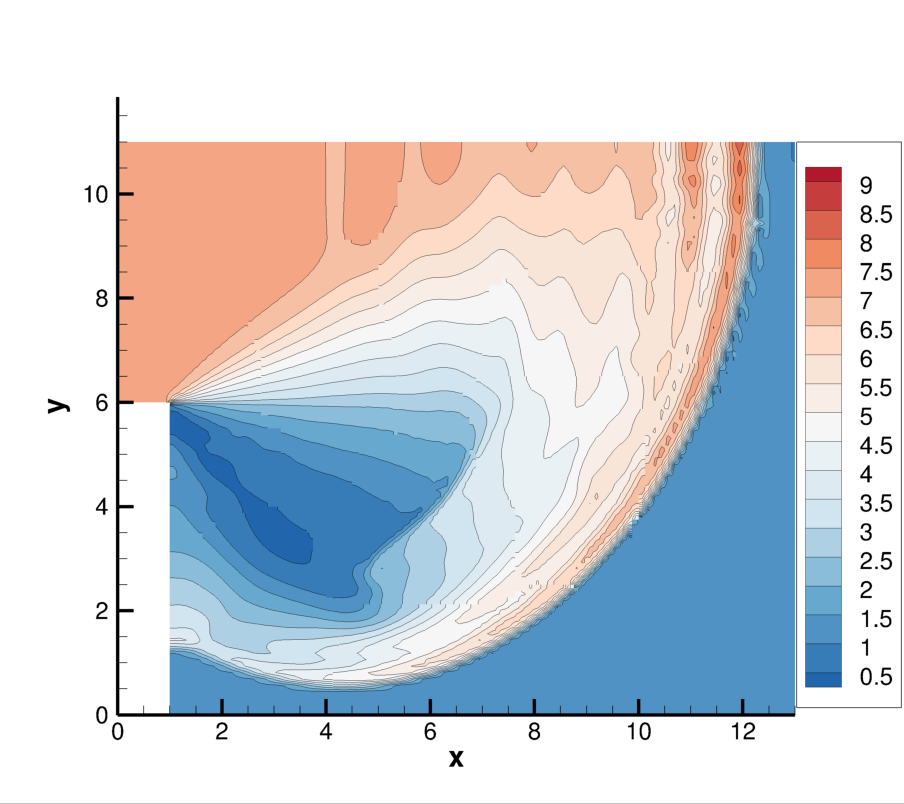}}
\caption{Shock diffraction problem: 18 evenly spaced density contours between 0.5 and 9 at time $t=2.3$. The DGSEM($3$) solution is post-processed by splitting each element into four subelements to evaluate the resolution within elements. The first computation uses an explicit time stepping (ERK3), while the other use the DIRK33 time stepping.}
\label{fig:shock_diffraction_pb}
\end{figure}

We finally consider the supersonic flow in a tunnel with a forward facing step \cite{woodward_collela_84}. The initial flow is uniform with density $\rho_0\equiv1.4$, ${\bf v}_0\equiv(3,0)^\top$ and pressure $\mathrm{p}_0\equiv1$ corresponding to a supersonic Mach $3$ flow. Inflow and outflow conditions are applied at the left and right boundaries, while we impose reflective walls at the top and bottom boundaries. The wind tunnel has a length of $3$, unit height and the step of height $0.2$ starts at $x = 0.6$. We use a coarse uniform Cartesian mesh with $4863$ cells corresponding to a cell size $\text{diam}\,\kappa=\tfrac{1}{40}$. \Cref{fig:FFS_pb} shows the density contours at final time, where we again compare the solutions of the limited scheme with two different time steps to the LO and explicit solutions. The conclusions are the same as those from the shock diffraction test case in \cref{fig:shock_diffraction_pb} and highlight robustness and good resolution capabilities of the limited scheme.

\begin{figure}
\centering
\captionsetup[subfigure]{labelformat=empty}
\subfloat[ERK3 (ZS limiter \cite{zhang2010positivity}, $\CFL=\frac{1}{4}$)]{\includegraphics[width=8.5cm]{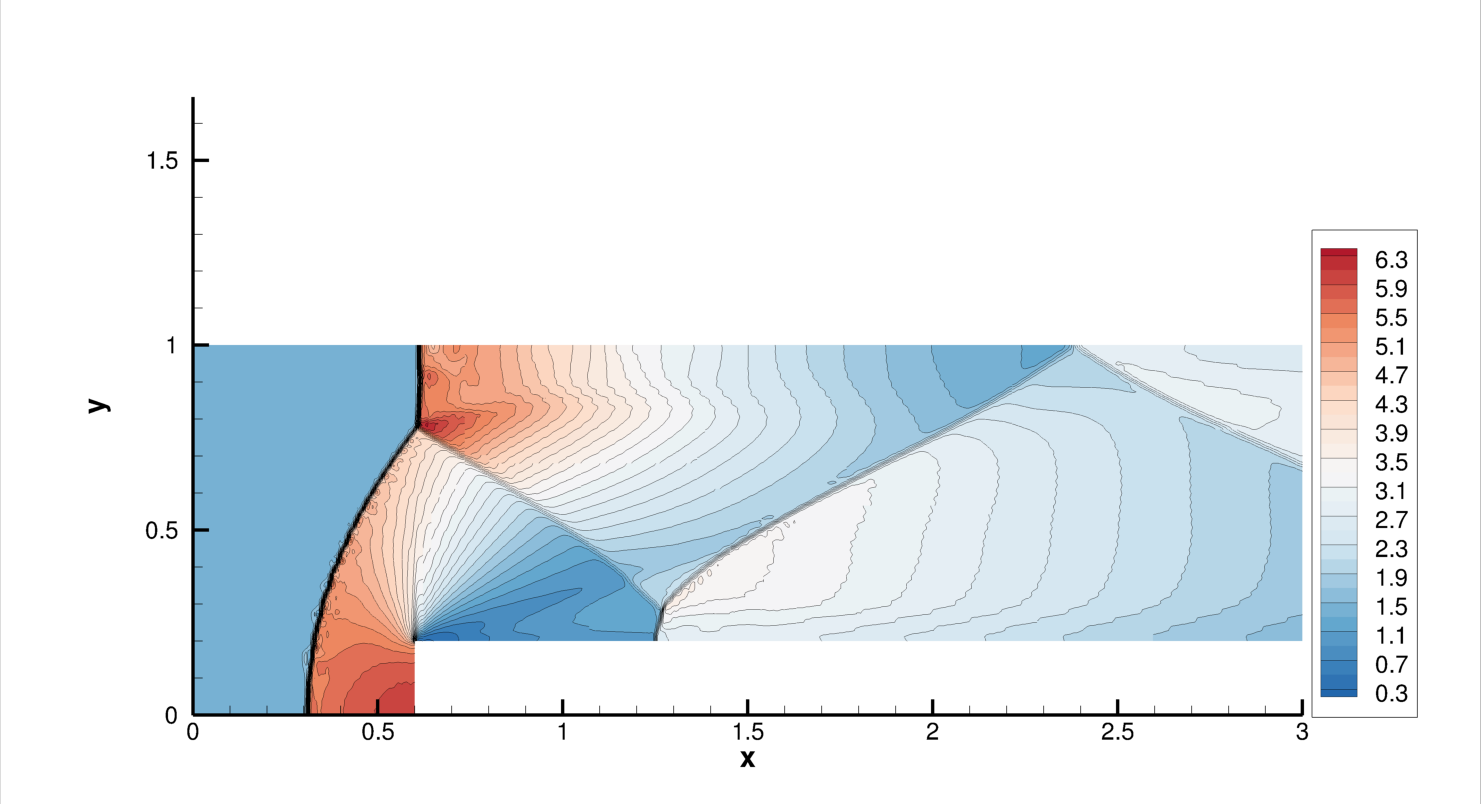}}
\subfloat[LO scheme ($\CFL=1$)]{\includegraphics[width=8.5cm]{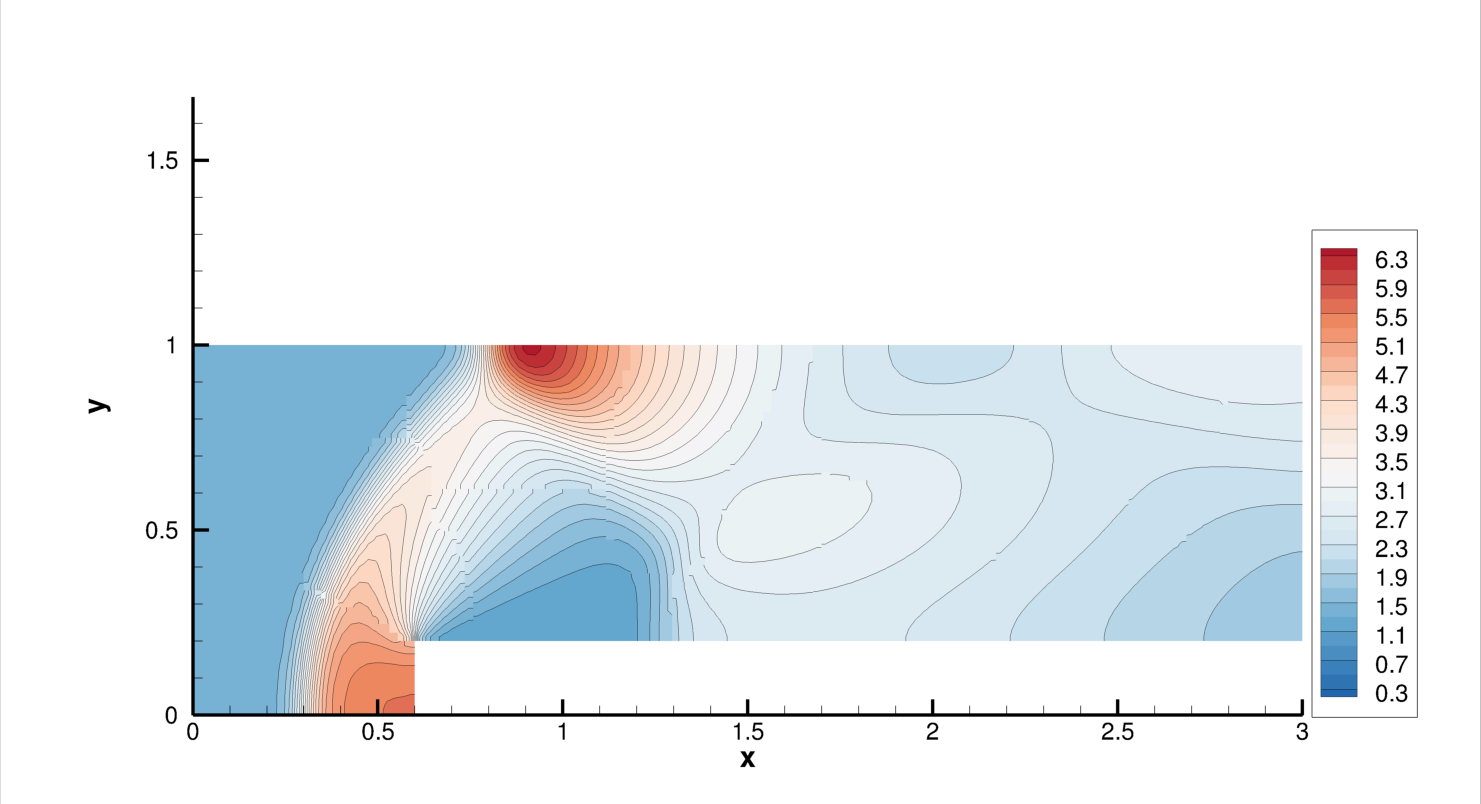}} \\
\subfloat[limited scheme ($\CFL=1$)]{\includegraphics[width=8.5cm]{FFS_5k_p3_dirk33_CFL1_IDP50_fromHO_lim2_noS_GV5-0.01_tol02.png}}
\subfloat[limited scheme ($\CFL=10$)]{\includegraphics[width=8.5cm]{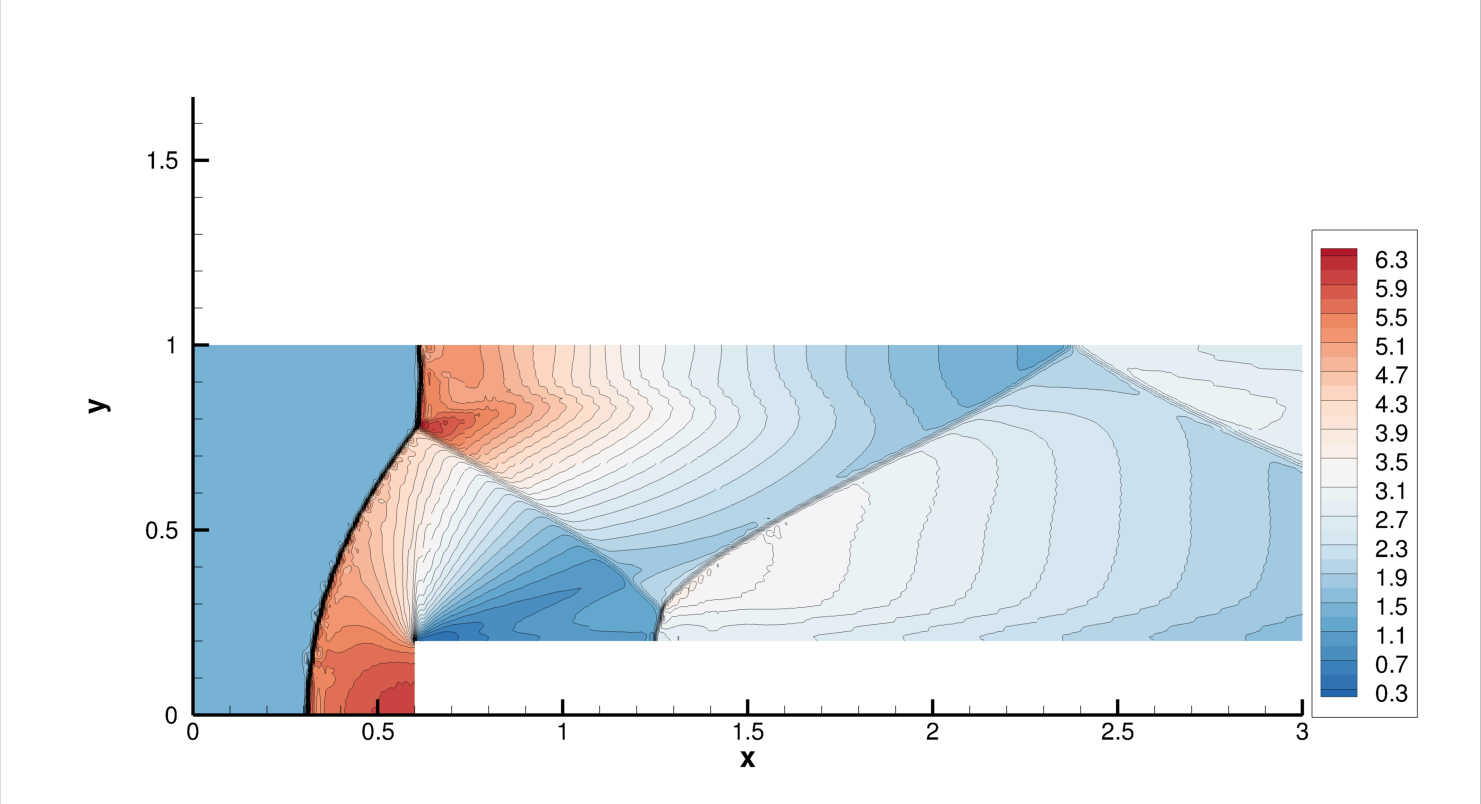}}
\caption{Forward facing step problem: 31 evenly spaced density contours between 0.3 and 6.3 at time $t=4$. The fourth-order accurate in space DGSEM($3$) solution is post-processed by splitting each element into four subelements to evaluate the resolution within elements. The first computation uses an explicit time stepping (ERK3), while the other use the DIRK33 time stepping.}
\label{fig:FFS_pb}
\end{figure}

\subsection{Space-time discontinuous Galerkin schemes}\label{sec:num-xp_stDG}

We here consider again the DGSEM in space as in \cref{sec:num-xp_DG_RK}, but coupled with a DGSEM time discretization as described in \cref{sec:time_DG_discretization} for the approximation of scalar hyperbolic problems. The LO scheme is the one proposed in \cite[Sec.~4]{renac_mpp_dgsem_nlsca_24} which consists in adding a space-time graph viscosity to the HO scheme. The only difference is the present limiter in place of the FCT limiter used in \cite{renac_mpp_dgsem_nlsca_24}. Once again, all results are obtained with the fourth-order accurate DGSEM($3,3$) in space and time as HO scheme. We use the exact Riemann solution to evaluate the numerical flux at interfaces, while we use symmetric two-point fluxes that are entropy conservative for the square entropy $\eta(u)=\tfrac{1}{2}u ^2$ within cells. The graph viscosity coefficient in \cref{eqn:fully-discr_LO_DG} is set from \cite[Th.~3.2]{renac_mpp_dgsem_nlsca_24} as  $d_\kappa=38.8L_f$ for $p=q=3$ with $L_f$ the Lipschitz constant of the flux (see \cref{tab:scalar_pbs}). The time step is computed from $\Delta t^{(n)} = \tfrac{\CFL}{L_f}\min_{\kappa\in\Omega_h}\text{diam}\,\kappa$, where the $\CFL$ values are given in \cref{tab:scalar_pbs}. The limiter is applied to impose the maximum principle at all DOFs. For nonlinear equations, we further impose a discrete entropy inequality for the Kru\v{z}kov's entropy $\eta(u)=|u|$ (see \cref{ex:scalar_eq}) and we refer to \cite[Sec.~5]{renac_mpp_dgsem_nlsca_24} for details. 

\begin{table}
 \begin{bigcenter}
     \caption{Definitions of the scalar problems and numerical parameters.}
     \begin{tabular}{llllcccccc}
        \noalign{\smallskip}\hline\noalign{\smallskip}
        problem & ${\bf f}(u)$ & $\Omega$ & $u_0(x)$ & $t$ & $p$ & $q$ & $\text{diam}\,\kappa$ & $\CFL$ & $L_f$ \\
        \noalign{\smallskip}\hline\noalign{\smallskip}
        linear & $u$ & $[0,1]$ & \cite[Ex.~1]{jiang_shu_WENO_96} & 1 & 3 & 3 & $\tfrac{1}{40}$ & $1$ & $1$ \\
        Burgers & $\tfrac{1}{2}\begin{pmatrix}u^2\\ u^2\end{pmatrix}$ & $[0,1]^2$ & $\tfrac{7}{4}1_{\|{\bf x}-(\tfrac{5}{8},\tfrac{5}{8})\|_\infty \leq \tfrac{1}{4}}-\tfrac{3}{4}$ & $\tfrac{3}{8}$ & $3$ & $3$ & $\tfrac{1}{128}$ & $10$ & $\sqrt{2}$ \\
        %
        %
        KPP & $\begin{pmatrix}\sin u\\\cos u\end{pmatrix}$ & $[-2,2]\!\times\![-\tfrac{5}{2},\tfrac{3}{2}]$ & $\left\{ \begin{array}{rl}  \tfrac{7\pi}{2} & \text{if } |{\bf x}| \leq 1, \\  \tfrac{\pi}{4} & \text{else.} \end{array}\right.$ & $1$ & $3$ & $3$ & $\tfrac{1}{25}$ & $10$ & $1$ \\ 
        \noalign{\smallskip}\hline\noalign{\smallskip} 
    \end{tabular}
    \label{tab:scalar_pbs}
 \end{bigcenter}
\end{table}

We first consider the 1D linear transport equation with nonsmooth initial data from \cite[Ex.~1]{jiang_shu_WENO_96} and periodic boundary conditions in \cref{fig:lin1D}. The LO scheme results in a very diffusive solution, while the HO scheme produces spurious oscillations that violate the maximum principle. The limiter successfully removes the oscillations and satisfies the maximum principle, while resolving the features of the solution sharply.

\begin{figure}
\centering
\captionsetup[subfigure]{labelformat=empty}
\subfloat[HO scheme]{\includegraphics[width=4.5cm]{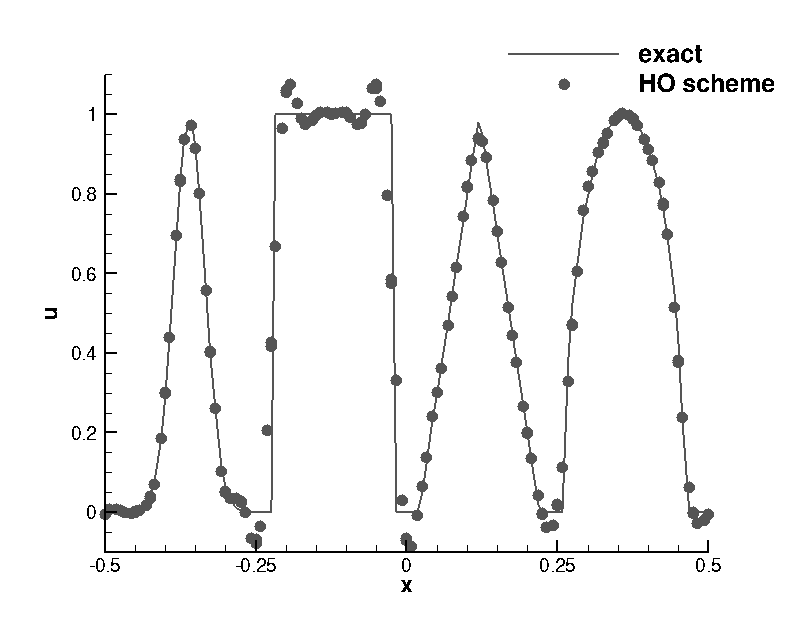}}
\subfloat[LO scheme]{\includegraphics[width=4.5cm]{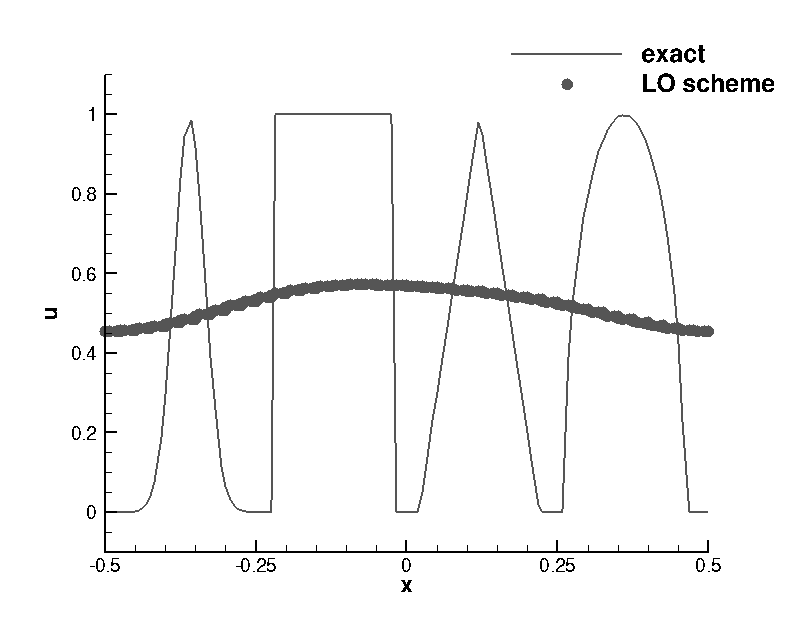}}
\subfloat[limited scheme]{\includegraphics[width=4.5cm]{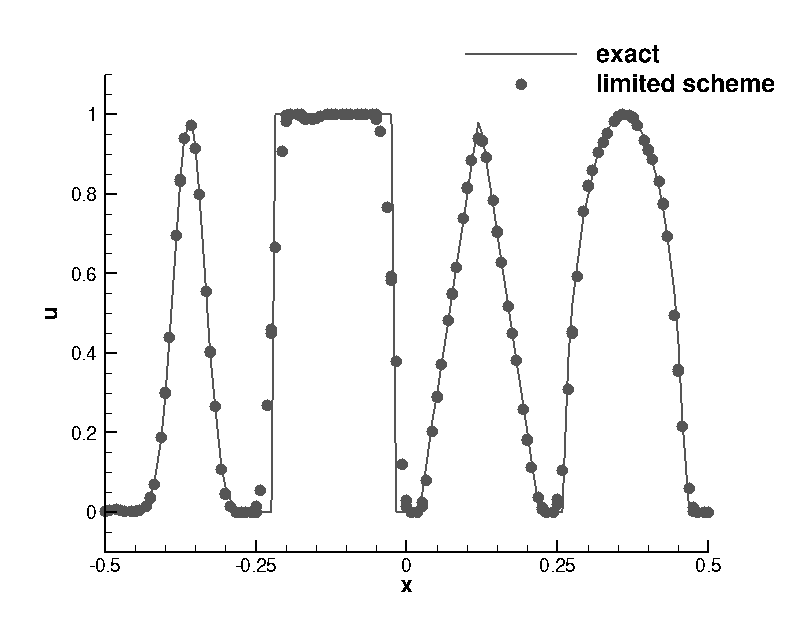}}
\caption{Linear transport with nonsmooth solution (see \cref{tab:scalar_pbs}): DGSEM($3,3$) solutions on a mesh with $N=40$ elements. The $p+1=4$ DOFs per mesh element are displayed at the final time $t=1$ (bullets) and compared to the exact solution (lines).}
\label{fig:lin1D}
\end{figure}

We now consider nonlinear problems in two space dimensions, described in \cref{tab:scalar_pbs}, that we simulate with a large CFL value on unstructured meshes with curved elements spanned by quadratic polynomials. We first consider the inviscid Burgers' equation with a discontinuous initial data \cite{LeVeque_FV_hyp,guermond_popov_IDP_CFE_scalar_17}, and then consider the KPP problem \cite{KPP_2007} defined by a nonconvex flux. The results are displayed in \cref{fig:burgers2D,fig:KPP}. where we again observe that the limiter effectively damps the spurious oscillations and maintains the maximum principle. Note that the HO scheme does not capture the physical solution for the KPP problem with a nonconvex flux, but the limited scheme successfully solves the problem.

\begin{figure}
\centering
\captionsetup[subfigure]{labelformat=empty}
\subfloat[HO scheme]{\includegraphics[width=4.5cm]{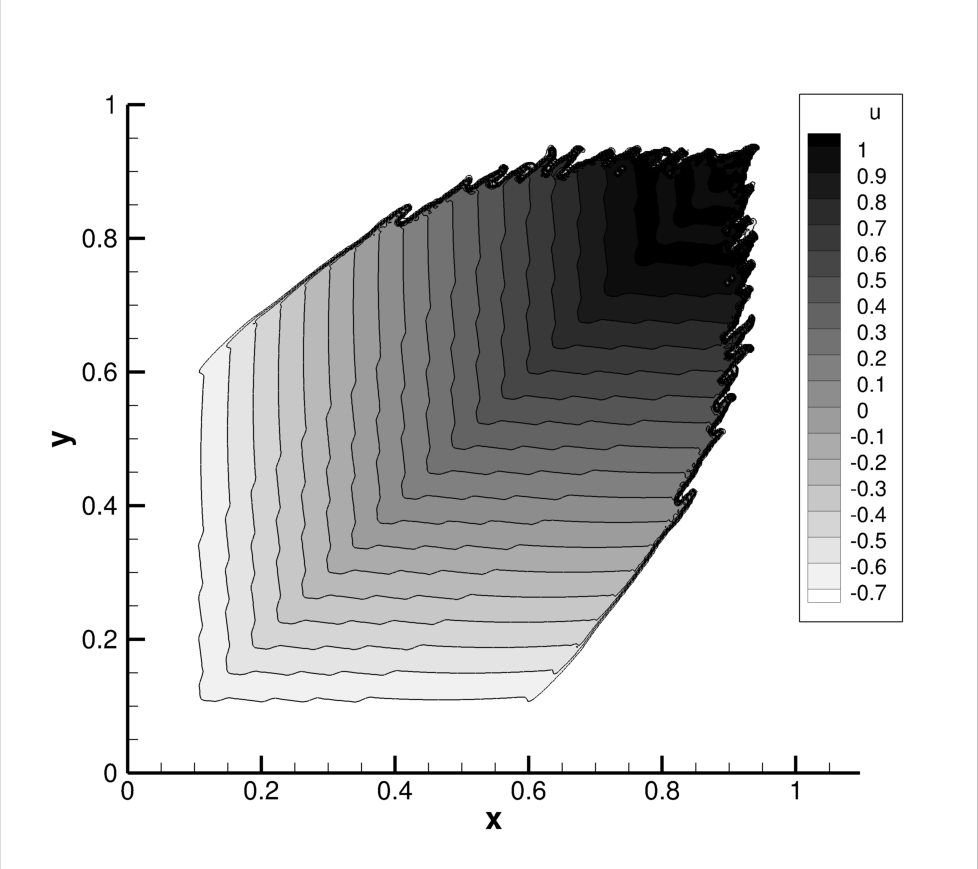}}
\subfloat[LO scheme]{\includegraphics[width=4.5cm]{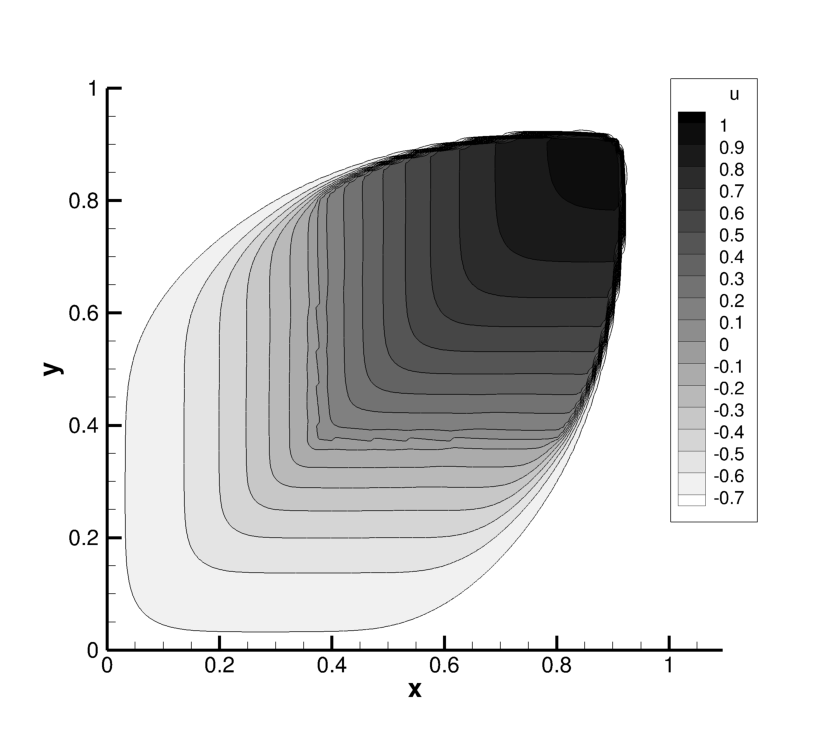}}
\subfloat[limited scheme]{\includegraphics[width=4.5cm]{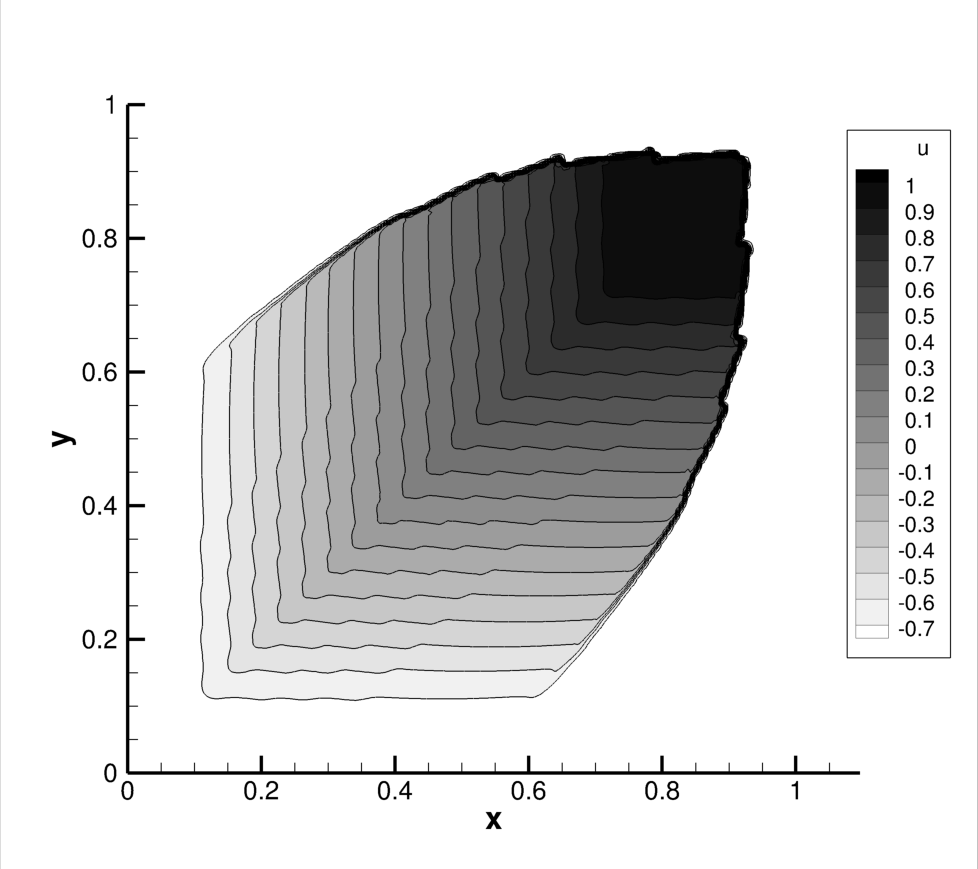}}
\caption{2D Burgers' equation (see \cref{tab:scalar_pbs}): solutions at time $t=0.375$. The DGSEM(3,3) solution is post-processed on a fine grid defined by subdividing every cell into four subcells.}
\label{fig:burgers2D}
\end{figure} 

\begin{figure}
\centering
\captionsetup[subfigure]{labelformat=empty}
\subfloat[HO scheme]{\includegraphics[width=4.5cm]{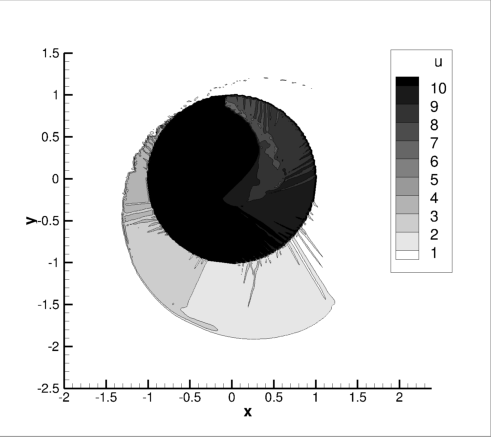}}
\subfloat[LO scheme]{\includegraphics[width=4.5cm]{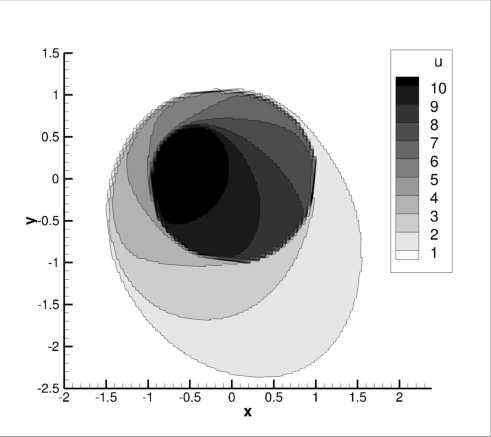}}
\subfloat[limited scheme]{\includegraphics[width=4.5cm]{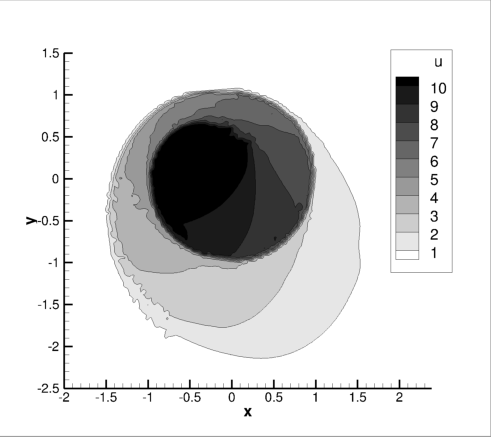}}
\caption{KPP problem (see \cref{tab:scalar_pbs}): solutions at time $t=1$. The DGSEM(3,3) solution is post-processed on a fine grid defined by subdividing every cell into four subcells.}
\label{fig:KPP}
\end{figure} 

%
%
\section{Concluding remarks}\label{sec:conclusion}

We propose and analyze an IDP limiter for HO discretizations of hyperbolic systems of conservation laws. The limiter generalizes the FCT limiter to systems of conservation laws by limiting the HO solution around a given IDP LO solution. It limits antidiffusive fluxes as in the FCT framework, but defines the limiting coefficients so as to express the limited solution as a convex combination of IDP quantities similarly to convex limiting. The limiter preserves conservation and can be applied to every conservative discretization and to a wide range of explicit and implicit time integration schemes. It can be applied iteratively to improve the accuracy of the limited solution, while preserving the invariant domains and keeping conservation. A heuristic is finally proposed to accelerate the convergence of the iterative limiter.

As an illustration, the limiter is applied to four different schemes: FV schemes with either ERK, or DIRK time stepping; DGSEM with DIRK time stepping; and to a space-time DGSEM. Implementation details are provided and we propose formulations of the antidiffusive fluxes for DG schemes that do not satisfy partition of unity. Numerical experiments on 1D and 2D scalar hyperbolic equations and the compressible Euler equations illustrate the stability and robustness of the limited solution and its ability to preserve the accuracy of the HO solution in smooth regions. Future work will consider the extension of the present limiter to linearized backward-Euler time stepping for the approximation of steady-state solutions of hyperbolic systems of conservation laws.
\bibliographystyle{elsarticle-num} 
\bibliography{biblio_generale}

\end{document}